\documentclass[a4paper,11pt,reqno]{amsart}

\usepackage{amsmath, amsfonts, amssymb, amsthm, mathrsfs, mathtools, mathalfa}
\usepackage{comment,todonotes}
\usepackage{yfonts}
\usepackage{enumerate}
\usepackage{enumitem}
\usepackage{stmaryrd}
\usepackage{fancyhdr}
\usepackage{bm}
\usepackage[utf8]{inputenc}
\usepackage[pdfencoding=auto, hyperfootnotes=false]{hyperref}
\usepackage{tikz,tikz-cd}
\usepackage[hang,flushmargin]{footmisc}
\usepackage{graphicx}
\usepackage{xr}
\usepackage{extarrows}

\makeatletter
\def\@tocline#1#2#3#4#5#6#7{\relax
  \ifnum #1>\c@tocdepth 
  \else
    \par \addpenalty\@secpenalty\addvspace{#2}%
    \begingroup \hyphenpenalty\@M
    \@ifempty{#4}{%
      \@tempdima\csname r@tocindent\number#1\endcsname\relax
    }{%
      \@tempdima#4\relax
    }%
    \parindent\z@ \leftskip#3\relax \advance\leftskip\@tempdima\relax
    \rightskip\@pnumwidth plus4em \parfillskip-\@pnumwidth
    #5\leavevmode\hskip-\@tempdima
      \ifcase #1
       \or\or \hskip 1em \or \hskip 2em \else \hskip 3em \fi%
      #6\nobreak\relax
    \dotfill\hbox to\@pnumwidth{\@tocpagenum{#7}}\par
    \nobreak
    \endgroup
  \fi}
\makeatother

\newtheorem{theorem}{Theorem}[section]
\newtheorem{lemma}[theorem]{Lemma}

\newtheorem{proposition}[theorem]{Proposition}

\theoremstyle{definition}
\newtheorem{defn}[theorem]{Definition}
\newtheorem{remark}[theorem]{Remark}

\newcommand{\mc}{\mathcal}

\newcommand{\mf}{\mathbf}
\newcommand{\mb}{\mathbb}

\newcommand{\wh}{\widehat}
\newcommand{\wt}{\widetilde}
\newcommand{\st}{\mathrm{st}}

\newcommand{\ud}{\,\mathrm{d}}
\newcommand{\id}{\mathrm{id}}

\DeclareMathOperator{\stab}{Stab}
\DeclareMathOperator{\aut}{Aut}
\DeclareMathOperator{\ab}{Z}

\DeclareMathOperator{\tran}{\Theta}

\DeclareMathOperator{\poly}{poly}

\DeclareMathOperator{\supp}{supp}

\DeclareMathOperator{\codim}{codim}

\DeclareMathOperator{\coup}{\mathsf{Cg}}
\DeclareMathOperator{\q}{c}
\DeclareMathOperator{\ns}{X}
\DeclareMathOperator{\nss}{Y}
\DeclareMathOperator{\co}{\circ\hspace{-0.02 cm}}
\DeclareMathOperator{\cu}{C}

\DeclareMathOperator{\cs}{s}
\DeclareMathOperator{\upmod}{\perp\!\!\!\perp}

\newcommand*{\sbr}[1]{\scalebox{0.8}{$(#1)$}}
\newcommand*{\db}[1]{\llbracket #1\rrbracket}

\begin{document}

\title[A generalization of the Green--Tao--Ziegler theorem]{Regularity and inverse theorems for uniformity norms on compact abelian groups and nilmanifolds}

\author{Pablo Candela}
\address{Universidad Aut\'onoma de Madrid and ICMAT\\ Ciudad Universitaria de Cantoblanco\\ Madrid 28049\\ Spain}
\email{pablo.candela@uam.es}

\author{Bal\'azs Szegedy}
\address{MTA Alfr\'ed R\'enyi Institute of Mathematics\\ 
Re\'altanoda utca 13-15\\
Budapest, Hungary, H-1053}
\email{szegedyb@gmail.com}
\subjclass[2010]{11B30, 43A85, 37A45}
\begin{abstract}
We prove a general form of the regularity theorem for uniformity norms, and deduce an inverse theorem for these norms which holds for a class of compact nilspaces including all compact abelian groups, and also nilmanifolds; in particular we thus obtain the first non-abelian versions of such theorems. We derive these results from a general structure theorem for cubic couplings, thereby unifying these results with the Host--Kra Ergodic Structure Theorem. A unification of this kind had been propounded as a conceptual prospect by Host and Kra. Our work also provides new results on nilspaces. In particular, we obtain a new stability result for nilspace morphisms. We also strengthen a result of Gutman, Manners and Varj\'u, by proving that a $k$-step compact nilspace of finite rank is a toral nilspace (in particular, a connected nilmanifold) if and only if its $k$-dimensional cube set is connected. We also prove that if a morphism from a cyclic group of prime order into a compact finite-rank nilspace is sufficiently balanced (i.e.\ equidistributed in a certain quantitative and multidimensional sense), then the nilspace is toral. As an application of this, we obtain a new proof of a refinement of the Green--Tao--Ziegler inverse theorem.
\end{abstract} 
\date{}
\vspace*{-1cm}
\maketitle
\vspace{-0.5cm}
\section{Introduction}
\noindent The inverse theorem for the Gowers norms is a major result in arithmetic combinatorics, with remarkable applications (see for instance \cite{GTarith,GTprimes}), and is central to the theory known as higher-order Fourier analysis, initiated by Gowers in his seminal paper \cite{GSz} (see also the survey \cite{GHFA}). The inverse theorem was proved in the breakthrough paper \cite{GTZ} by Green, Tao and Ziegler  in the case of finite cyclic groups (more precisely, finite intervals of integers), and analogous results were obtained for vector spaces over a finite field of fixed characteristic in \cite{BTZ,T&Z,T&Z2}. 

The Gowers norms can be defined on any compact abelian group, and these norms are special cases of more general uniformity norms, which can also be defined on nilmanifolds (see Definition \ref{def:Unorms}, or \cite[Ch.\ 12, \S 2]{HKbook}). The uniformity norms also have counterparts in other areas, especially in ergodic theory, where  seminorms of a similar kind were introduced by Host and Kra in \cite{HK}. The main result regarding these seminorms, known as the Ergodic Structure Theorem (established in \cite[Theorem 10]{HK}; see also \cite{HKbook}), is an analogue of, and was in fact an inspiration for, the inverse theorem for the Gowers norms, notably in its use of nilmanifolds. \enlargethispage{0.7cm}

An approach to higher-order Fourier analysis different from that in \cite{GTZ} was initia- ted by the second named author in \cite{Szegedy1}, inspired on one hand by the work of Host and Kra, especially their introduction of \emph{parallelepiped structures} \cite{HKparas}, and on the other hand by the non-standard analysis viewpoint in graph limit theory \cite{E&S}. This approach led to the development of the theory of \emph{nilspaces} by Antol\'in Camarena and the second named author in \cite{CamSzeg}, and initial applications of this theory to higher-order Fourier analysis were given in \cite{Szegedy:reg,Szegedy:HFA}. The theory of nilspaces has since been detailed further; see for instance the treatment in \cite{Cand:Notes1,Cand:Notes2} detailing in particular the measure-theoretic aspects, and also the development by Gutman, Manners and Varj\'u in \cite{GMV1,GMV2,GMV3} with more emphasis on topological aspects and applications in dynamics. Nilspace related topics have now  grown to generate an active research area, which has found further uses in ergodic theory \cite{CGS,GutLian}, probability theory \cite{CScouplings}, and topological dynamics \cite{GlaGutYe}.

It became conceivable that more conceptual light could be shed on higher-order Fourier analysis by unifying the nilspace approach from \cite{Szegedy:reg,Szegedy:HFA} with the ergodic theo- retic methods from \cite{HK}, a prospect raised notably by Host and Kra in \cite[end of Ch.\ 17]{HKbook}. In \cite{CScouplings}, a framework for such a unification was put forward, based on the concept of a cubic coupling, inspired especially by the cubic measures from \cite[\S 3.1]{HK}. A first application of cubic couplings was given in \cite{CScouplings} by recovering and extending the Ergodic Structure Theorem of Host and Kra in this framework. Another central application was announced in the same paper \cite{CScouplings}, namely a result extending the inverse theorem from \cite{GTZ} to compact abelian groups and also to nilmanifolds and more general nilspaces. The main purpose of this paper is to prove this result. Let us emphasize that while the combination of nilspace theory with non-standard analysis in the preprints \cite{Szegedy:reg,Szegedy:HFA} already yielded inverse theorems for uniformity norms, these were markedly less general than those presented here, and the results in the present paper follow a more conceptual approach using solely the material from the published (or to appear) papers \cite{Cand:Notes1, Cand:Notes2, CScouplings}. Crucially, it is the use of the cubic coupling framework here which enables the extension of the inverse theorem beyond abelian groups and its unification with the Ergodic Structure Theorem.

Let us set up some terminology. First we describe the class of nilspaces involved in our main results. This class consists essentially of filtered (possibly disconnected) nilmanifolds. Such a nilmanifold can always be viewed as a nilspace, by equipping it with the cube sets determined by the filtration; see \cite[Definition 1.1.2]{Cand:Notes2}. Since we shall work in the category of nilspaces, we want to capture precisely these nilmanifolds within this category, which we do with Definition \ref{def:CFRcoset} below.

Recall that $\ns$ is a \emph{compact finite-rank nilspace} (abbreviated to \emph{\textsc{cfr} nilspace}) if $\ns$ is a compact nilspace and every structure group of $\ns$ is a Lie group \cite[Definition 2.5.1]{Cand:Notes2}. (Following \cite{CamSzeg} and \cite{Cand:Notes2}, we assume compact spaces to be second-countable, unless specifically stated otherwise. \textsc{cfr} nilspaces are called \emph{Lie-fibred} nilspaces in \cite{GMV3}.)
\begin{defn}[\textsc{cfr} coset nilspaces]\label{def:CFRcoset}
We say that a $k$-step \textsc{cfr} nilspace is a \emph{coset nilspace} if it is isomorphic to a nilmanifold $G/\Gamma$ (thus $G$ is a nilpotent Lie group and $\Gamma$ is a discrete cocompact subgroup of $G$) equipped with cube sets of the form $\cu^n(G_\bullet)/ \cu^n(\Gamma_\bullet)$, $n\geq 0$, where $G_\bullet=(G_i)_{i\geq 0}$ is a filtration of degree at most $k$ of closed subgroups $G_i\lhd G$, and $\Gamma_\bullet=(\Gamma_i)_{i\geq 0}$ is a filtration on $\Gamma$ where $\Gamma_i=\Gamma\cap G_i$ is cocompact in $G_i$, $i\geq 0$.
\end{defn}
\noindent Our main results concern the class of compact nilspaces that are inverse limits of \textsc{cfr} coset nilspaces (see \cite[\S 2.7]{Cand:Notes2} for the inverse limit construction in this category). This includes all compact abelian groups, and more generally all inverse limits of nilmanifolds.

\enlargethispage{0.8cm}

We deduce the inverse theorem from a \emph{regularity theorem} for functions on  nilspaces in the above class, namely Theorem \ref{thm:reglem-intro}. Regularity results in arithmetic combinatorics are inspired by the well-known regularity lemmas from graph theory, and have hitherto focused on functions on abelian groups (see for instance \cite[Theorem 1.2]{GTarith}). The point of Theorem \ref{thm:reglem-intro} below is that a bounded measurable function on a \textsc{cfr} coset nilspace can always be decomposed into a sum of a structured function plus two errors, one error being very small in a prescribed uniformity norm, and the other being negligible in the $L^1$-norm. The structured function is a \emph{nilspace polynomial} of bounded complexity, a generalization of nilsequences that was introduced in \cite{Szegedy:reg}. To define nilspace polynomials, we first recall a general notion of complexity for \textsc{cfr} nilspaces. Recall that there are countably many \textsc{cfr} nilspaces up to isomorphism; see \cite[Theorem 3]{CamSzeg}, \cite[Theorem 2.6.1]{Cand:Notes2}.
\begin{defn}
By a \emph{complexity notion} for \textsc{cfr} nilspaces, we mean a bijection from  the countable set of isomorphism classes of \textsc{cfr} nilspaces to $\mb{N}$. Having fixed such a bijection, for $m>0$ we say that a \textsc{cfr} nilspace $\ns$ has \emph{complexity at most} $m$, and write $\textrm{Comp}(\ns)\leq m$, if its image under the bijection is at most $m$.
\end{defn}
\noindent Similarly to \cite{GTZ}, in this paper we do not pursue explicit bounds for our main results, so we do not need to be specific about the complexity notion being used. In fact our results hold for any prescribed complexity notion.
\begin{defn}[Nilspace polynomials]\label{def:nilspacepolys}
Let $\ns$ be a compact nilspace. A function $f:\ns\to\mb{C}$ is a \emph{nilspace polynomial} of degree $k$ if $f=F\co\phi$ where $\phi:\ns\to \nss$ is a continuous morphism, $\nss$ is a $k$-step \textsc{cfr} nilspace, and $F$ is continuous; $f$ has \emph{complexity } $\leq m$, denoted $\textrm{Comp}(f)\leq m$, if $F$ has Lipschitz constant $\leq m$ and $\textrm{Comp}(\nss)\leq m$.
\end{defn}
\noindent The Lipschitz constant here relates to a Riemannian metric that we fix from the start on each \textsc{cfr} nilspace, using the fact that these spaces are finite-dimensional manifolds \cite[Lemma 2.5.3]{Cand:Notes2}. Our regularity theorem ensures also that the morphism involved in the structured part satisfies a strong quantitative equidistribution property that we call \emph{balance} (following \cite{Szegedy:HFA}). This useful property has a technical definition (concerning morphisms and also nilspace polynomials), which we detail later; see Definition \ref{def:balance}. 

\begin{defn}[Uniformity seminorms on compact nilspaces]\label{def:Unorms}
For $d\geq 2$, the $U^d$-seminorm of a bounded Borel function $f:\ns\to\mb{C}$ on a compact nilspace $\ns$ is defined by $\|f\|_{U^d}=\big(\int_{\q\in\cu^d(\ns)} \prod_{v\in \{0,1\}^d} \mc{C}^{|v|}f(\q(v))\ud\mu(\q)\big)^{1/2^d}$, where $\mu$ is the Haar measure\footnote{This refers to the canonical Borel probability measure on a cube set in nilspace theory; see \cite[\S 2.2.2]{Cand:Notes2}.} on the cube set $\cu^d(\ns)$, $\mc{C}$ denotes the complex conjugation operator, and $|v|=\sum_{i=1}^d v\sbr{i}$.
\end{defn}
\noindent For a proof of the seminorm properties, and a discussion of when these quantities are norms, see Lemma \ref{lem:non-deg-app}. We can now state our main result.
\begin{theorem}[Regularity]\label{thm:reglem-intro}
Let $k\in\mb{N}$ and let $\mc{D}:\mb{R}_{>0}\times\mb{N}\to\mb{R}_{>0}$ be an arbitrary function. For every $\epsilon>0$ there exists $N=N(\epsilon,\mc{D})>0$ such that the following holds. For every compact nilspace $\ns$ that is an inverse limit of \textsc{cfr} coset nilspaces, and every Borel function $f:\ns\to \mb{C}$ with $|f|\leq 1$, there is a decomposition $f=f_s+f_e+f_r$ and number $m\leq N$ such that the following properties hold:
\begin{enumerate}[leftmargin=0.9cm]
\item $f_s$ is a $\mc{D}(\epsilon,m)$-balanced nilspace polynomial of degree $k$, $|f_s| \leq 1$, $\textup{Comp}(f_s)\leq m$, 
\item $\|f_e\|_{L^1}\leq\epsilon$,
\item $\|f_r\|_{U^{k+1}}\leq \mc{D}(\epsilon,m)$, $|f_r|\leq 1$ and $\max\{ |\langle f_r,f_s\rangle|,\,|\langle f_r,f_e\rangle|\}\leq \mc{D}(\epsilon,m)$.
\end{enumerate}
\end{theorem}
\noindent Here $\langle f,g \rangle$ denotes the inner product $\int_{\ns}f\,\overline{g}\ud\mu_{\ns}$ where $\mu_{\ns}$ is the Haar measure on $\ns$. We use the term \emph{1-bounded function} for a  function $f:\ns\to\mb{C}$ with modulus at most 1 everywhere (denoted $|f|\leq 1$). Using Theorem \ref{thm:reglem-intro}, we obtain our next main result.
\begin{theorem}[Inverse theorem]\label{thm:inverse-intro}
Let $k\in \mb{N}$ and $\delta\in (0,1]$. Then there is $m>0$ such that for every compact nilspace $\ns$ that is an inverse limit of \textsc{cfr} coset nilspaces, and every 1-bounded Borel function $f:\ns\to \mb{C}$ with $\|f\|_{U^{k+1}}\geq \delta$, there is a 1-bounded nilspace polynomial $F\co\phi$ of degree $k$ and complexity $\leq m$ such that $\langle f, F\co\phi\rangle \geq \delta^{2^{k+1}}/2$.
\end{theorem}
\noindent As detailed below, we deduce Theorem \ref{thm:reglem-intro} from  results on cubic couplings from \cite{CScouplings}. In particular, this yields \emph{directly} that the nilspace polynomial in this result is arbitrarily well balanced in relation to its complexity (this then holds also in the inverse theorem; see Theorem \ref{thm:inverse}). In the case of finite cyclic groups, a property implying the balance property, called \emph{irrationality}, can be added \emph{a posteriori} to the regularity theorem, using separate arguments; see \cite{GTarith}. Let us emphasize also that to obtain the extension beyond abelian groups in Theorem \ref{thm:inverse-intro}, our proof differs markedly from that in \cite{Szegedy:HFA}; see Section \ref{sec:ccax}, in particular Remark \ref{rem:abcase}, and Remark \ref{rem:gentocompns} on possible further extensions.

After proving Theorems \ref{thm:reglem-intro} and \ref{thm:inverse-intro}, we focus on the important case where $\ns$ consists of a cyclic group $\mb{Z}_p$ of prime order $p$, in order to show that in this case Theorem \ref{thm:inverse-intro} implies a refinement of the Green--Tao--Ziegler inverse theorem. More precisely, we obtain the following version of \cite[Conjecture 4.5]{GTZ}. This uses the notation $\poly(\mb{Z},G_\bullet)$ for the group of polynomial maps $\mb{Z}\to G$ relative to a filtration $G_\bullet$ (see \cite{Leib, GTOrb}).
\begin{theorem}\label{thm:inverseZp-intro}
Let $k\in \mb{N}$ and let $\delta\in (0,1]$. There exists a finite set $\mc{M}_{k,\delta}$ of connected filtered nilmanifolds $(G/\Gamma,G_\bullet)$, each equipped with a smooth Riemannian me- tric $d_{G/\Gamma}$, and a constant $C_{k,\delta}>0$, with the following property. For every prime $p$ and 1-bounded function $f:\mb{Z}_p\to \mb{C}$ with $\|f\|_{U^{k+1}}\geq \delta$, there exists $G/\Gamma\in \mc{M}_{k,\delta}$, a polynomial $g\in \poly(\mb{Z},G_\bullet)$ that is $p$-periodic mod $\Gamma$, and a continuous $1$-bounded function $F:G/\Gamma\to\mb{C}$ with Lipschitz constant at most $C_{k,\delta}$ relative to $d_{G/\Gamma}$, such that $|\mb{E}_{x\in \mb{Z}_p} f(x) \overline{F(g(x)\Gamma)}|\geq \delta^{2^{k+1}}/2$.
\end{theorem}
\begin{remark}
Theorem \ref{thm:inverseZp-intro} refines \cite[Theorem 1.3]{GTZ}
 in that $g$ is \emph{directly} ensured to be $p$-periodic mod $\Gamma$ (i.e.\ $g(n)^{-1}g(n+p)\in \Gamma$ for all $n\in \mb{Z}$), thus yielding a well-defined morphism $\mb{Z}_p\to G/\Gamma$. This periodicity was first established in the inverse theorem in \cite{Szegedy:reg}, and is a notable (though not exclusive) feature of the nilspace approach (periodicity is not obtained directly in \cite[Theorem 1.3]{GTZ}, but it is  obtained in the more recent proof in \cite{Man2}). Periodicity can also be included \emph{a posteriori} in \cite[Theorem 1.3]{GTZ} with additional arguments; see \cite{Man1}. Another useful refinement that our proof can add directly to Theorem \ref{thm:inverseZp-intro} is that the nilsequence is arbitrarily well balanced in relation to the complexity of $G/\Gamma$ (for the same reason mentioned above for Theorem \ref{thm:inverse}).
 \end{remark}
 \begin{remark}
Let us elaborate on how Theorem \ref{thm:inverse-intro} relates to previous non-quantitative inverse theorems such as \cite[Theorem 1.3]{GTZ} or \cite[Theorem 2]{Szegedy:HFA}. One aspect is that Theorem \ref{thm:inverse-intro} extends these results via its premise, by being applicable to functions $f$ on domains more general than compact abelian groups. Another aspect concerns how the theorem's conclusion relates to the conclusions of previous such results, and more precisely how the bounded-complexity nilspace polynomials, obtained as correlating harmonics in Theorem \ref{thm:inverse-intro}, relate to harmonics such as the nilsequences in \cite[Theorem 1.3]{GTZ}. The \textsc{cfr} nilspaces, underlying nilspace polynomials, are generalizations of nilmanifolds which still have strong structural properties akin to several of the most useful properties of nilmanifolds (such properties include an iterated-bundle structure with compact abelian Lie fibers \cite[\S 2.5]{Cand:Notes2}, \cite[\S 3.2.3]{Cand:Notes1}; a nilpotent Lie group action compatible with the cube structure \cite[\S 3.2.4 and Theorem 2.9.10]{Cand:Notes2}; and related tools in nilspace theory). Moreover, a key fact detailed in this paper is that when one restricts these nilspaces to the setting of previous results such as \cite[Theorem 1.3]{GTZ}, one recovers exactly the more explicit structure of nilmanifolds. More precisely, the crux of Theorem \ref{thm:inverseZp-intro}, compared to Theorem \ref{thm:inverse-intro}, is that in the specific $\mb{Z}_p$ setting of the former, the balanced nilspace polynomials obtained from the general setting are shown to be precisely nilsequences generated by $p$-periodic orbits on \emph{connected nilmanifolds} (these nilsequences are the same thing as nilspace polynomials from $\mb{Z}_p$ into connected \textsc{cfr} \emph{coset} nilspaces). This is established in Theorem \ref{thm:Zp-toral}.
\end{remark}
\noindent Recall that a compact nilspace is \emph{toral} if its structure groups are tori \cite[Definition 2.9.14]{Cand:Notes2} (it is then also a connected nilmanifold \cite[Theorem 2.9.17]{Cand:Notes2}). A key element in our proof of Theorem \ref{thm:Zp-toral} is the following new result about compact nilspaces.
\begin{theorem}\label{thm:k-cube-connect}
A $k$-step \textsc{cfr} nilspace is toral if and only if its $k$-cube set is connected.
\end{theorem}
\noindent A result in the direction of Theorem \ref{thm:k-cube-connect} was observed in \cite{GMV3}. Namely, \cite[Theorem 1.22]{GMV3} was noted to imply that if all the cube sets of a \textsc{cfr} nilspace are connected then the nilspace is toral. Theorem \ref{thm:k-cube-connect} strengthens this result: the connectedness of the set of $k$-cubes suffices. The proof of Theorem \ref{thm:k-cube-connect} is given in Appendix \ref{app:torality}.

\begin{remark}
Following terminology from \cite{Szegedy:HFA}, we say that a family of finite abelian groups $(\ab_i)_{i\in \mb{N}}$ is of \emph{characteristic} 0 if for every prime $p$ there are only finitely many indices $i$ such that $p$ divides the order of $\ab_i$. Our proof of Theorem \ref{thm:inverseZp-intro} can be adapted in a straightforward way to yield an analogue of this theorem where the groups $\mb{Z}_p$ are replaced by any family of characteristic 0. We omit the details in this paper.
\end{remark}
\noindent In the quantitative direction, a proof of the inverse theorem in the case of cyclic groups $\mb{Z}_p$ was given with reasonable bounds in a recent breakthrough by Manners \cite{Man2}, and in the case of vector spaces $\mb{F}_p^n$, in another recent breakthrough by Gowers and Mili\'cevi\'c \cite{GM}. As mentioned in \cite{Man2}, currently these quantitative results cannot be made to overlap. On a conceptual level, the present paper shows that the notion of nilspace polynomials (and nilspace theory more generally) offers a framework in which a more general inverse theorem can be obtained, valid in particular for any compact abelian group (namely Theorem \ref{thm:inverse-intro}), from which more specific inverse theorems such as the Green--Tao--Ziegler theorem can be fully recovered and extended.

The structure of the paper is as follows. In Section \ref{sec:ups} we recall some background on analysis in ultraproducts, and we outline its use in proving Theorem \ref{thm:reglem-intro}. In Section \ref{sec:ccax}, we analyze  ultraproducts of \textsc{cfr} coset nilspaces to locate certain factors that have a cubic coupling structure. This will enable us to apply our structure theorem from \cite{CScouplings}, as a crucial step in our proof of Theorem \ref{thm:reglem-intro}. In Section \ref{sec:ctsmorphism}, we prove a new stability result for morphisms into \textsc{cfr} nilspaces, Theorem \ref{thm:rigid}, which is central to our proof of Theorem \ref{thm:reglem-intro} and seems to be also of intrinsic interest. In Section \ref{sec:mps} we combine the above elements to prove Theorems \ref{thm:reglem-intro} and \ref{thm:inverse-intro}. In Section \ref{sec:Zp} we prove Theorem \ref{thm:inverseZp-intro}.

\smallskip

\noindent \textbf{Acknowledgements.} We thank Terence Tao for useful feedback. The first-named author received funding from Spain’s MICINN project MTM2017-83496-P. The second-named author received funding from the European Research Council under the European Union’s Seventh Framework Programme (FP7/2007-2013)/ERC grant agreement 617747. The research was supported partially by the NKFIH ``Élvonal" KKP 133921 grant and partially by the Mathematical Foundations of Artificial Intelligence project of the National Excellence Programme (grant no. 2018-1.2.1-NKP-2018-00008). We also thank the anonymous referee for valuable feedback helping to improve this paper.

\section{Ultraproducts of nilspaces, and an outline of the main proof}\label{sec:ups}
\noindent We begin by recalling some basic notions concerning ultraproducts and the Loeb measure. We do so primarily to gather the required terminology and notation. For more background on these tools we refer to standard texts such as \cite{Ross}, or more recent treatments such as \cite[\S 1.7, \S 2.10]{TaoH}. More detail on the use of these tools specifically in higher-order Fourier analysis can also be found in \cite{Warner}.

For each $i\in\mb{N}$ let $\ns_i$ be a set equipped with a $\sigma$-algebra $\mc{B}_i$ and a probability measure $\lambda_i$ on $\mc{B}_i$. We also fix from now on a non-principal ultrafilter $\omega$ on $\mb{N}$ (see \cite[\S 1.7.1]{TaoH}). We denote by $\prod_{i\to\omega} \ns_i$ the ultraproduct of the sets $\ns_i$, that is, the quotient of the cartesian product $\prod_{i\in \mb{N}} \ns_i$ under the equivalence relation $(x_i)_i\sim (y_i)_i\;\Leftrightarrow \{i\in \mb{N}:x_i=y_i\}\in \omega$. We often denote such ultraproducts using boldface, thus $\mf{X}=\prod_{i\to \omega} \ns_i$. We can equip $\mf{X}$ with a $\sigma$-algebra and a probability measure as follows. A set $B\subset \mf{X}$ is called an \emph{internal set} if $B=\prod_{i\to \omega} B_i$ for some sequence of sets $B_i\subset \ns_i$, $i\in \mb{N}$, and is an \emph{internal measurable set} if $\{i:B_i\in \mc{B}_i\}\in \omega$. For each internal measurable set $B$, we define the real number $\lambda(B)\in [0,1]$ to be the standard part of the \emph{ultralimit} (see \cite[Definition 1.7.9]{TaoH}) of the numbers $\lambda_i(B_i)$, that is $\lambda(B)=\st\big( \lim_{i\to \omega} \lambda_i(B_i)\big)$. More generally, for any compact Hausdorff  space $Y$, for every sequence of functions $f_i:\ns_i\to Y$ we can define a function $\mf{X}\to Y$, $x\mapsto \st\big(\lim_{i\to \omega} f_i(x_i)\big)$, where $(x_i)_i$ is any representative of the class $x$, the value of this function being the unique point $y\in Y$ such that\footnote{To see the existence of $y$, note that if no such $y$ existed then using compactness we could cover $\nss$ with finitely many open sets $U$ with $\{i:f_i(x_i)\in U\}\not\in \omega$, which would contradict that $\omega$ is an ultrafilter. The uniqueness follows from the Hausdorff property and a similar use of the ultrafilter's properties.} for every open set $U\ni y$ we have  $\{i:f_i(x_i)\in U\}\in \omega$. As in several texts in this area, we shorten the notation $\st\big(\lim_{i\to \omega} f_i\big)$; we denote this by $\lim_\omega f_i$. 
\begin{defn}
Given probability spaces $(\ns_i,\mc{B}_i,\lambda_i)$, $i\in \mb{N}$,  and a non-principal ultrafilter $\omega$ on $\mb{N}$, we define the corresponding \emph{Loeb measure} to be the probability measure $\lambda$ obtained by applying the Hahn--Kolmogorov extension theorem to the premeasure $\prod_{i\to \omega} B_i\mapsto \lim_\omega \lambda_i(B_i)$ defined on internal measurable subsets of $\mf{X}$ (see \cite[Theorem 2.1]{Ross}, \cite[Theorem 2.10.2]{TaoH}). The corresponding \emph{Loeb $\sigma$-algebra}, denoted by $\mc{L}_{\mf{X}}$, is the completion of the $\sigma$-algebra on $\mf{X}$ generated by the internal measurable sets.
\end{defn}
\noindent Recall that for any sequence of functions $(f_i:\ns_i\to Y)_{i\in\mb{N}}$ into a compact set $Y\subset \mb{C}$, if $f_i$ is $\mc{B}_i$-measurable for all $i$ in some set $S\in\omega$, then $\lim_\omega f_i:\mf{X}\to Y$ is $\mc{L}_{\mf{X}}$-measurable (see \cite[Theorem 5.1]{Ross}).

We now focus on ultraproducts of nilspaces. If each set $\ns_i$ is a nilspace, with cube sets $\cu^n(\ns_i)$, $n\geq 0$ (where $\cu^0(\ns_i)=\ns_i$), then it is easily checked that the ultraproduct $\mf{X}$ equipped with cube sets $\cu^n(\mf{X}):= \prod_{i\to \omega} \cu^n(\ns_i)$  satisfies the nilspace axioms as well.

Let us now outline the proof of Theorem \ref{thm:reglem-intro}, and especially our use of ultraproducts. We argue by contradiction, supposing that there is a sequence of 1-bounded Borel functions $f_i:\ns_i\to\mb{C}$ that disproves the theorem (thus for some $\epsilon>0$ and real numbers $N_i\to\infty$ as $i\to\infty$, for each $i$ the required decomposition fails for $f_i$, $\epsilon$ and $N_i$). We then consider the 1-bounded function $f=\lim_\omega f_i:\mf{X}\to\mb{C}$, and analyze this using results on cubic couplings from \cite{CScouplings}. To detail this further, we need to recall the notion of a cubic coupling. To this end we first recall the following notation from \cite{CScouplings}.

We write $\db{n}$ for the \emph{discrete $n$-cube} $\{0,1\}^n$. Two $(n-1)$-faces $F_0,F_1\subset\db{n}$ are \emph{adjacent} if $F_0\cap F_1\neq \emptyset$. For finite sets $T\subset S$ and a system of sets $(A_v)_{v\in S}$, we write $p_T$ for the coordinate projection $\prod_{v\in S} A_v\to \prod_{v\in T} A_v$. Given a probability space $\varOmega=(\Omega, \mc{A},\lambda)$, we write $\mc{A}^S$ for the product $\sigma$-algebra $\bigotimes_{v\in S} \mc{A} = \bigvee_{v\in S} p_v^{-1}(\mc{A})$ on $\Omega^S$ (where, given $\sigma$-algebras $\mc{B}_v$ on a set, $\bigvee_{v\in S} \mc{B}_v$ denotes their \emph{join}, i.e.\ the smallest $\sigma$-algebra on this set that includes $\mc{B}_v$ for all $v\in S$). We write $\mc{A}^S_T$ for the sub-$\sigma$-algebra of $\mc{A}^S$ consisting of sets depending only on coordinates indexed in $T$, i.e.\ $\mc{A}^S_T=\bigvee_{v\in T} p_v^{-1}(\mc{A})$.  We write $\mc{B}_0\wedge_\lambda \mc{B}_1$ for the \emph{meet} of $\sigma$-algebras $\mc{B}_0,\mc{B}_1\subset \mc{A}$ (see \cite[Definition 2.6]{CScouplings}), and $\mc{B}_0\upmod_\lambda \mc{B}_1$ for the relation of \emph{conditional independence}, which holds if and only if $\forall f\in L^\infty(\mc{B}_0)$, $\mb{E}(f|\mc{B}_1)\in L^\infty(\mc{B}_0)$; see \cite[Proposition 2.10]{CScouplings}. (We omit the subscript $\lambda$ from $\wedge_\lambda,\upmod_\lambda$ when the measure $\lambda$ is clear.) Inclusion and equality up to $\lambda$-null sets are denoted by $\subset_\lambda$ and $=_\lambda$ respectively \cite[\S 2.1]{CScouplings}. We write $\coup(\varOmega,S)$ for the \emph{space of self-couplings} of $\varOmega$ indexed by $S$ \cite[Definition 2.20]{CScouplings}. Finally, given $\mu\in \coup(\varOmega,S)$ and an injection $\phi:R\to S$, we write $\mu_\phi$ for the \emph{subcoupling of $\mu$ along $\phi$} \cite[Definition 2.26]{CScouplings}. Let us now recall the notion of a cubic coupling \cite[Definition 3.1]{CScouplings}. \vspace{-0.7cm}\\
\begin{defn}\label{def:cc}
A \emph{cubic coupling} on a probability space $\varOmega=(\Omega,\mc{A},\lambda)$ is a  sequence $\big(\mu^{\db{n}}\in \coup(\varOmega,\db{n})\big)_{n\geq 0}$ satisfying the following axioms for all $m,n\geq 0$:
\begin{enumerate}[leftmargin=0.7cm]
\item[1.] (Consistency)\; If $\phi:\db{m}\to\db{n}$ is an injective cube morphism then $\mu^{\db{n}}_\phi=\mu^{\db{m}}$.
\item[2.] (Ergodicity)\; The measure $\mu^{\db{1}}$ is the product measure $\lambda \times \lambda$.
\item[3.] (Conditional independence)\; For every pair of adjacent faces $F_0,F_1$ of codimension 1 in $\db{n}$, we have $\mc{A}^{\db{n}}_{F_0} \upmod_{\mu^{\db{n}}} \mc{A}^{\db{n}}_{F_1}$ and $\mc{A}^{\db{n}}_{F_0}\wedge_{\mu^{\db{n}}} \mc{A}^{\db{n}}_{F_1}=_{\mu^{\db{n}}} \mc{A}^{\db{n}}_{F_0\cap F_1}$.
\end{enumerate} \vspace{-0.3cm}
\end{defn}
\noindent Given any cubic coupling, one can define an associated family of uniformity seminorms that generalize the Gowers norms \cite[Definition 3.15]{CScouplings}. The structure theorem for cubic couplings \cite[Theorem  4.2]{CScouplings} tells us that the characteristic factor corresponding to the $k$-th order uniformity seminorm on a cubic coupling is  a $k$-step compact nilspace. Given the functions $f_i:\ns_i\to\mb{C}$ that we started with above, which were supposed not to satisfy the decomposition in Theorem \ref{thm:reglem-intro}, our goal is to apply the structure theorem to some suitable cubic coupling obtained using $\mf{X}$ and $f$, in order to obtain eventually the contradiction that some function $f_i$ does in fact satisfy the required decomposition.

To carry out the above argument, our first main task is to obtain such a cubic coupling using $\mf{X}$ and $f$. Now each compact nilspace $\ns_i$ has an associated cubic-coupling structure, given by the Haar measures $\mu_{\cu^n(\ns_i)}$ on the cube sets $\cu^n(\ns_i)$, $n\geq 0$ (see \cite[\S 2.2]{Cand:Notes2} for background on these Haar measures). More precisely, the cubic coupling in question is the sequence $(\mu_{\ns_i}^{\db{n}})_{n\geq 0}$ where $\mu_{\ns_i}^{\db{n}}$ is defined to be $\mu_{\cu^n(\ns_i)}$ viewed as a measure on $\ns_i^{\db{n}}$, i.e.\ for any set $B$ in the product $\sigma$-algebra $\mc{B}(\ns_i)^{\db{n}}$ (where $\mc{B}(\ns_i)$ is the Borel $\sigma$-algebra on $\ns_i$) we define $\mu_{\ns_i}^{\db{n}}(B):=\mu_{\cu^n(\ns_i)}\big(B\cap \cu^n(\ns_i)\big)$.  The fact that $(\mu_{\ns_i}^{\db{n}})_{n\geq 0}$ is a cubic coupling is established in \cite[Proposition 3.6]{CScouplings}. We can then apply the Loeb measure construction to the sequence of probability spaces $(\ns_i^{\db{n}},\mc{B}(\ns_i)^{\db{n}},\mu_{\ns_i}^{\db{n}})$, $i\in \mb{N}$, and thus obtain the Loeb probability space that we shall denote by $(\mf{X}^{\db{n}},\mc{L}_{\mf{X}^{\db{n}}},\mu^{\db{n}})$.  Note that the ultraproduct of cube sets $\cu^n(\mf{X}):=\prod_{i\to\omega} \cu^n(\ns_i)$ is a subset of $\mf{X}^{\db{n}}$, and that $\mu^{\db{n}}$ is concentrated on $\cu^n(\mf{X})$.

As we shall see in the next section, the cubic coupling axioms hold to some extent for these measures $\mu^{\db{n}}$. However, two problems prevent this construction from forming a genuine cubic coupling.

The first (and main) problem is that, for a sequence of measures $(\mu^{\db{n}})_{n\geq 0}$ to form a cubic coupling, the $\sigma$-algebras involved in satisfying the three axioms (especially the third axiom) must be the \emph{product} $\sigma$-algebras $\mc{A}^{\db{n}}$ (where $\mc{A}$ is the $\sigma$-algebra of the original probability space $\varOmega$). For $\Omega=\mf{X}$, this requires that the axioms be satisfied, not with the Loeb $\sigma$-algebras $\mc{L}_{\mf{X}^{\db{n}}}$ obtained above, but rather with the product $\sigma$-algebras $\mc{L}_{\mf{X}}^{\db{n}}=\bigotimes_{v\in \db{n}} \mc{L}_{\mf{X}}$. However, we then face an analogue in the present setting of a well-known fact about Loeb measure spaces, namely, we face the fact that $\mc{L}_{\mf{X}}^{\db{n}}\subset \mc{L}_{\mf{X}^{\db{n}}}$ and that this inclusion may be strict (i.e.\ with  $\mc{L}_{\mf{X}}^{\db{n}}\neq \mc{L}_{\mf{X}^{\db{n}}}$). Indeed, the inclusion $\mc{L}_{\mf{X}}^{\db{n}}\subset \mc{L}_{\mf{X}^{\db{n}}}$ can be seen using  that each measure $\mu_{\ns_i}^{\db{n}}$ is a coupling of $\mu_{\ns_i}^{\db{0}}$, and standard properties of ultralimits (e.g.\ by applying for each $v\in \db{n}$ Lemma \ref{lem:applemf} with $\pi_i$ the projection $p_v:\ns_i^{\db{n}}\to\ns_i$, to deduce that the projection $p_v:\mf{X}^{\db{n}}\to\mf{X}$  satisfies $p_v^{-1}(\mc{L}_{\mf{X}})\subset \mc{L}_{\mf{X}^{\db{n}}}$, and then concluding that $\mc{L}_{\mf{X}}^{\db{n}} = \bigvee_{v\in \db{n}}p_v^{-1}(\mc{L}_{\mf{X}})\subset \mc{L}_{\mf{X}^{\db{n}}}$). The possible strictness of this inclusion can be seen already for $n=1$, where the associated measure $\mu^{\db{1}}$ can be seen to be the product measure $\mu^{\db{0}}\times\mu^{\db{0}}$, and where we then have examples of this strict inclusion such as  \cite[Example 3.13]{Cutland} (see also \cite[Remark 2.10.4]{TaoH}). Given the above fact, we cannot ensure directly that the third axiom in Definition \ref{def:cc} is satisfied with $\mc{L}_{\mf{X}}^{\db{n}}$ as required. This problem occupies us for most of the next section, where we show that if the nilspaces $\ns_i$ are \textsc{cfr} \emph{coset} nilspaces then the cubic coupling axioms do hold with the smaller $\sigma$-algebras $\mc{L}_{\mf{X}}^{\db{n}}$, as required.

The second problem is that the Loeb measure spaces are typically not separable, thus failing to be Borel probability spaces (i.e.\ probability spaces $(\Omega,\mc{A},\lambda)$ where the measurable space $(\Omega,\mc{A})$ is standard Borel; see \cite[Definition 2.15]{CScouplings}), which is required in \cite[Theorem 4.2]{CScouplings}. This problem is addressed in the second part of the next section, using the given function $f$ to generate a suitable separable factor of $\mf{X}$ which still satisfies the axioms in Definition \ref{def:cc}.

\section{The cubic coupling axioms for ultraproducts of \textsc{cfr} coset nilspaces}\label{sec:ccax}
\noindent Recall that for each compact nilspace $\ns$ and $n\geq 0$, we write $\mu_{\ns}^{\db{n}}$ for the measure $B\mapsto \mu_{\cu^n(\ns)}\big(B\cap \cu^n(\ns)\big)$ on $\mc{B}(\ns)^{\db{n}}$, where $\mu_{ \cu^n(\ns)}$ is the Haar probability measure on the cube set $\cu^n(\ns)$. (Note that $\mu_{\ns}^{\db{0}}$ is just the Haar measure $\mu_{\ns}$ on $\ns$.) 

Our main aim in this section is to prove the following result.
\begin{proposition}\label{prop:upccaxioms}
For each $i\in \mb{N}$ let $\ns_i$ be a $k$-step \textsc{cfr} coset nilspace. For $n\geq 0$ let $\mu^{\db{n}}$ be the Loeb measure on $(\mf{X}^{\db{n}},\mc{L}_{\mf{X}^{\db{n}}})$ corresponding to the measures $\mu^{\db{n}}_{\ns_i}$. Then the measures $\mu^{\db{n}}$ restricted to the $\sigma$-algebras $\mc{L}_{\mf{X}}^{\db{n}}$ satisfy the axioms in Definition \ref{def:cc}.
\end{proposition}
\noindent The first two axioms hold in fact for all compact nilspaces.
\begin{lemma}\label{lem:12}
For each $i\in \mb{N}$ let $\ns_i$ be a $k$-step compact nilspace. For $n\geq 0$ let $\mu^{\db{n}}$ be the Loeb measure on $(\mf{X}^{\db{n}},\mc{L}_{\mf{X}^{\db{n}}})$ corresponding to the measures $\mu^{\db{n}}_{\ns_i}$. Then the measures $\mu^{\db{n}}$ restricted to the $\sigma$-algebras $\mc{L}_{\mf{X}}^{\db{n}}$ satisfy axioms 1, 2 in Definition \ref{def:cc}.
\end{lemma}
\begin{proof}
We first check the ergodicity axiom. The $\sigma$-algebra $\mc{L}_{\mf{X}}^{\db{1}}=\mc{L}_{\mf{X}}\otimes\mc{L}_{\mf{X}}$ is generated by rectangles of the form $\mf{E}_1\times \mf{E}_2$ where $\mf{E}_i\in \mc{L}_{\mf{X}}$. By part 4 of \cite[Theorem 2.1]{Ross} applied to $\mu^{\db{0}}$, there are internal measurable sets $\mf{F}_1=\prod_{i\to \omega} F_{1,i}$, $\mf{F}_2=\prod_{i\to \omega} F_{2,i}$ such that $\mu^{\db{0}}(\mf{E}_i\Delta\mf{F}_i)=0$ for $i=1,2$. Compact nilspaces are  known to satisfy the ergodicity axiom, so $\mu^{\db{1}}_{\ns_i}=\mu_{\ns_i}\times \mu_{\ns_i}$, whence $\mu^{\db{1}}(\mf{F}_1\times \mf{F}_2)=\lim_\omega \mu_{\ns_i}(F_{1,i})\mu_{\ns_i}(F_{2,i})=\mu^{\db{0}}(\mf{F}_1)\mu^{\db{0}}(\mf{F}_2)$. Note also that $\mf{E}_1\times \mf{E}_2\in \mc{L}_{\mf{X}^{\db{1}}}$ and $\mu^{\db{1}}(\mf{E}_1\times \mf{E}_2)=\mu^{\db{1}}(\mf{F}_1\times \mf{F}_2)$ (these facts are seen similarly to the inclusion $\mc{L}_{\mf{X}}^{\db{n}} \subset \mc{L}_{\mf{X}^{\db{n}}}$ in Section \ref{sec:ups}, using Lemma \ref{lem:applemf}). The ergodicity axiom follows.

To check the consistency axiom, we need to show that given any injective morphism $\phi:\db{m}\to\db{n}$, we have $\mu^{\db{n}}_{\phi}=\mu^{\db{m}}$. This holds on the larger $\sigma$-algebra $\mc{L}_{\mf{X}^{\db{m}}}$, because $\mu^{\db{n}}$ is the Loeb measure associated with the measures $\mu_{\ns_i}^{\db{n}}$ and the consistency axiom holds for $(\mu_{\ns_i}^{\db{n}})_{n\geq 0}$ (note that the measurability of the map $\mf{X}^{\db{n}}\to \mf{X}^{\db{m}}$, $\q\mapsto \q\co\phi$ with respect to $\mc{L}_{\mf{X}^{\db{n}}}$, $\mc{L}_{\mf{X}^{\db{m}}}$ is itself ensured by the fact that the measures $\mu_{\ns_i}^{\db{n}}$ obey the consistency axiom, and Lemma \ref{lem:applemf}). But then the equality $\mu^{\db{n}}_{\phi}=\mu^{\db{m}}$ holds also in the smaller $\sigma$-algebra $\mc{L}_{\mf{X}}^{\db{m}}$, since if $B\in \mc{L}_{\mf{X}}^{\db{m}}$ and $F:=\phi(\db{m})\subset \db{n}$, then $p_F^{-1}(B)$ is in $\mc{L}_{\mf{X}^{\db{n}}}$ and so $\mu^{\db{n}}\big(p_F^{-1}(B)\big)=\mu^{\db{m}}(B)$.
\end{proof}
\noindent We turn to the main task, i.e.\ to check that the conditional independence axiom holds not only with the $\sigma$-algebras $\mc{L}_{\mf{X}^{\db{n}}}$, but also with the smaller ones  $\mc{L}_{\mf{X}}^{\db{n}}$. 
As recalled in Section \ref{sec:ups}, for $F\subset \db{n}$ we denote by $(\mc{L}_{\mf{X}})^{\db{n}}_F$ the $\sigma$-algebra $\bigvee_{v\in F} p_v^{-1}(\mc{L}_{\mf{X}})\subset \mc{L}_{\mf{X}}^{\db{n}}$.

\begin{remark}\label{rem:abcase}
In the special case of Proposition \ref{prop:upccaxioms} where each $\ns_i$ is a compact abelian group (equipped with its standard cubes; see \cite[Proposition 2.1.2]{Cand:Notes1}), the ultraproduct $\mf{X}$ is also an abelian group. This can be used to prove the conditional independence axiom with an argument that is markedly simpler than the one we use below for the more general case. Indeed, in the abelian case, the group structure on $\mf{X}$ yields a useful expression for the conditional expectation $\mb{E}\big(f| (\mc{L}_{\mf{X}})^{\db{n}}_{F_i}\big)$, namely that this is almost-surely equal to the function $\mf{x}\mapsto \int_{\mf{X}} f(\mf{x}+ t^{F_i})\ud \lambda(t)$, where $t^{F_i}$ is the element of the group $\mf{X}^{\db{n}}$ with $t^{F_i}(v)=t$ if $v\in F_i$ and $t^{F_i}(v)=0$ otherwise. These integral expressions for these expectation operators make it easy to see that for the two faces $F_0,F_1$ the operators commute. This implies the conditional independence axiom (via \cite[Proposition 2.10]{CScouplings}, say). While this case is much simpler than the argument in the general case, it still has significant content, and looking at its details can be helpful to understand the rest of this section.
\end{remark}
\noindent Let us introduce a simplified notation for $\sigma$-algebras for the rest of this section. For $S\subset \db{n}$, when the ultraproduct nilspace $\mf{X}$ and the dimension $n$ are clear from the context, we write simply $\mc{A}$ for $(\mc{L}_{\mf{X}})^{\db{n}}$, and $\mc{A}_S$ for $(\mc{L}_{\mf{X}})^{\db{n}}_S$. Similarly, we write $\mc{B}$ for $\mc{L}_{\mf{X}^{\db{n}}}$ and $\mc{B}_S$ for the $\sigma$-algebra $p_S^{-1}(\mc{L}_{\mf{X}^S})$ on $\mf{X}^{\db{n}}$. By the explanation at the end of Section \ref{sec:ups} we see that $\mc{A}_S\subset \mc{B}_S$ (and this inclusion may be strict).

Our main task, then, is to prove that for any adjacent faces $F_0,F_1\subset\db{n}$ of codimension 1, we have $\mc{A}_{F_0}\upmod_{\mu^{\db{n}}} \mc{A}_{F_1}$ and $\mc{A}_{F_0}\wedge_{\mu^{\db{n}}} \mc{A}_{F_1} = \mc{A}_{F_0\cap F_1}$.

We say that two faces of codimension 1 in $\db{n}$ are \emph{opposite faces} if they are not adjacent (i.e.\ if their intersection is empty). Given a $\sigma$-algebra $\mc{X}$ on a set $X$, and a finite set $S$, we say an $\mc{X}^S$-measurable function $f:X^S\to \mb{C}$ is a \emph{rank 1 function} if $f=\prod_{v\in S} f_v\co p_v$ where each $f_v:X\to \mb{C}$ is $\mc{X}$-measurable. 

We begin by reducing our main task as follows. 
\begin{lemma}\label{lem:reduc1}
The conditional independence axiom holds with $\mc{A}$, $\mu^{\db{n}}$ \textup{(}$\forall n\in \mb{N}$\textup{)} if the following statement holds: $\forall n\in\mb{N}$, for any opposite faces $F_0,F_1\subset \db{n}$ of codimension 1, every rank 1 bounded $\mc{A}_{F_0}$-measurable  function $f$ satisfies $\mb{E}(f|\mc{B}_{F_1})\in L^\infty(\mc{A}_{F_1})$.
\end{lemma}
\noindent Here and below, in notions involving equality up to null sets, unless otherwise stated these are null sets relative to $\mu^{\db{n}}$ and are allowed to be from the largest ambient $\sigma$-algebra on $\mf{X}^n$, i.e.\ $\mc{L}_{\mf{X}^n}$. Thus ``$\mb{E}(f|\mc{B}_{F_1})\in L^\infty(\mc{A}_{F_1})$" here means that $\mb{E}(f|\mc{B}_{F_1})$ agrees with some $\mc{A}_{F_1}$-measurable bounded function outside some $\mu^{\db{n}}$-null set (recall that  $\mb{E}(f|\mc{B}_{F_1})$ is defined up to $\mu^{\db{n}}$-null sets anyway). Similarly, equalities between conditional expectations are meant up to a null-set in the ambient measure (if there is danger of confusion, we indicate the measure by a subscript in the equality).
\begin{proof}
To confirm that the conditional independence axiom holds, we have to show that for any adjacent faces $F_0',F_1'\subset \db{n}$ of codimension 1 we have $\mc{A}_{F_0'}\upmod_{\mu^{\db{n}}} \mc{A}_{F_1'}$ and $\mc{A}_{F_0'}\wedge_{\mu^{\db{n}}} \mc{A}_{F_1'} =_{\mu^{\db{n}}} \mc{A}_{F_0'\cap F_1'}$. By \cite[Lemma 2.30]{CScouplings}, it suffices to prove that if $f$ is a rank 1 bounded $\mc{A}_{F_0'}$-measurable function then $\mb{E}(f|\mc{A}_{F_1'})\in L^\infty(\mc{A}_{F_0'\cap F_1'})$. We have $\mb{E}(f|\mc{A}_{F_1'})=\mb{E}(\mb{E}(f|\mc{B}_{F_1'})|\mc{A}_{F_1'})$, since $\mc{A}_{F_1'}\subset \mc{B}_{F_1'}$. We also have $\mb{E}(f|\mc{B}_{F_1'})=\mb{E}(f|\mc{B}_{F_0'\cap F_1'})$ because the conditional independence axiom holds for the measures $\mu_{\ns_i}^{\db{n}}$, and this is then seen to imply the same property for $\mu^{\db{n}}$ on $\mc{B}$ using Lemma \ref{lem:upcialgs}. Hence $\mb{E}(f|\mc{A}_{F_1'})=\mb{E}(  \mb{E}(f|\mc{B}_{F_0'\cap F_1'})  |\mc{A}_{F_1'})$. Therefore, if we prove
\begin{equation}\label{eq:reduc1}
\mb{E}(f|\mc{B}_{F_0'\cap F_1'}) \in L^\infty(\mc{A}_{F_0'\cap F_1'}),
\end{equation}
then $\mb{E}(f|\mc{B}_{F_0'\cap F_1'})=\mb{E}(f|\mc{A}_{F_0'\cap F_1'})$ (since $\mc{B}_{F_0'\cap F_1'}\supset \mc{A}_{F_0'\cap F_1'}$), which implies that $\mb{E}(f|\mc{A}_{F_1'})$ $=\mb{E}(  \mb{E}(f|\mc{A}_{F_0'\cap F_1'})  |\mc{A}_{F_1'})=\mb{E}(f|\mc{A}_{F_0'\cap F_1'})$, so $\mb{E}(f|\mc{A}_{F_1'})\in L^\infty(\mc{A}_{F_0'\cap F_1'})$ as required.

Since $f$ is a rank 1 function $\prod_{v\in F_0'} f_v\co p_v$, and $\prod_{v\in F_0'\cap F_1'} f_v\co p_v$ is $\mc{A}_{F_0'\cap F_1'}$-measurable, we have $\mb{E}(f|\mc{B}_{F_0'\cap F_1'}) = (\prod_{v\in F_0'\cap F_1'} f_v\co p_v)\, \mb{E}( \prod_{v\in F_0'\setminus F_1'} f_v\co p_v  |\mc{B}_{F_0'\cap F_1'})$. Hence, if it holds that $\mb{E}( \prod_{v\in F_0'\setminus F_1'} f_v\co p_v  |\mc{B}_{F_0'\cap F_1'})\in L^\infty(\mc{A}_{F_0'\cap F_1'})$ then \eqref{eq:reduc1} follows. But this is indeed seen to hold  by relabeling $F_0'$ as $\db{n}$, $F_0'\setminus F_1'$ as $F_0$, and $F_0'\cap F_1'$ as $F_1$, and using the statement in the lemma.
\end{proof}
\noindent To prove the statement in Lemma \ref{lem:reduc1}, we work with the $\sigma$-algebra $\mc{I}:=\mc{B}_{F_0}\wedge_{\mu^{\db{n}}} \mc{B}_{F_1}\subset \mc{L}_{\mf{X}^{\db{n}}}$. First we note the following expression for $\mc{I}$ in terms of a $\sigma$-algebra $\mc{I}'\subset \mc{L}_{\mf{X}^{\db{n-1}}}$.
\begin{lemma}\label{lem:I-expr}
Let $F_0,F_1$ be opposite faces of codimension 1 in $\db{n}$. Let $\mf{\mc{I}}'$ be the $\sigma$-algebra of sets $A'\in \mc{L}_{\mf{X}^{\db{n-1}}}$ such that $p_{F_0}^{-1}(A') =_{\mu^{\db{n}}} p_{F_1}^{-1}(A')$. Then we have  $p_{F_0}^{-1}(\mc{I}') =_{\mu^{\db{n}}}p_{F_1}^{-1}(\mc{I}') =_{\mu^{\db{n}}} \mc{I}$.
\end{lemma}
\begin{proof}
It is clear from the definitions that $p_{F_0}^{-1}(\mc{I}') =_{\mu^{\db{n}}}p_{F_1}^{-1}(\mc{I}')\subset_{\mu^{\db{n}}} \mc{I}$, so it suffices to prove that $\mc{I} \subset_{\mu^{\db{n}}}  p_{F_0}^{-1}(\mc{I}')$. The idea is that the analogous inclusion is known to hold for the nilspaces $\ns_i$, and the inclusion for $\mc{I}$ then follows by straightforward arguments with ultraproducts. More precisely, let $\mc{B}_i$ denote the Borel $\sigma$-algebra on $\ns_i$ for each $i\in \mb{N}$, and recall that the cubic Haar measures $\mu_{\ns_i}^{\db{m}}$, $m\geq 0$ form a cubic coupling \cite[Proposition 3.6]{CScouplings}, so by \cite[Lemma 3.4]{CScouplings} the measure $\mu_{\ns_i}^{\db{n}}$ is an idempotent coupling, and so by \cite[Lemma 2.62 (iii) and Proposition 2.66]{CScouplings} we have $(\mc{B}_i)_{F_0}^{\db{n}} \upmod_{\mu_{\ns_i}^{\db{n}}} (\mc{B}_i)_{F_1}^{\db{n}}$, for each $i\in \mb{N}$. By Lemma \ref{lem:upcialgs}, for every $A\in \mc{I}$ there are sets $A_i\in (\mc{B}_i)_{F_0}^{\db{n}}\wedge_{\mu_{\ns_i}^{\db{n}}} (\mc{B}_i)^{\db{n}}_{F_1}$, $i\in \mb{N}$, such that $A=_{\mu^{\db{n}}}\prod_{i\to \omega} A_i$. Then by \cite[Lemma 2.62 (iii)]{CScouplings}, there is $A_i'\in \mc{B}_i^{\db{n-1}}$ such that $p_{F_0}^{-1}(A_i')=_{\mu_i^{\db{n}}}A_i=_{\mu_i^{\db{n}}}p_{F_1}^{-1}(A_i')$. Now $A':=\prod_{i\to\omega} A_i'$ is in $\mc{I}'$ and  $A=_{\mu_i^{\db{n}}} p_{F_0}^{-1}(A')$. The desired inclusion follows.
\end{proof}
\noindent Using this expression of $\mc{I}$, we now perform a second reduction, using Lemma \ref{lem:reduc1}. 
\begin{lemma}\label{lem:reduc2}
The conditional independence axiom holds with $(\mc{A},\mu^{\db{n}})$ if the following statement holds. For every pair of opposite faces $F_0,F_1$ of codimension 1 in $\db{n}$, the $\sigma$-algebra $\mc{I}=\mc{B}_{F_0}\wedge_{\mu^{\db{n}}} \mc{B}_{F_1}$ satisfies $\mc{A}_{F_0}\upmod_{\mu^{\db{n}}} \mc{I}$.
\end{lemma}

\begin{proof}
By Lemma \ref{lem:reduc1}, it suffices to prove that for every rank 1 bounded $\mc{A}_{F_0}$-measurable function $f$ we have $\mb{E}(f|\mc{B}_{F_1})\in L^\infty(\mc{A}_{F_1})$. We claim that $\mc{B}_{F_0} \upmod \mc{B}_{F_1}$. As in the proof of Lemma \ref{lem:I-expr}, this follows from a similar property holding for the nilspaces $\ns_i$. Indeed, as recalled in that proof, for each $i$ the coupling $\mu_{\ns_i}^{\db{n}}$ is idempotent. By \cite[Lemma 2.62 (iii) and Proposition 2.66]{CScouplings} the claimed conditional independence holds for the analogues of $\mc{B}_{F_0},\mc{B}_{F_1}$ on $\ns_i^{\db{n}}$. Our claim then follows by Lemma \ref{lem:upcialgs}. Now, since $f$ is $\mc{B}_{F_0}$-measurable (as $\mc{B}_{F_0}\supset \mc{A}_{F_0}$), by $\mc{B}_{F_0} \upmod \mc{B}_{F_1}$ we have $\mb{E}(f|\mc{B}_{F_1})=\mb{E}(f|\mc{B}_{F_0}\wedge\mc{B}_{F_1})=\mb{E}(f|\mc{I})$. Hence, it suffices to prove that $\mb{E}(f|\mc{I})\in L^\infty(\mc{A}_{F_1})$.

We now claim that $\mc{I}\wedge \mc{A}_{F_0} =_{\mu^{\db{n}}} \mc{I}\wedge \mc{A}_{F_1}$. Confirming this claim would complete the proof. Indeed, by assumption $\mc{A}_{F_0}\upmod \mc{I}$, so we would have $\mb{E}(f|\mc{I})\in L^\infty(\mc{A}_{F_0}\wedge \mc{I})= L^\infty(\mc{A}_{F_1}\wedge \mc{I})\subset L^\infty(\mc{A}_{F_1})$, as required. To prove the claim, let $\sigma$ be the reflection map on $\mf{X}^{\db{n}}$ induced by the reflection on $\db{n}$ that permutes $F_0$ and $F_1$. By Lemma \ref{lem:I-expr}, for every $U\in \mc{I}$ we have $\sigma(U)=_{\mu^{\db{n}}} U$. Since $\sigma(\mc{A}_{F_0})=\mc{A}_{F_1}$, if follows that for every $U\in \mc{I}\wedge \mc{A}_{F_0}$ we have $U=_{\mu^{\db{n}}}\sigma(U)\in \sigma(\mc{A}_{F_0})=\mc{A}_{F_1}$, so $\mc{I}\wedge \mc{A}_{F_0}\subset_{\mu^{\db{n}}} \mc{I}\wedge \mc{A}_{F_1}$. Similarly $\mc{I}\wedge \mc{A}_{F_1}\subset_{\mu^{\db{n}}} \mc{I}\wedge \mc{A}_{F_0}$.
\end{proof}
\noindent To prove the statement in Lemma \ref{lem:reduc2}, we now work towards a useful description of $\mc{I}$ in terms of an invariance under a certain group action. For this, we start using the coset nilspace structure. Thus, we now suppose that $\mf{X}$ is an ultraproduct of \textsc{cfr} coset nilspaces $\ns_i=(G^{(i)}/\Gamma^{(i)},G^{(i)}_\bullet)$, $i\in \mb{N}$. Note that $\mf{X}$ is then a coset nilspace $(G/\Gamma,G_\bullet)$ (in the algebraic sense of \cite[Proposition 2.3.1]{Cand:Notes1}), where $G$, $\Gamma$ are the groups $\prod_{i\to\omega} G^{(i)}$, $\prod_{i\to\omega} \Gamma^{(i)}$ respectively, and $G_\bullet=(G_j)_{j\geq 0}$ is a filtration with $G_j= \prod_{i\to\omega} G^{(i)}_j$.

Given a filtration $G_\bullet$ and $\ell\in \mb{N}$, we denote by $G_\bullet^{+ \ell}$ the \emph{shifted filtration} whose $j$-th term is $G_{j+\ell}$ (strictly speaking, this is a \emph{prefiltration}; see \cite[Apppendix C]{CS}). We use the notion of a $1$-\emph{arrow} of cubes on a nilspace $\ns$ \cite[Definition 2.2.18]{Cand:Notes1}: for $\q_0,\q_1\in \cu^n(\ns)$, the $1$-arrow $\langle \q_0,\q_1\rangle_1\in \ns^{\db{n+1}}$ is defined by $\langle \q_0,\q_1\rangle_1(v,j)=\q_j(v)$, $j=0,1$. 

Given any nilspace $\ns$, we define an equivalence relation $\sim$ on $\cu^{n-1}(\ns)$ by declaring that $\q_0\sim\q_1$ if $\langle\q_0,\q_1\rangle_1\in \cu^n(\ns)$. The following result gives a useful algebraic description of this relation when $\ns$ is a coset nilspace $(G/\Gamma,G_\bullet)$ (the purely algebraic definition of a coset nilspace can be recalled from \cite[Proposition 2.3.1]{Cand:Notes1}).

\begin{lemma}\label{lem:simchar}
Let $\ns=(G/\Gamma,G_\bullet)$ be a coset nilspace. Then $\q_0\sim\q_1$ if and only if there exist $\wt{\q}_0, \wt{\q}_1\in \cu^{n-1}(G_\bullet)$ with $\q_i=\pi_\Gamma\co \wt{\q}_i$, $i=0,1$, and $\wt{\q}_0^{\,-1}\, \wt{\q}_1\in\cu^{n-1}(G_\bullet^{+1})$. Thus, the equivalence classes of $\sim$ are the orbits of the action of $\cu^{n-1}(G_\bullet^{+1})$ on $\cu^{n-1}(\ns)$.
\end{lemma}
Here $\pi_\Gamma$ denotes the canonical quotient map $G\to G/\Gamma$.
\begin{proof}
Suppose that $\q_0\sim \q_1$. Thus $\langle \q_0,\q_1\rangle_1\in \cu^n(\ns)$, so there is $\q\in \cu^n(G_\bullet)$ such that $\langle\q_0,\q_1\rangle_1 = \pi_\Gamma\co\q$. For $i\in\{0,1\}$ let $\wt{\q}_i$ be the restriction of $\q$ to the face $\{v\in \db{n}:v\sbr{n}=i\}$. Then $\pi_{\Gamma}\co\wt{\q}_i=\q_i$. Since $\langle\wt{\q}_0,\wt{\q}_1\rangle_1=\q$ is a cube, we have by \cite[Lemma 2.2.19]{Cand:Notes1} that $\wt{\q}_0^{\,-1}\,\wt{\q}_1\in\cu^{n-1}(G_\bullet^{+1})$. The backward implication is also clear, using the backward implication in \cite[Lemma 2.2.19]{Cand:Notes1}. For the last claim, suppose that $\wt{\q}_0\Gamma^{\db{n-1}}\sim\wt{\q}_1\Gamma^{\db{n-1}}$, and note that $\wt{\q}_1\Gamma^{\db{n-1}}=\wt{\q}_0(\wt{\q}_0^{\,-1}\,\wt{\q}_1)\Gamma^{\db{n-1}} = g\, \wt{\q}_0\Gamma^{\db{n-1}}$, where $g:=\wt{\q}_0(\wt{\q}_0^{\,-1}\wt{\q}_1)\wt{\q}_0^{\,-1}$ is in $\cu^{n-1}(G_\bullet^{+1})$ since this is a normal subgroup of $\cu^{n-1}(G_\bullet)$.
\end{proof}
\noindent We use this algebraic expression of the relation $\sim$ to prove the following description of the $\sigma$-algebra $\mc{I}'$ from Lemma \ref{lem:I-expr}, as a key step toward the proof of Proposition \ref{prop:upccaxioms}.
\begin{lemma}\label{lem:I-char}
For each $i\in \mb{N}$ let $\ns_i$ be a \textsc{cfr} coset nilspace $(G^{(i)}/\Gamma^{(i)},G^{(i)}_\bullet)$. Let $\mf{H}$ be the ultraproduct group $\prod_{i\to\omega}\cu^{n-1}\big((G^{(i)})_\bullet^{+1}\big)$. Then a set $A\in \mc{L}_{\mf{X}^{\db{n-1}}}$ is in $\mc{I}'$ if and only if $g\cdot A =_{\mu^{\db{n-1}}} A$ for every $g\in \mf{H}$.
\end{lemma}
\noindent To prove this we first obtain the following analogous result for \textsc{cfr} coset nilspaces.
\begin{lemma}\label{lem:I-char-ns}
Let $\ns$ be a \textsc{cfr} coset nilspace $(G/\Gamma,G_\bullet)$, let $H=\cu^{n-1}(G_\bullet^{+1})$, and let $\mc{J}$ be the $\sigma$-algebra of Borel sets $A\subset \ns^{\db{n-1}}$ such that $p_{F_0}^{-1}(A)=_{\mu_{\ns}^{\db{n}}}p_{F_1}^{-1}(A)$. Then a Borel set $A\subset \ns^{\db{n-1}}$ is in $\mc{J}$ if and only if $g\cdot A=_{\mu_{\ns}^{\db{n-1}}}A$ for every $g\in H$.
\end{lemma}
\noindent Recall that $\mu_{\ns}^{\db{n}}$ denotes the Haar measure on $\cu^n(\ns)$ viewed as a measure on $\ns^{\db{n}}$.
\begin{proof}
Assume that $p_{F_0}^{-1}(A)=_{\mu_{\ns}^{\db{n}}} p_{F_1}^{-1}(A)$, and let $A'=A\cap \cu^{n-1}(\ns)$. Note that every element in $p_{F_0}^{-1}(A')$ that lies in $\cu^n(\ns)$ is of the form $\langle \q_0,\q_1\rangle_1$ for $\q_0\sim\q_1$, with $\q_0\in A'$. Since $\mu_{\ns}^{\db{n}}$ is concentrated on $\cu^n(\ns)$, we have $p_{F_0}^{-1}(A)=_{\mu_{\ns}^{\db{n}}}p_{F_0}^{-1}(A')=_{\mu_{\ns}^{\db{n}}}\{\langle\q_0,g\cdot \q_0\rangle_1:\q_0\in A', g\in H\}$, by Lemma \ref{lem:simchar}. Letting $H'$ denote the group $\{\langle\id_H,g\rangle_1:g\in H\}$, it follows that $p_{F_0}^{-1}(A)=_{\mu_{\ns}^{\db{n}}} g'\cdot p_{F_0}^{-1}(A)$ for every $g'=\langle\id_H,g\rangle_1\in H'$. By our assumption, this implies $p_{F_1}^{-1}(A)=_{\mu_{\ns}^{\db{n}}} g'\cdot p_{F_1}^{-1}(A)$. Moreover $
g'\cdot p_{F_1}^{-1}(A)=_{\mu_{\ns}^{\db{n}}}g'\cdot \{\langle h \cdot \q_1,  \q_1\rangle_1:\q_1\in A', h\in H\}$ and this equals $\{\langle h\cdot \q_1, \q_1\rangle_1:\q_1\in g\cdot A', h\in H\}=_{\mu_{\ns}^{\db{n}}} p_{F_1}^{-1}(g\cdot A)$. Hence $p_{F_1}^{-1}(A)=_{\mu_{\ns}^{\db{n}}} p_{F_1}^{-1}(g\cdot A)$, which implies that $A =_{\mu_{\ns}^{\db{n-1}}} g\cdot A$ as required.

Conversely, if $A =_{\mu_{\ns}^{\db{n-1}}} g\cdot A$ for all $g\in H$, then by \cite[Theorem 3]{Ma} there is $A'=_{\mu_{\ns}^{\db{n-1}}} A$ such that $g\cdot A'= A'$ for every $g\in H$. Using Lemma \ref{lem:simchar} as above yields $p_{F_0}^{-1}(A')=_{\mu_{\ns}^{\db{n}}}\{\langle\q_0, \q_1\rangle_1:\q_0,\q_1\in A, \q_0\sim \q_1\}=_{\mu_{\ns}^{\db{n}}}p_{F_1}^{-1}(A')$, whence $S\in \mc{J}$.
\end{proof}

\begin{proof}[Proof of Lemma \ref{lem:I-char}]
We first prove the forward implication. If $A\in \mc{I}'$, then by definition $\wt{A}:=p_{F_0}^{-1}(A)=_{\mu^{\db{n}}} p_{F_1}^{-1}(A)$, so in particular $\wt{A}\in \mc{B}_{F_0}\wedge \mc{B}_{F_1}$. By Lemma  \ref{lem:upcialgs} there are Borel sets $\wt{A}_i \in \mc{B}_{i,F_0}\wedge \mc{B}_{i,F_1}$, $i\in \mb{N}$,  such that $\wt{A}=_{\mu^{\db{n}}}\prod_{i\to\omega} \wt{A}_i$ (where $\mc{B}_{i,F_0}$ is the analogue of $\mc{B}_{F_0}$ for $\ns_i$). For each $i$, combining the idempotence of $\mu_{\ns_i}^{\db{n}}$ with \cite[Lemma 2.62]{Cand:Notes2} as in previous proofs, we obtain Borel sets $A_i\in \ns_i^{\db{n-1}}$ such that $\wt{A}_i=_{\mu_{\ns_i}^{\db{n}}} p_{F_0}^{-1}(A_i)=_{\mu_{\ns_i}^{\db{n}}}p_{F_1}^{-1}(A_i)$. Hence $p_{F_0}^{-1}(A)=_{\mu^{\db{n}}}\prod_{i\to\omega}p_{F_0}^{-1}(A_i)=_{\mu^{\db{n}}}p_{F_0}^{-1}(\prod_{i\to\omega}A_i)$. Consequently $A=_{\mu^{\db{n-1}}}\prod_{i\to \omega} A_i$. By Lemma \ref{lem:I-char-ns} every such set $A_i$ is $H_i$-invariant for $H_i:=\cu^{n-1}\big((G^{(i)})_\bullet^{+1}\big))$. It follows that $A$ is $\mf{H}$-invariant as required.

Conversely, if $\mu^{\db{n-1}}(A\Delta h\cdot A)=0$ for all $h\in \mf{H}$, then by  \cite[Theorem 2.1]{Ross} there are Borel sets $A_i\subset \ns_i^{\db{n-1}}$ such that $A=_{\mu^{\db{n-1}}}\prod_{i\to\omega} A_i$. For each $i$ let $s_i=\sup_{h\in H_i} \mu_{\ns_i}^{\db{n-1}}\big(A_i\Delta (h\cdot A_i)\big)$. We claim that for every $\epsilon >0$ we have $\{i: s_i< \epsilon\}\in \omega$. Otherwise there is $\epsilon >0$ such that $\{i: s_i\geq  \epsilon\}\in \omega$, so for every such $i$ there is $h_i\in H_i$ such that $\mu_{\ns_i}^{\db{n-1}} \big(A_i\Delta (h_i\cdot A_i)\big)\geq \epsilon/2$. Letting $h=\lim_{i\to \omega}h_i \in \mf{H}$, we would have $\mu^{\db{n-1}} \big(A \Delta (h \cdot A)\big)\geq \epsilon/2 >0$, a contradiction. This proves our claim. Hence, for every $\epsilon>0$, for every $i$ such that $s_i<\epsilon$, by Lemma \ref{lem:finapprox} there is an $H_i$-invariant set $A'_i$ such that $\mu_{\ns_i}^{\db{n-1}}\big(A_i\Delta A_i') \leq 5\epsilon^{1/4}$. Let $A'=\prod_{i\to\omega} A_i'$. Then $\mu^{\db{n-1}}\big(A \Delta A')\leq 5\epsilon^{1/4}$. Since $A_i'\in \mc{J}_i$, we have $A'\in \mc{I}'$ by Lemma \ref{lem:upcialgs}. Letting $\epsilon\to 0$, we deduce that $A\in \mc{I}'$.
\end{proof}
We can now complete the proof of Proposition \ref{prop:upccaxioms}, by proving the following result.
\begin{proposition}
For every pair of opposite faces $F_0,F_1$ of codimension 1 in $\db{n}$, the $\sigma$-algebra $\mc{I}=\mc{B}_{F_0}\wedge \mc{B}_{F_1}$ satisfies $\mc{A}_{F_0}\upmod \mc{I}$.
\end{proposition}
\begin{proof}
As $\mc{A}_{F_0}=p_{F_0}^{-1}(\mc{L}_{\mf{X}}^{\db{n-1}})$ and $\mc{I}=_{\mu^{\db{n}}}p_{F_0}^{-1}(\mc{I}')$, it suffices to show that $\mc{L}_{\mf{X}}^{\db{n-1}}\upmod \mc{I}'$. For this proof let $\mc{A}$ denote $\mc{L}_{\mf{X}}^{\db{n-1}}$. Let $f\in L^\infty(\mc{I}')$ and $h\in \mf{H}$. Then $f^h=_{\mu^{\db{n-1}}}f$, by Lemma \ref{lem:I-char} (where $f^h(x):=f(h\cdot x)$),  
so $\mb{E}(f|\mc{A})=_{\mu^{\db{n-1}}}\mb{E}(f^h|\mc{A})$. Note the global invariance $\mc{A}^h=_{\mu^{\db{n-1}}}\mc{A}$, since $g^h\in L^\infty(\mc{A})$ for every $g\in L^\infty(\mc{A})$ of rank 1. Hence $\mb{E}(f^h|\mc{A})=_{\mu^{\db{n-1}}} \mb{E}(f^h|\mc{A}^h)$. As $h$ is measure preserving, $\mb{E}(f^h|\mc{A}^h)=_{\mu^{\db{n-1}}} \mb{E}(f|\mc{A})^h$, so $\mb{E}(f|\mc{A})=_{\mu^{\db{n-1}}} \mb{E}(f|\mc{A})^h$. This holds for all $h$, so $\mb{E}(f|\mc{A})\in L^\infty(\mc{I}')$. Hence $\mc{I}'\upmod \mc{A}$.
\end{proof}
\begin{remark}\label{rem:gentocompns}
To prove Proposition \ref{prop:upccaxioms}, we have made significant use of the transitive group action present on a \textsc{cfr} coset nilspace. We do not know whether the cubic coupling axioms can be proved for ultraproducts of more general compact nilspaces, where such a group action is not necessarily available. If the axioms still hold in such a  setting, then this may yield an extension of Theorem \ref{thm:reglem-intro} valid for all compact nilspaces.
\end{remark}

\subsection{Locating a separable factor yielding a Borel cubic coupling}\label{sec:abcasesepfact}\hfill \smallskip\\
\noindent Given a probability space $(\Omega,\mc{A},\lambda)$, we say that a $\sigma$-algebra $\mc{X}\subset \mc{A}$ is separable if $L^1_\lambda(\mc{X})$ is separable as a metric space. In this subsection we prove the following result.
\begin{proposition}\label{prop:sepcc}
Let $(\ns_i)_{i\in \mb{N}}$ be a sequence of \textsc{cfr} coset nilspaces. Then for every separable $\sigma$-algebra $\mc{X}_0\subset \mc{L}_{\mf{X}}$ there is a separable $\sigma$-algebra $\mc{X}\subset \mc{L}_{\mf{X}}$ such that $\mc{X}_0\subset \mc{X}$ and such that the Loeb measures $\mu^{\db{n}}$ on the $\sigma$-algebras $\mc{X}^{\db{n}}$ form a cubic coupling.
\end{proposition}
The proof relies on the following couple of lemmas.
\begin{lemma}\label{lem:sep-joins}
Let $(\Omega,\mc{A},\lambda)$ be a probability space and let $S$ be a finite set. For each $v\in S$ let $\mc{X}_v$ be a sub-$\sigma$-algebra of $\mc{A}$, and let $\mc{C}\subset \bigvee_{v\in S} \mc{X}_v$ be a separable $\sigma$-algebra. Then there are separable $\sigma$-algebras $\mc{X}_v'\subset \mc{X}_v$ for $v\in S$ such that $\mc{C}\subset_\lambda\bigvee_{v\in S} \mc{X}_v'$.
\end{lemma}
\begin{proof}
The separability of $\mc{C}$ implies that there is a dense sequence of functions $(f_\ell)_{\ell\in\mb{N}}$ in $L^1(\mc{C})$. 
By \cite[Lemma 2.2]{CScouplings}, for each $\ell$ there is a sequence of functions $(f_{k,\ell})_{k\in \mb{N}}$, where for each $k$ we have $\|f_{k,\ell}-f_\ell\|_{L^1}\leq 1/k$ and $f_{k,\ell}$ is a finite sum of bounded rank 1 functions, i.e.\  $f_{k,\ell}=\sum_{j=1}^{m_{k,\ell}} \prod_{v\in S} g_{v,j,k,\ell}$ where $g_{v,j,k,\ell}\in L^\infty(\mc{X}_v)$ for every $j$. Let $\mc{X}_v'$ be the separable sub-$\sigma$-algebra of $\mc{X}_i$ generated by the collection  $\{g_{v,j,k,\ell}:\ell,k\in \mb{N},j\in [m_{k,\ell}]\}$. This collection is countable, so $\mc{X}_v'$ is separable. Now given any $f\in L^1(\mc{C})$, for any $\epsilon>0$ there is $\ell$ such that $\|f-f_\ell\|_{L^1}< \epsilon/2$, and there is $k$ such that $\|f_\ell-f_{\ell,k}\|_{L^1}< \epsilon/2$, so $\|f-f_{k,\ell}\|_{L^1}<\epsilon$, and by construction $f_{k,\ell}\in L^1(\bigvee_{v\in S} \mc{X}_v')$. Letting $\epsilon\to 0$, we deduce that $\mc{C}\subset_\lambda\bigvee_{v\in S} \mc{X}_v'$.
\end{proof}
\noindent Let us single out the adjacent faces $F_{n,0}:=\{0\}\times \db{n-1}$,\, $F_{n,1}:=\db{n-1}\times\{0\}$ in $\db{n}$. For $p\in [1,\infty]$ we denote by $\mc{U}^p(\mc{A})$ the unit ball of $L^p(\mc{A})$.
\begin{lemma}\label{lem:inductstep}
Let $\mc{C}$ be a separable sub-$\sigma$-algebra of $\mc{L}_{\mf{X}}$. There is a separable $\sigma$-algebra $\mc{D}$ with $\mc{C}\subset \mc{D}\subset \mc{L}_{\mf{X}}$, such that for every $n\in\mb{N}$, for every system $(f_v)_{v\in F_{n,0}}$ of bounded $\mc{C}$-measurable functions $f_v$, we have $\mb{E}\big(\prod_{v\in F_{n,0}} f_v\co p_v|(\mc{L}_{\mf{X}})^{\db{n}}_{F_{n,1}}\big)\in L^\infty(\mc{D}^{\db{n}}_{F_{n,0}\cap F_{n,1}})$.
\end{lemma}
\begin{proof}
By assumption the metric space $L^1(\mc{C})$ is separable, and therefore so is the subset $\mc{U}^\infty(\mc{C})\subset L^1(\mc{C})$, so there is a sequence $\mc{S}\subset \mc{U}^\infty(\mc{C})$ that is dense in $\mc{U}^\infty(\mc{C})$ relatively to the $L^1$-norm. Recall that $\mc{A}$ denotes $\mc{L}_{\mf{X}}^{\db{n}}$. Let $\langle\mc{C}\rangle_n$ denote the sub-$\sigma$-algebra of $\mc{A}_{F_{n,1}}$ generated by all expectations $\mb{E}(\prod_{v\in F_{n,0}} g_v\co p_v|\mc{A}_{F_{n,1}})$ for systems $(g_v)_{v\in F_{n,0}}$ of functions in $\mc{S}$. Since $\langle\mc{C}\rangle_n$ is generated by countably many functions, it is separable. By the conditional independence axiom (Proposition \ref{prop:upccaxioms}) we have $\mb{E}(\prod_{v\in F_{n,0}} g_v\co p_v|\mc{A}_{F_{n,1}})\in L^\infty(\mc{A}_{F_{n,0}\cap F_{n,1}})$. Hence $\langle\mc{C}\rangle_n\subset_\lambda \mc{A}_{F_{n,0}\cap F_{n,1}}$. By Lemma \ref{lem:sep-joins}, there is a separable $\sigma$-algebra $\mc{D}_n\subset \mc{L}_{\mf{X}}$ such that $\langle\mc{C}\rangle_n\subset_\lambda (\mc{D}_n)_{F_{n,0}\cap F_{n,1}}^{\db{n}}$. Let $\mc{D}=\mc{C}\vee\big(\bigvee_{n\in \mb{N}} \mc{D}_n\big)$. Fix any system $\big(f_v\in \mc{U}^\infty(\mc{C})\big)_{v\in F_{n,0}}$. For every $\epsilon > 0$, for each $v$ there is $g_v\in \mc{S}$ such that $\|f_v-g_v\|_{L^1}\leq \epsilon$. Using telescoping sums we have $\| \mb{E}(\prod_{v\in F_{n,0}} f_v\co p_v|\mc{A}_{F_{n,1}})-\mb{E}(\prod_{v\in F_{n,0}} g_v\co p_v|\mc{A}_{F_{n,1}})\|_{L^1}\leq 2^n\,\epsilon$. Letting $\epsilon\to 0$ yields $\mb{E}(\prod_{v\in F_{n,0}} f_v\co p_v|\mc{A}_{F_{n,1}})\in L^1\big((\mc{D}_n)_{F_{n,0}\cap F_{n,1}}^{\db{n}}\big)\subset L^1(\mc{D}_{F_{n,0}\cap F_{n,1}}^{\db{n}})$. The result follows.
\end{proof}

\begin{proof}[Proof of Proposition \ref{prop:sepcc}]
The consistency and ergodicity axioms hold with $\mc{L}_{\mf{X}}$ (by Lemma \ref{lem:12}), so they clearly hold also for any sub-$\sigma$-algebra of $\mc{L}_{\mf{X}}$. In particular, for each $n$ we have to check the conditional independence axiom (for the suitable separable $\sigma$-algebra $\mc{X}\subset \mc{L}_{\mf{X}}$) only for $F_{n,0},F_{n,1}$, rather than for all pairs of adjacent $(n-1)$-faces in $\db{n}$ (indeed, the consistency axiom implies conditional independence for every such pair of faces, once we have it just for $F_{n,0},F_{n,1}$). So let us prove that there is a separable $\sigma$-algebra $\mc{X}\subset \mc{L}_{\mf{X}}$ such that for each $n$, for every system $(f_v)_{v\in F_{n,0}}$ in $L^\infty(\mc{X})$, we have $\mb{E}(\prod_{v\in F_{n,0}} f_v\co p_v|\mc{A}_{F_{n,1}})\in L^\infty(\mc{X}^{\db{n}}_{F_{n,0}\cap F_{n,1}})$ (this is enough, since by \cite[Lemma 2.2]{CScouplings} every integrable $\mc{X}_{F_{n,0}}^{\db{n}}$-measurable function is a limit of finite sums of rank 1 functions $\prod_{v\in F_{n,0}} f_v\co p_v$). If we prove this, then we also have $\mb{E}(\prod_{v\in F_{n,0}} f_v\co p_v|\mc{X}^{\db{n}}_{F_{n,1}})\in L^\infty(\mc{X}^{\db{n}}_{F_{n,0}\cap F_{n,1}})$, since $\mc{X}^{\db{n}}_{F_{n,0}\cap F_{n,1}} \subset \mc{X}^{\db{n}}_{F_{n,1}}\subset \mc{A}_{F_{n,1}}$. To obtain $\mc{X}$, we argue as follows: let $\mc{X}_0$ be the initial separable $\sigma$-algebra in the proposition, and let $(\mc{X}_i)_{i\in \mb{N}}$ be the increasing sequence of separable sub-$\sigma$-algebras of $\mc{L}_{\mf{X}}$ defined inductively by letting $\mc{X}_i$ be the $\sigma$-algebra $\mc{D}$ obtained by applying Lemma \ref{lem:inductstep} with $\mc{C}=\mc{X}_{i-1}$. Let $\mc{X}=\bigvee_{i\geq 0} \mc{X}_i$. To see that this has the required property, fix any $n$ and let $(f_v)_{v\in F_{n,0}}$ be any system of functions in $L^\infty(\mc{X})$. We have to check that $\mb{E}(\prod_{v\in F_{n,0}} f_v\co p_v|\mc{A}_{F_{n,1}})\in L^\infty(\mc{X}^{\db{n}}_{F_{n,0}\cap F_{n,1}})$. It clearly suffices to do this assuming that $f_v\in \mc{U}^\infty(\mc{X})$. Fix any $\epsilon>0$. For each $v$ there is $f_v'\in \mc{U}^\infty(\mc{X}_i)$ for some $i=i(v)$ such that $\|f_v-f_v'\|_{L^1}< \epsilon$ (indeed we can take $f_v'$ to be a version of $\mb{E}(f_v|\mc{X}_i)$). Letting $j=\max_{v\in F_{n,0}} i(v)$, we have $f'_v\in \mc{U}^\infty(\mc{X}_j)$ for all $v$. It then follows by construction and Lemma \ref{lem:inductstep} that $\mb{E}(\prod_{v\in F_{n,0}} f_v'\co p_v|\mc{A}_{F_{n,1}})\in L^\infty\big((\mc{X}_{j+1})^{\db{n}}_{F_{n,0}\cap F_{n,1}}\big)\subset L^\infty\big(\mc{X}^{\db{n}}_{F_{n,0}\cap F_{n,1}}\big)$. As in the previous proof, this expectation converges to $\mb{E}(\prod_{v\in F_{n,0}} f_v\co p_v|\mc{A}_{F_{n,1}})$ as $\epsilon\to 0$, so the latter expectation is also $\mc{X}^{\db{n}}_{F_{n,0}\cap F_{n,1}}$-measurable modulo null sets, as required.
\end{proof}

\section{Stability of morphisms into compact finite-rank nilspaces}\label{sec:ctsmorphism}
\noindent By a \emph{compatible metric} on a topological space $X$ we mean a metric $d$ on $X$ which generates the given topology on $X$. Given such a metric $d$ on $X$, for any $x,y\in X$ and $\epsilon>0$ we write $x\approx_\epsilon y$ to mean that $d(x,y)\leq \epsilon$. Recall that if $G$ is a compact group acting continuously on a metric space $X$ with metric $d$, then we can always define a compatible metric $d'$ on $X$ which is also \emph{$G$-invariant}, meaning that for all $x,y\in X$ and $g\in G$ we have $d'(gx,gy)=d'(x,y)$ (see \cite[Proposition 1.1.12]{Palais}).

Given compact nilspaces $\ns,\nss$, with a compatible metric $d$ on $\nss$, we define a pseudometric $d_1$ on the space of Borel measurable functions $\phi:\ns\to\nss$ by the formula $d_1(\phi_1,\phi_2)=\int_{\ns} d(\phi_1(x),\phi_2(x)\big)\ud\mu_{\ns}(x)$.  

\begin{defn}
Let $\ns,\nss$ be $k$-step compact nilspaces, and let $d$ be a compatible metric on $\nss$. For $\delta>0$, a $(\delta,1)$-\emph{quasimorphism} from $\ns$ to $\nss$ (relative to $d$) is a Borel measurable map $\phi:\ns\to\nss$ satisfying
\begin{equation}\label{eq:quasim}
\mu_{\ns}^{\db{k+1}}\,\big(\{\q\in \cu^{k+1}(\ns): \exists \q'\in \cu^{k+1}(\nss),\,\forall\, v\in \db{k+1},\, \phi\co\q(v)\approx_\delta \q'(v)\}\big)\geq 1-\delta,
\end{equation}
where $\mu_{\ns}^{\db{k+1}}$ denotes the Haar probability measure on $\cu^{k+1}(\ns)$.
\end{defn}
\noindent We write ``$(\delta,1)$-quasimorphism", rather than just ``$\delta$-quasimorphism", to distinguish this notion from the quasimorphisms defined in \cite[Definition 2.8.1]{Cand:Notes2}, which we call here $(\delta,\infty)$\emph{-quasimorphisms}; these are defined by replacing property \eqref{eq:quasim} with the uniform (and stronger) property $\forall\q\in \cu^{k+1}(\ns),\, \exists \q'\in \cu^{k+1}(\nss),\,\forall\, v\in \db{k+1},\, \phi\co\q(v)\approx_\delta \q'(v)$.

In our proof of Theorem \ref{thm:reglem-intro} in Section \ref{sec:mps}, a key ingredient is the following stability (or rigidity) result for morphisms.
\begin{theorem}\label{thm:rigid}
Let $\nss$ be a $k$-step \textsc{cfr} nilspace with compatible metric $d$. For every $\epsilon>0$ there exists $\delta=\delta(\epsilon,\nss)>0$ such that if $\ns$ is a compact nilspace and $\phi:\ns \to\nss$ is a $(\delta,1)$-quasimorphism, then there exists a continuous morphism $\phi':\ns\to\nss$ such that $d_1(\phi,\phi') \leq \epsilon$.
\end{theorem}
\noindent This theorem is an analogue, for $(\delta,1)$-quasimorphisms, of the uniform stability result for $(\delta,\infty)$-quasimorphisms given in \cite[Theorem 5]{CamSzeg} (see also \cite[Theorem 2.8.2]{Cand:Notes2}). Indeed, we obtain the statement of this uniform stability result by replacing in Theorem \ref{thm:rigid}  every ``1" by ``$\infty$" (where $d_\infty(\phi_1,\phi_2)=\sup_{x\in \ns} d(\phi_1(x),\phi_2(x)$).
\subsection{Cocycles close to the 0 cocycle are coboundaries}\hfill\\
Recall that the group $\aut(\db{k})$ of automorphisms of the cube $\db{k}$ is generated by permutations of $[k]=\{1,2,\ldots,k\}$ and coordinate reflections. For $\theta\in \aut(\db{k})$ we write $r(\theta)$ for the number of reflections involved in $\theta$. Equivalently, $r(\theta)$ is the number of coordinates equal to 1 of $\theta(0^k)$. Two $n$-cubes $\q_1,\q_2$ on a nilspace are  \emph{adjacent} if $\q_1(v,1)=\q_2(v,0)$ for all $v\in \db{n-1}$; we can then form their \emph{concatenation}, which is the $n$-cube $\q$ such that $\q(v,0)=\q_1(v,0)$ and $\q(v,1)=\q_2(v,1)$ for all $v\in \db{n-1}$ (see \cite[Lemma 3.1.7]{Cand:Notes1}). 

We now recall the definition of a nilspace cocycle, which is fundamental to the structural analysis of nilspaces (see \cite[Definition 2.14]{CamSzeg} or \cite[Definition 3.3.14]{Cand:Notes1}).
\begin{defn}\label{def:cocycle} 
Let $\ns$ be a nilspace, $\ab$ an abelian group, and $k\in \mb{Z}_{\geq -1}$. A $\ab$-valued \emph{cocycle of degree} $k$ on $\ns$ is a function $\rho: \cu^{k+1}(\ns)\to \ab$ with the following properties:\\ \vspace{-0.6cm}
\begin{enumerate}[leftmargin=0.7cm]
\item If $\q \in \cu^{k+1}(\ns)$ and $\theta\in \aut(\db{k+1})$, then $\rho(\q\co\theta )=(-1)^{r(\theta)}\rho(\q)$.
\item If $\q_3$ is the concatenation of cubes $\q_1,\q_2\in \cu^{k+1}(\ns)$ then $\rho(\q_3)=\rho(\q_1)+\rho(\q_2)$.
\end{enumerate}
\end{defn}
\noindent We recall also that for any $n\in \mb{N}$ and any group $G$ we denote by $\sigma_n$ the Gray-code map $G^{\db{n}}\to G$ from \cite[Definition 2.2.22]{Cand:Notes1}; in particular if $G$ is abelian we have $\sigma_n(g):=\sum_{v\in \db{n}} (-1)^{|v|} g(v)$ for every $g: \db{n}\to G$. Using this notation, we say that a cocycle $\rho$ of degree $k$ on $\ns$ is a \emph{coboundary} (of degree $k$) if there is a function $f:\ns\to\ab$ such that $\rho(\q)=\sigma_{k+1}(f\co\q)$ for every $\q\in \cu^{k+1}(\ns)$. We refer to \cite[\S 3.3.3]{Cand:Notes1} for more background on cocycles and coboundaries.

The proof of Theorem \ref{thm:rigid}, given in Subsection \ref{subsec:proofstab},  relies on the following stability result for cocycles, which is the main result in this subsection.
\begin{proposition}\label{prop:costab}
Let $\ab$ be a compact abelian group, and let $d_Z$ be a compatible $\ab$-invariant metric on $\ab$. There exists $\epsilon>0$ such that the following holds. If $\ns$ is a compact nilspace and $\rho:\cu^k(\ns)\to\ab$ is a Borel cocycle such that $d_1(0,\rho):=\int_{\cu^k(\ns)} d_{\ab}\big(\rho(\q),0_{\ab}\big)\ud\mu_{\cu^k(\ns)}(\q)\leq \epsilon$, then $\rho$ is a coboundary.
\end{proposition}
\noindent A key element in the proof of Proposition \ref{prop:costab} is the following result.
\begin{lemma}\label{lem:coremoval}
Let $\ns$ be a compact nilspace, let $\ab$ be a compact abelian group with compatible $\ab$-invariant  metric $d_{\ab}$, let $\rho:\cu^k(\ns)\to\ab$ be a Borel measurable cocycle, let $0<\epsilon< 2^{-4k}$, and suppose that $d_1(\rho,0)\leq\epsilon$. Then there is a Borel set $S\subset \ns$ such that $\mu_{\ns}(S)> 1-\epsilon^{1/2}$ and $d_{\ab}(\rho(\q),0)\leq 2^k\epsilon^{1/4}$ for every $\q\in \cu^k(\ns)\cap S^{\db{k}}$.
\end{lemma}
\noindent The proof employs tricubes, which are very useful tools in nilspace theory (\cite[\S 3.1.3]{Cand:Notes1}), especially because they enable an operation akin to convolution (called \emph{tricube composition}) to be performed with cubes (see \cite[Lemma 3.1.16]{Cand:Notes1}). A crucial property of cocyles, which is used repeatedly in this section, is that they commute with this operation in the sense captured in \cite[Lemma 2.18]{CamSzeg} (see also \cite[Lemma 3.3.31]{Cand:Notes1}).
\begin{proof}[Proof of Lemma \ref{lem:coremoval}]
Let 
\[
S=\big\{x\in \ns:  \mu_{\cu^k_x(\ns)}\big(\{\q\in \cu^k_x(\ns): d_{\ab}(\rho(\q),0)\leq \epsilon^{1/4}\}\big)\geq 1-\epsilon^{1/4}\big\},
\]
where $\cu^k_x(\ns):=\{\q\in\cu^k(\ns):\q(0^k)=x\}$, and $\mu_{\cu^k_x(\ns)}$ denotes the Haar probability measure on $\cu^k_x(\ns)$ (see \cite[Lemma 2.2.17]{Cand:Notes2}). By Markov's inequality, we have
\[
\mu_{\ns}(\ns\setminus S)\;\epsilon^{1/2} < \int_{\ns} \int_{\cu^k_x(\ns)} d_{\ab}(\rho(\q),0) \ud\mu_{\cu^k_x(\ns)}(\q) \ud\mu_{\ns}(x) = d_1(\rho,0) \leq \epsilon.
\]
Hence $\mu_{\ns}(S)> 1-\epsilon^{1/2}$.

Now if $\q\in\cu^k(\ns)\cap S^{\db{k}}$, then for each $v\in \db{k}$, by definition of $S$ there is a measure at least $1-\epsilon^{1/4}$ of cubes $\q'\in \cu^k_{\q(v)}(\ns)$ such that $d_{\ab}(\rho(\q'),0)\leq \epsilon^{1/4}$. Recall that the restricted tricube space $\mc{T}(\q):=\hom_{\q\co\omega_k^{-1}}(T_k,\ns)$, being an iterated compact abelian bundle, has a Haar measure (see \cite[Lemma 2.2.12]{Cand:Notes2}, and see \cite[Definition 3.1.15]{Cand:Notes1} for the notion of the \emph{outer-point map} $\omega_k$). Let us denote this Haar measure by $\mu_{\mc{T}(\q)}$. For each $v\in \db{k}$ the map $\mc{T}(\q)\to\cu^k_{\q(v)}(\ns)$, $t\mapsto t\co \Psi_v$ takes this measure $\mu_{\mc{T}(\q)}$ to the Haar measure on $\cu^k_{\q(v)}(\ns)$ (see \cite[Corollary 2.2.22]{Cand:Notes2}, and see \cite[Definition 3.1.13]{Cand:Notes1} for the maps $\Psi_v$). It follows from this and the union bound that
\[
\mu_{\mc{T}(\q)}\big(\big\{t\in \mc{T}(\q): \forall \,v\in \db{k}, \, d_{\ab}\big(\rho(t\co \Psi_v),0\big)\leq \epsilon^{1/4}\big\}\big)\geq 1-2^k \epsilon^{1/4}.
\]
Our assumption for $\epsilon$ implies that this measure is positive, so there exists $t\in \mc{T}(\q)$ with this property, namely such that $d_{\ab}\big(\rho(t\co \Psi_v),0\big)\leq \epsilon^{1/4}$ for every $v\in \db{k}$. For this tricube $t$, we apply the formula $\rho(\q)=\sum_{v\in \db{k}}(-1)^{|v|}\rho(t\co\Psi_v)$, which holds for every tricube in $\mc{T}(\q)$ by \cite[Lemma 3.3.31]{Cand:Notes1}. By the triangle inequality and $\ab$-invariance of $d_{\ab}$, we obtain $d_{\ab}(\rho(\q),0)\leq \sum_{v\in \db{k}} d_{\ab}(\rho(t\co\Psi_v),0)\leq  2^k \epsilon^{1/4}$, as claimed.
\end{proof}
\noindent Using the set $S$ provided by Lemma \ref{lem:coremoval}, we can define a function $g:\ns\to\ab$ such that, subtracting the coboundary $\q\mapsto \sigma_k(g\co\q)$ from $\rho$, we obtain a new cocycle $\rho'$ whose values are \emph{uniformly} close to $0$ (not just close in $d_1$), as follows.
\begin{lemma}\label{lem:approxfn}
Let $\ns$ be a compact nilspace, let $\ab$ be a compact abelian group with compatible  $\ab$-invariant metric $d_{\ab}$, let $C$ denote the diameter of $\ab$ relative to $d_{\ab}$, let $\rho:\cu^k(\ns)\to\ab$ be a Borel cocycle, let $\epsilon\in (0,2^{-4k})$, and suppose that $d_1(\rho,0)\leq \epsilon$. Then there is a Borel function $g:\ns\to\ab$ with $d_1(g,0)\leq (2+C)4^k\epsilon^{1/4}$ such that $\rho':\q\mapsto \rho(\q)-\sigma_k(g\co\q)$ satisfies $d_{\ab}(\rho'(\q),0)\leq 8^k\epsilon^{1/4}$, $\forall\q\in \cu^k(\ns)$.
\end{lemma}
\begin{proof}
Let $S$ be the subset of $\ns$ given by Lemma \ref{lem:coremoval}. 

We claim that for every $x\in \ns$ there exists an element $g(x)\in \ab$ such that 
\begin{equation}\label{eq:rootco}
\mu_{\cu^k_x(\ns)} \big(\big\{\q\in\cu^k_x(\ns):d_{\ab}\big(\rho(\q),g(x)\big)\leq 4^k\epsilon^{1/4}\big\}\big)> 1-4^k\epsilon^{1/2}.
\end{equation}
To see this, fix any $x\in \ns$, and note that for each $v\neq 0^k$, the map $\cu^k_x(\ns)\to \ns$, $\q\mapsto \q(v)$ preserves the Haar measures (by \cite[Lemma 2.2.14]{Cand:Notes2} with $n=k$, $P=\db{k}$, $P_1=\{0^k\}$, $P_2=\{v\}$). Since $\mu(S)> 1-\epsilon^{1/2}$, by the union bound we therefore have 
$ \mu_{\cu^k_x(\ns)} \big(\big\{\q\in\cu^k_x(\ns): \forall\,v\neq 0^k,\, \q(v)\in S\big\}\big)> 1-(2^k-1) \epsilon^{1/2}$. Fix any cube $\q_0\in \cu^k_x(\ns)$ with $\q_0(v)\in S$ for every $v\neq 0^k$. Combining the last inequality with the fact (used in the previous proof) that the map $\mc{T}(\q_0)\to \cu^k_{\q_0(v)}(\ns)$, $t\mapsto t\co\Psi_v$ preserves the Haar measures, we deduce by the union bound that
\[
\mu_{\mc{T}(\q_0)}\big(\big\{t\in\mc{T}(\q_0): \forall\,v\neq 0^k,\, t\co\Psi_v \in S^{\db{k}}\big\}\big) > 1-(2^k-1)^2 \epsilon^{1/2} > 1-4^k\epsilon^{1/2}.
\]
Let $g(x):=\rho(\q_0)$, and note that $\q_0$ can be chosen to make the function $g:\ns\to\ab$ Borel, by \cite[Theorem (12.16), (12.18)]{Ke} and the continuity of the map $\q\mapsto \q(0^k)$.

For every tricube $t$ in the above set, we have $\rho(\q_0)= \sum_{v\in \db{k}} (-1)^{|v|} \rho(t\co\Psi_v)$ and, for every $v\neq 0^k$, since $t\co\Psi_v\in S^{\db{k}}$, we have $d_{\ab}(\rho(t\co\Psi_v),0)\leq 2^k\epsilon^{1/4}$ by Lemma \ref{lem:coremoval}. We deduce that $d_{\ab}\big(g(x),\rho(t\co\Psi_{0^k})\big) \leq 4^k\epsilon^{1/4}$. Hence
\begin{equation}\label{eq:trico1}
\mu_{\mc{T}(\q_0)}\big(\big\{t\in\mc{T}(\q_0): g(x)\approx_{4^k\epsilon^{1/4}}\rho(t\co\Psi_{0^k})\big\}\big) > 1-4^k\epsilon^{1/2}.
\end{equation}
Since the map $\mc{T}(\q_0)\to \cu^k_x(\ns)$, $t\mapsto t\co\Psi_{0^k}$ preserves the Haar measures, we have that \eqref{eq:trico1} is equivalent to \eqref{eq:rootco}, which proves our claim.

Define the coboundary $f:\cu^k(\ns)\to\ab$ by $f(\q)=\sigma_k(g\co\q)$. Fix any cube $\q\in \cu^k(\ns)$. By the measure-preserving properties used earlier, the union bound, and \eqref{eq:rootco}, we have
\[
\mu_{\mc{T}(\q)}\big(\big\{t\in\mc{T}(\q): \forall\,v\in\db{k},\, d_{\ab}\big(\rho(t\co\Psi_v),g\co\q(v)\big)\leq 4^k\epsilon^{1/4} \big\}\big) > 1-8^k\epsilon^{1/2}.
\]
By our assumption on $\epsilon$ we have $8^k\epsilon^{1/2}<1$, so there exists $t\in\mc{T}(\q)$ with the above property. Applying the formula $\rho(\q)= \sum_{v\in \db{k}} (-1)^{|v|} \rho(t\co\Psi_v)$ for this $t$, and the triangle inequality (and shift invariance of $d_{\ab}$), we deduce that $d_{\ab}\big(\rho(\q),f(\q)\big)\leq 8^k\epsilon^{1/4}$, as required. Finally, we have
\begin{eqnarray*}
& & d_1(g,0)\;  = \int_{\ns} d_{\ab}(g(x),0) \ud\mu_{\ns}(x) \; = \; \int_{\ns}\int_{\cu^k_x(\ns)}  d_{\ab}(g(x),0) \ud\mu_{\cu^k_x(\ns)}(\q)\ud\mu_{\ns}(x) \\
& \leq & \int_{\ns}\int_{\cu^k_x(\ns)}  d_{\ab}(g(x)-\rho(\q),0) \ud\mu_{\cu^k_x(\ns)}(\q)\ud\mu_{\ns}(x) \,+\,\int_{\cu^k(\ns)}  d_{\ab}(\rho(\q),0) \ud\mu_{\cu^k(\ns)}(\q).
\end{eqnarray*}
The latter integral is $d_1(\rho,0)$, and by \eqref{eq:rootco} the former integral is at most $(1+C) 4^k\epsilon^{1/4}$. Hence $d_1(g,0)\leq d_1(\rho,0)  + (1+C) 4^k\epsilon^{1/4}\leq (2+C)4^k\epsilon^{1/4}$, as required.
\end{proof}
We can now complete the proof of the stability result for cocycles.
\begin{proof}[Proof of Proposition \ref{prop:costab}]
We know by \cite[Lemma 2.5.7]{Cand:Notes2} that there exists $\epsilon_0>0$ depending only on $\ab$ and $k$ such that if a cocycle $\rho':\cu^k(\ns)\to \ab$ takes all its values within distance $\epsilon_0$ of $0_{\ab}$, then $\rho'$ is a coboundary. Applying Lemma \ref{lem:approxfn} with $\epsilon$ sufficiently small in terms of $\epsilon_0$ and $k$, we conclude that $\rho-f$ is a coboundary, where $f(\q)=\sigma_k(g\co \q)$. Since $f$ is also a coboundary, it follows that $\rho$ is a coboundary.
\end{proof}

\subsection{Proof of the stability result for morphisms}\label{subsec:proofstab}\hfill\\
\noindent Given a $k$-step nilspace $\ns$, for $j\in [k]$ we denote by $\ns_j$ the $j$-th \emph{factor} of $\ns$ (also denoted by $\mc{F}_j(\ns)$, with $\mc{F}_k(\ns)=\ns$), and by $\pi_j$ the factor map $\ns\to\ns_j$ (see \cite[Lemma 3.2.10]{Cand:Notes1}). If $\ns$ is compact, with a compatible $\ab_k$-invariant metric $d$, we can always metrize $\ns_{k-1}$ with the quotient metric corresponding to $d$ the standard way (see \cite[(2.2)]{Cand:Notes2}).

We shall use the following rectification result for cubes (see  \cite[Lemma 2.8.3]{Cand:Notes2}).
\begin{lemma}\label{lem:verticubapprox}
Let $\ns$ be a $k$-step compact nilspace with compatible $\ab_k$-invariant metric $d$, and let $d'$ be the quotient metric on $\ns_{k-1}$. For every $\epsilon>0$ there exists $\delta>0$ such that the following holds. If $\q\in \cu^{k+1}(\ns)$ satisfies $d'\big(\pi_{k-1}\co \q(\cdot,0),\pi_{k-1}\co\q(\cdot,1)\big)\leq \delta$ on $\db{k}$, then there is $\q'\in \cu^{k+1}(\ns)$ with $\q\approx_\epsilon \q'$ and $\pi_{k-1}\co \q'(\cdot,0) = \pi_{k-1}\co\q'(\cdot,1)$ on $\db{k}$.
\end{lemma}
\noindent Recall from \cite[Definition 2.2.30]{Cand:Notes1} the notation $\mc{D}_k(\ab)$ for the degree-$k$ nilspace structure on an abelian group $\ab$. In our proof of Theorem \ref{thm:rigid}, we argue by induction on $k$. Each step of the induction uses the following special case of the theorem. 
\begin{lemma}\label{lem:stabiabcase}
Let $\ab$ be a compact abelian Lie group equipped with a compatible $\ab$-invariant metric $d_{\ab}$, and let $k\in\mb{Z}_{\geq 0}$. For every $\epsilon>0$ there exists $\delta=\delta(\epsilon,k,\ab) > 0$ such that if $\phi$ is a $(\delta,1)$-quasimorphism from a compact $k$-step nilspace $\ns$ to $\mc{D}_k(\ab)$, then there is a morphism $\phi':\ns\to \mc{D}_k(\ab)$ such that $d_1(\phi,\phi')\leq \epsilon$. 
\end{lemma}
\begin{proof}
Let $C$ be the diameter of $\ab$ relative to $d_{\ab}$. Let $\delta'\in \big(0, \epsilon/(2+C)\big)$ be sufficiently small for the conclusion of \cite[Theorem 2.8.2]{Cand:Notes2} to hold with initial parameter $\epsilon/2$, for every $(\delta',\infty)$-quasimorphism $\ns\to\mc{D}_k(\ab)$. Let $0<\delta <\delta'^4/\big(8^{4(k+1)} (2^{k+1}+C)\big)$.

Let $\rho$ be the coboundary $\q\mapsto \sigma_{k+1}(\phi\co\q)$. From our assumption, inequality \eqref{eq:quasim}, and the definition of the cube structure on $\mc{D}_k(\ab)$ (see \cite[formula (2.9)]{Cand:Notes1}) it follows that $d_1(\rho,0)\leq (2^{k+1}+C)\delta$. By Lemma \ref{lem:approxfn} applied with $\epsilon_0=(2^{k+1}+C)\delta$, there exists a Borel function $g:\ns\to \ab$ such that $d_{\ab}\big(\rho(\q)-\sigma_{k+1}(g\co \q),0\big)\leq 8^{k+1}\epsilon_0^{1/4}<\delta'$ for \emph{every} cube $\q\in\cu^{k+1}(\ns)$. Equivalently, the map $\phi_1:\ns\to\ab$, $x\mapsto \phi(x)-g(x)$ satisfies $d_{\ab}\big(\sigma_{k+1}(\phi_1\co\q),0\big)\leq \delta'$. Let $\q'\in\cu^{k+1}\big(\mc{D}_k(\ab)\big)$ be the cube such that $\q'(v)=\phi_1\co\q(v)$ for $v\neq 0^{k+1}$ and $\q'(0^{k+1})=\phi_1\co\q(0^{k+1})-\sigma_{k+1}(\phi_1\co\q)$ (note that $\q'$ is indeed in $\cu^{k+1}\big(\mc{D}_k(\ab)\big)$ since $\sigma_{k+1}(\q')=0$). We clearly have $d_{\ab}\big(\q'(v),\phi_1\co\q(v)\big)\leq \delta'$ for every $v\in\db{k+1}$. We have thus shown that $\phi_1$ is a $(\delta',\infty)$-quasimorphism. 

We can thus apply \cite[Theorem 2.8.2]{Cand:Notes2} to conclude that there is a continuous morphism $\phi': \ns\to \mc{D}_k(\ab)$ such that $d_{\ab}\big(\phi_1(x),\phi'(x)\big)\leq \epsilon /2$ for all $x\in \ns$. Hence $d_1(\phi,\phi')\leq d_1(\phi,\phi_1)+ d_1(\phi_1,\phi')\leq d_1(g,0)+\epsilon/2$. By Lemma \ref{lem:approxfn} we have $d_1(g,0)\leq (2+C)4^{k+1}\epsilon_0^{1/4}=\frac{(2+C)\delta'}{2^{k+1}}\leq \epsilon/2$.
\end{proof}
\noindent We need one more lemma before the proof of Theorem \ref{thm:rigid}. This lemma enables us to lift certain Borel maps, and is useful for the inductive step in the proof of the theorem.
\begin{lemma}\label{lem:lift}
Let $\nss$ be a $k$-step \textsc{cfr} nilspace, with $k$-th structure group $\ab_k$, let $d$ be a $\ab_k$-invariant compatible metric on $\nss$, with corresponding quotient metric $d'$ on $\nss_{k-1}$. For every $\epsilon>0$ there exists $\delta>0$ such that the following holds. Let $\ns$ be a  $k$-step compact nilspace, let $\phi:\ns\to\nss$ be a Borel map, let $\phi_1=\pi_{k-1,\nss}\co \phi:\ns\to\nss_{k-1}$, and let $\phi_2:\ns\to\nss_{k-1}$ be a continuous map such that for some Borel set $A\subset \ns$ we have  $d'(\phi_1(x),\phi_2(x))<\delta$ for every $x\in A$. Then there is a Borel map $\phi_3:\ns\to\nss$ such that for every $x\in \ns$, $\pi_{k-1,\nss}\co \phi_3(x)=\phi_2(x)$, and for every $x\in A$, $d(\phi(x),\phi_3(x))<\epsilon$.
\end{lemma}
\begin{proof}
By Gleason's slice theorem $\nss$ is a locally trivial $\ab_k$-bundle over $\nss_{k-1}$ (see \cite[Proposition 2.5.2]{Cand:Notes2}). Hence, for each $y\in \nss_{k-1}$ there is $\delta_y>0$ such that the $\ab_k$-bundle $\nss$ trivializes over the closed ball $\overline{B_{\delta_y}(y)}\subset \nss_{k-1}$. Thus we have a $\ab_k$-bundle isomorphism $\theta_y: \pi_{k-1}^{-1}\big(\overline{B_{\delta_y}(y)}\big)\to \overline{B_{\delta_y}(y)}\times \ab_k$, $w\mapsto (\pi_{k-1}(w),z)$, i.e., $\theta_y$ is a $\ab_k$-equivariant homeomorphism (where the action of $\ab_k$ on $B_{\delta_y}(y)\times \ab_k$ is defined by $z'\cdot (\pi_{k-1}(w),z)=(\pi_{k-1}(w),z+z')$). By uniform continuity of $\theta_y^{-1}$ on the compact set $\overline{B_{\delta_y}(y)}\times\ab_k$, there is $\delta_y'>0$ such that, letting $d''$ denote the metric $d'+d_{\ab_k}$ on $B_{\delta_y}(y)\times \ab_k$ (with $d_{\ab_k}$ the metric on $\ab_k$), we have $d''\big(\theta_y(w),\theta_y(w')\big)\leq \delta_y'$ $\Rightarrow$ $d(w,w')\leq \epsilon$.

Since the balls $B_{\delta_y/2}(y)$ cover $\nss_{k-1}$, by compactness there is a finite subcover by balls $B_{\delta_i/2}(y_i)$, $i\in [M]$, where $\delta_i=\delta_{y_i}$. Thus $\nss$ trivializes over each ball $\overline{B_{\delta_i}(y_i)}$. Let $\delta<\frac{1}{2}\min\{\delta_i,\delta'_{y_i}:i\in [M]\}$. Then, for each $x\in \ns$, there is $i\in [M]$ such that $d'(\phi_2(x),y_i)<\delta_i/2$, whence if $x\in A$ then $d'\big(\phi_1(x),y_i\big)\leq d'\big(\phi_1(x),\phi_2(x)\big)+d'\big(\phi_2(x),y_i\big)< \delta + \delta_i/2 < \delta_i$.
In particular, for every $x\in A$ there is $i\in [M]$ such that $\phi_1(x),\phi_2(x)\in B_{\delta_i}(y_i)$.

Now we claim that for each $i\in [M]$ there is a Borel function $f_i:\phi_2^{-1}\big(B_{\delta_i/2}(y_i)\big)\to\nss$ such that $\pi_{k-1}\co f_i=\phi_2$ and $d\big(f_i(x),\phi(x)\big)\leq \epsilon$ for all $x\in A\cap\phi_2^{-1}\big(B_{\delta_i/2}(y_i)\big)$. To see this, let $\theta_i=\theta_{y_i}: \pi_{k-1}^{-1}\big(B_{\delta_i}(y_i)\big)\to B_{\delta_i}(y_i)\times \ab_k$, $y\mapsto (\pi_{k-1}(y),z)$ be the trivializing bundle isomorphism.
Fix any $x\in \ns$, and let $i$ be such that $\phi_2(x)\in B_{\delta_i/2}(y_i)$. If $x\in A$ then, since $\phi_1(x)\in B_{\delta_i}(y_i)$, there is $z_x\in \ab_k$ such that $\theta_i\co\phi(x)=(\phi_1(x),z_x)$. In this case let $f_i(x):=\theta_i^{-1}(\phi_2(x),z_x)$. If $x\in \phi_2^{-1}\big(B_{\delta_i/2}(y_i)\big)\setminus A$, then we just let $f_i(x)=\cs\co\phi_2(x)$, where $\cs:\nss_{k-1}\to\nss$ is a fixed Borel cross section for $\nss$ (which always exists for such bundles, see \cite[Lemma 2.4.5]{Cand:Notes2}). Thus clearly $\pi_{k-1}\co f_i=\phi_2$. We can see that $f_i$ is Borel  as follows. Let $p_2$ denote the projection to the $\ab_k$ component on $B_{\delta_i}(y_i)\times \ab_k$. Let $g$ denote the function which ``corrects" the $\ab_k$ component of $\cs\co \phi_2(x)$, namely $g:x\mapsto\theta_i\co\cs\co \phi_2(x) + \big(p_2\co\theta_i\co\phi(x)-p_2\co \theta_i\co\cs\co \phi_2(x)\big)=(\phi_2(x),z_x)$. Then $g$ is Borel, and $f_i(x)=\theta_i^{-1}\co g(x)$ for $x\in A$, so $f_i$ is also Borel. Let us now confirm that $d\big(f_i(x),\phi(x)\big)\leq \epsilon$ for all $x\in A\cap \phi_2^{-1}\big(B_{\delta_i/2}(y_i)\big)$. Since $\theta_i\co f_i(x)$ and $\theta_i\co\phi(x)$ have the same $\ab_k$-component $z_x$ (by construction of $f_i$), we have $d''(\theta_i\co f_i(x),\theta_i\co \phi(x))=d'\big(\phi_2(x),\phi_1(x)\big)\leq \delta$. Hence, since $\delta<\delta_i'$, we have $d(f_i(x),\phi(x))\leq \epsilon$ by the choice of $\delta_i'$ above. This proves our claim.

We can greedily form a Borel partition of the domain of $\phi_2$ out of the sets $\phi_2^{-1}(B_{\delta_i/2}(y_i))$. Thus with each $x$ in this domain we associate a unique $i\in [M]$ such that $\phi_2(x)\in B_{\delta_i/2}(y_i)$. We set $\phi_3(x):= f_i(x)$, which makes $\phi_3$ a Borel function.
\end{proof}

\begin{proof}[Proof of Theorem \ref{thm:rigid}]
We argue by induction on $k$. The case $k=0$ is trivial (a non-empty 0-step nilspace is a one-point nilspace). For $k>0$, let $\phi:\ns\to\nss$ be a $(\delta,1)$-quasimorphism relative to the given compatible metric $d$. Note that letting $\tilde{d}$ be the corresponding $\ab_k$-invariant metric on $\nss$ (see \cite[Lemma 2.1.11]{Cand:Notes2}), the identity map on $\nss$ is uniformly continuous $(\nss,d)\to (\nss,\tilde{d})$, so $\phi$ is a $(\tilde{\delta},1)$-quasimorphism relative to $\tilde{d}$ for some $\tilde{\delta}(\delta)>0$ with $\tilde{\delta}=o(1)_{\delta\to 0}$, and therefore we may relabel $\tilde{d},\tilde{\delta}$ as $d,\delta$ and assume without loss of generality that $d$ was already $\ab_k$-invariant. Now let $\phi_1'=\pi_{k-1}\co\phi$, and note that $\phi_1'$ is also a $(\delta,1)$-quasimorphism relative to the quotient metric $d'$ on $\nss_{k-1}$. By induction, for some positive $\delta_1=\delta_1(\delta)=o(1)_{\delta\to 0}$, there exists a continuous morphism $\phi_2:\ns\to\nss_{k-1}$ such that $d_1(\phi_2,\phi_1')\leq \delta_1$. This implies by Markov's inequality that for some Borel set $A\subset \ns$ with $\mu_{\ns}(A)\geq 1-\delta_1^{1/2}$ we have $d'(\phi_2(x),\phi_1'(x))\leq \delta_1^{1/2}$ for all $x\in A$. Applying Lemma \ref{lem:lift} with initial parameter $\delta_2>0$, we obtain a Borel map $\phi_3:\ns\to\nss$ such that $\phi_2=\pi_{k-1}\co\phi_3$ and $d(\phi(x),\phi_3(x)\big)\leq \delta_2= o(1)_{\delta\to 0}$ for every $x\in A$, which implies that $d_1(\phi,\phi_3)<\delta_2+\delta_1^{1/2} C$, where $C$ is the diameter of $(\nss,d_{\nss})$. Note that this implies that $\phi_3$ is also a $(\delta',1)$-quasimorphism for some positive $\delta'=o(1)_{\delta\to 0}$, and what we have gained compared to $\phi$ is that $\phi_3$ is a lift of the \emph{morphism} $\phi_2$ (i.e.\ $\pi_{k-1}\co\phi_3=\phi_2$). We shall now use this to show that $\phi_2$ can in fact be lifted to a continuous morphism $\psi:\ns\to\nss$ (not just to a quasimorphism like $\phi_3$).

Let $W$ be the fiber product $\{(x,y)\in \ns\times \nss: \phi_2(x)=\pi_{k-1,\nss}(y)\}$. This is a compact sub-nilspace of the product nilspace $\ns\times\nss$, i.e.\ $W$ is a $k$-step compact nilspace if we equip it with the cubes $\q$ on the product nilspace $\ns\times \nss$ such that $\q$ takes values in $W$ (see the proof of \cite[Lemma 4.2]{CGS}, applied taking $\psi_1$ in that proof to be $\pi_{k-1,\nss}$ here). Note that this $k$-step nilspace $W$ is an extension of degree $k$ of $\ns$ by the abelian group $\ab_k(\nss)$, because the action of $\ab_k(\nss)$ on the $\nss$-component of $W$ is transitive on each fiber of the projection $\pi:W\to \ns$, $(x,y)\mapsto x$ (recall \cite[Definition 3.3.13]{Cand:Notes1}). 

The map $\phi_3$ induces a Borel cross section $\cs:\ns\to W$, $x\mapsto (x,\phi_3(x))$. With this cross section we can associate a cocycle following \cite[Lemma 3.3.21]{Cand:Notes1}, namely the cocycle $\rho_{\cs}:\cu^{k+1}(\ns)\to\ab_k(\nss)$ defined by $\q\mapsto \sigma_{k+1}(\cs\co\q -\q')$ for any cube $\q'\in\cu^{k+1}(W)$ such that $\pi\co\q'=\q$. It then follows from the definitions that $\rho_{\cs}(\q)=\sigma_{k+1}(\phi_3\co\q-\q'')$ for any $\q''\in\cu^{k+1}(\nss)$ such that $\pi_{k-1,\nss}\co\q''=\phi_2\co\q$. Since $d_1(\phi,\phi_3)<\delta_2+\delta_1^{1/2} C$, and $\phi$ is a $(\delta,1)$-quasimorphism, we deduce using Lemma \ref{lem:verticubapprox} that $d_1(\rho_{\cs},0)<\delta_3$, where $\delta_3>0$ tends to $0$ as $\delta\to 0$ (recall that $\delta_1,\delta_2$ are both $o(1)_{\delta\to 0}$). By Proposition  \ref{prop:costab}, $\rho_{\cs}$ is a coboundary, so $W$ is a split extension of $\ns$, whence there is a Borel morphism $\psi:\ns\to\nss$ such that $\pi_{k-1}\co\psi=\phi_2$, and $\psi$ is then continuous by \cite[Theorem 2.4.6]{Cand:Notes2}. 

Let $\phi_4:\ns\to\mc{D}_k(\ab_k(\nss))$, $x\mapsto\phi_3(x)-\psi(x)$, where the subtraction here is enabled by the fact that $\phi_3(x),\psi(x)$ lie in the same fiber of $\pi_{k-1}$ in $\nss$ (every such fiber is an affine copy of the group $\ab_k(\nss)$; see \cite[Corollary 3.2.16]{Cand:Notes1}). Note that $\phi_4$ is a $(\delta_4,1)$-quasimorphism for some positive $\delta_4=\delta_4(\delta)=o(1)_{\delta\to 0}$. By Lemma \ref{lem:stabiabcase} there is a continuous morphism $\phi_5:\ns \to\mc{D}_k(\ab_k)$ such that $d_1(\phi_4-\phi_5,0)<\delta_5$ for some positive $\delta_5=\delta_5(\delta)=o(1)_{\delta\to 0}$. Now let $\phi':\ns\to\nss$, $x\mapsto\psi(x)+\phi_5(x)$. Then $\phi'$ is a continuous morphism and  $d_1(\phi,\phi')\leq d_1(\phi,\psi+\phi_4) + d_1(\psi+\phi_4,\phi') =   d_1(\phi,\phi_3) + d_1(\phi_4-\phi_5,0) < \delta_2+\delta_1^{1/2} C+\delta_5$, which is less than $\epsilon$ for $\delta$ sufficiently small.
\end{proof}

\section{Proof of the regularity and inverse theorems}\label{sec:mps}

\noindent Recall that given a Polish space $\nss$, the space $\mc{P}(\nss)$ of Borel probability measures on $\nss$ equipped with the weak topology is metrizable, and is in fact a Polish space (see \cite[Theorems (17.23) and (17.19)]{Ke}). Given a nilspace morphism $\phi:\ns\to\nss$ and $n\in \mb{N}$, we denote by $\phi^{\db{n}}$ the map $\cu^n(\ns)\to \cu^n(\nss)$, $\q\mapsto \phi\co\q$.

In the decomposition given by Theorem \ref{thm:reglem-intro}, the structured part is guaranteed to have the following useful property.
\begin{defn}[Balance]\label{def:balance}
Let $\nss$ be a $k$-step compact nilspace. For each $n\in \mb{N}$ fix a metric $d_n$ on the space $\mc{P}(\cu^n(\nss))$. Let $\ns$ be a  compact nilspace, and let $\phi:\ns\to \nss$ be a continuous morphism. Then for $b>0$ we say that $\phi$ is \emph{$b$-balanced} if for every $n\leq 1/b$ we have $d_n\big(\mu_{\cu^n(\ns)}\co(\phi^{\db{n}})^{-1}, \mu_{\cu^n(\nss)}\big)\leq b$. A nilspace polynomial $F\co\phi$ is \emph{$b$-balanced} if the morphism $\phi$ is $b$-balanced.\vspace{-0.2cm}
\end{defn}
\noindent The balance property is an approximate form of multidimensional equidistribution: the image of $\phi^{\db{n}}$, $n\in [1/b]$, tends toward being equidistributed in $\cu^n(\nss)$ as $b$ decreases. This property is useful in problems involving averages of functions over certain configurations. It appeared in \cite{Szegedy:HFA}, and is related to a property of approximate \emph{irrationality} from \cite{GTarith}. In fact, from results in the latter paper it follows that, for nilsequences, high irrationality implies $b$-balance for small $b$ (see \cite[Theorem 3.6]{GTarith}, or \cite[Theorem 4.1]{CS}). 
\begin{proof}[Proof of Theorem \ref{thm:reglem-intro}]
We begin by noting that it suffices to prove the result for \textsc{cfr} coset nilspaces. Indeed, if $\ns$ is an inverse limit of such nilspaces, then the preimages of the Borel $\sigma$-algebras on these spaces under the limit maps form an increasing sequence of $\sigma$-algebras $\mc{B}_i$ on $\ns$ such that $\bigvee_{i\in \mb{N}}\mc{B}_i=_{\mu_{\ns}} \mc{B}_{\ns}$, the Borel $\sigma$-algebra on $\ns$. By standard results $\mb{E}(f|\mc{B}_i)\to f$ in $L^1$ as $i\to\infty$. This implies (using \cite[Lemma 2.17]{CScouplings}) that given any $\epsilon>0$, there is a limit map $\psi:\ns\to\ns'$, i.e.\ a continuous fibration onto a \textsc{cfr} coset nilspace $\ns'$, and a 1-bounded Borel function $f':\ns'\to\mb{C}$, such that $h:= f-f'\co\psi$ satisfies $\|h\|_{L^1} \leq \epsilon/2$. Let $f'=f'_s+f'_e+f'_r$ be the decomposition for $f'$ applied with initial parameter $\epsilon/2$ and with $\mc{D}'(\epsilon,m):=\mc{D}(2\epsilon,m)$, and let $f_s=f'_s\co\psi$, $f_e= h + f_e'\co \psi$, $f_r=f'_r\co\psi$. We have (using that $\psi$ is a Haar-measure-preserving morphism \cite[Corollary 2.2.7]{Cand:Notes2}) that $f=f_s+f_e+f_r$ is a valid decomposition for $\epsilon$, $\mc{D}$.

To prove the theorem for \textsc{cfr} coset nilspaces, we argue by contradiction. Suppose that the theorem fails for some $\epsilon>0$. This means that there is a sequence of functions $(f_i)_{i\in \mb{N}}$ where $f_i:\ns_i\to \mb{C}$ is Borel measurable on a compact coset nilspace $\ns_i$ with $|f_i|\leq 1$, such that $f_i$ does not satisfy the statement with $\epsilon$ and $N=i$. Let $\omega$ be a non-principal ultrafilter on $\mb{N}$ and let $\mf{X}$ be the ultraproduct $\prod_{i\to \omega} \ns_i$ equipped with the Loeb probability measure $\lambda'$ on $\mc{L}_{\mf{X}}$. Let $f:\mf{X}\to \mb{C}$ be the Loeb measurable function $\lim_\omega f_i$, and let $\mc{B}_0$ be the separable sub-$\sigma$-algebra of $\mc{L}_{\mf{X}}$ generated by $f$.

By Proposition \ref{prop:sepcc} there is a $\sigma$-algebra $\mc{B}'\subset \mc{L}_{\mf{X}}$ including $\mc{B}_0$ such that the probability space $\varOmega'=(\mf{X},\mc{B}',\lambda')$ is separable, and such that the sequence of measures $\mu^{\db{n}}$ on $(\mf{X}^{\db{n}},{\mc{B}'}^{\db{n}})$ form a cubic coupling. By \cite[(17.44), iv)]{Ke}, the measure algebra of $\varOmega'$ is isomorphic to the measure algebra of a Borel probability space $\varOmega=(\Omega,\mc{B},\lambda)$. By \cite[343B(vi)]{Fremlin3} (using \cite[211L(a)-(c)]{Fremlin2} and \cite[324K(b)]{Fremlin3}) there is a mod 0 isomorphism $\theta:\Omega'\to \Omega$ realizing this measure-algebra isomorphism. Moreover, by \cite[Proposition A.11]{CScouplings} the images of the measures $\mu^{\db{n}}$ under the maps $\theta^{\db{n}}$ form a cubic coupling on $\varOmega$. From now on we identify $f$ and $f\co\theta^{-1}$, so we view $f$ as a function on $\Omega$.

Let $\mc{F}_k$ be the $k$-th Fourier $\sigma$-algebra on $\Omega$ (see \cite[Definition 3.18]{CScouplings}). Then we have $f=f_s+f_r$, where $f_s=\mb{E}(f|\mc{F}_k)$, and $f_r=f-\mb{E}(f|\mc{F}_k)$ satisfies $\|f_r\|_{U^{k+1}}=0$. We now apply the structure theorem for cubic couplings \cite[Theorem 4.2]{CScouplings}. More precisely, applying this theorem to the above cubic coupling $\big(\varOmega,(\mu^{\db{n}})_{n\geq 0}\big)$, we obtain a $k$-step compact nilspace $\nss$, and a measurable map $\gamma_k:\Omega\to\nss$ such that $\gamma_k^{\db{n}}$ takes $\mu^{\db{n}}$ to the Haar measure $\mu_{\cu^n(\nss)}$ for each $n\geq 0$. Moreover, this nilspace $\nss$ is related to $\mc{F}_k$ in the sense that, letting $\mc{B}_{\nss}$ denote the Borel $\sigma$-algebra on $\nss$, we have that the $\sigma$-algebra $\gamma_k^{-1}(\mc{B}_{\nss})$ equals $\mc{F}_k$ modulo null sets (see \cite[Lemma 3.42]{CScouplings}). Then by \cite[Lemma 2.17]{CScouplings}  there is a Borel function $g:\nss\to \mb{C}$ such that $f_s=_\lambda g\co \gamma_k$.

By \cite[Theorem 2.7.3]{Cand:Notes2}, the nilspace $\nss$ is an inverse limit of $k$-step \textsc{cfr} nilspaces $\nss_j$, $j\in \mb{N}$, where the limit maps $\psi_j:\nss\to \nss_j$ are continuous fibrations. Let $\mc{Y}_j$ denote the $\sigma$-algebra on $\nss$ generated by $\psi_j$. Arguing as in the first paragraph of the proof, there is $j\in\mb{N}$ such that $g_j:=\mb{E}(g|\mc{Y}_j)$ satisfies $\|g-g_j\|_1\leq\epsilon/3$. For this $j$ let $\gamma=\psi_j\co\gamma_k:\Omega\to \nss_j$. As fibrations take cube sets onto cube sets in a measure-preserving way, the map $\gamma$ has the same measure-preserving properties as $\gamma_k$. Furthermore, by Lusin's theorem  combined with \cite[Theorem 1]{Georg}, there is a continuous function $h:\nss_j\to\mb{C}$ with $|h|\leq 1$ and with finite Lipschitz constant $C$ such that $\|g_j-h\|_{L^1(\nss)} \leq\epsilon/3$. Let $q= h\co \gamma:\Omega\to\mb{C}$. The measure-preserving properties of $\gamma_k$ and $\psi_j$ imply that $\|f_s-q\|_{L^1(\Omega)}= \|g-h\co\psi_j\|_{L^1(\nss)}\leq 2\epsilon/3$. Let $f_e=f_s-q=f-q-f_r$. 

Next, we show that there are continuous morphisms $\phi_i:\ns_i\to\nss_j$, $i\in \mb{N}$, such that $\gamma=_\lambda \lim_\omega \phi_i$. Note that since $\gamma$ is $\mc{L}_{\mf{X}}$-measurable, by \cite[Corollary 5.1]{Ross} it has a \emph{lifting}, i.e.\ there are Borel maps $g_i:\ns_i\to\nss_j$, $i\in \mb{N}$ such that $\gamma=_\lambda \lim_\omega g_i$. This together with the measure-preserving property of $\gamma^{\db{k+1}}$ implies that the preimage of $\cu^{k+1}(\nss_j)$ under $( \lim_\omega g_i)^{\db{k+1}}$ has $\mu^{\db{k+1}}$-probability 1. For each $i$ let $\delta_i=\inf\{t: g_i\textrm{ is a }(t,1)\textrm{-quasimorphism}\}\in [0,1]$. Then $\lim_\omega \delta_i=0$. Indeed, otherwise for some $\delta>0$ the set $S_1=\{i\in \mb{N}: g_i\textrm{ is not a }(\delta,1)\textrm{-quasimorphism}\}$ is in $\omega$. Then for each $i\in S_1$ there is a Borel set $B_i\subset\cu^{k+1}(\ns_i)$ of measure at least $\delta$ such that for every $\q\in B_i$ the image $g_i\co \q$ is $\delta$-separated from cubes, that is for every $\q'\in \cu^{k+1}(\nss_j)$ we have $\max_{v\in \db{k+1}} d_{\nss_j}(g_i\co \q(v),\q'(v))\geq \delta$. Since $S_1\in\omega$, we can take $B=\prod_{i\to \omega} B_i\subset \Omega$, and we have $\mu^{\db{k+1}}(B)\geq \delta$. Then, for every $\q\in B$ the composition $(\lim_\omega g_i)\co\q$ is also $\delta$-separated from cubes, so it cannot be in $\cu^{k+1}(\nss_j)$. This contradicts the above fact that $(\lim_\omega g_i)^{\db{k+1}}$ maps almost every $\q\in\cu^{k+1}(\Omega)$ into $\cu^{k+1}(\nss_j)$, so we indeed have $\lim_\omega \delta_i=0$. Hence there is a sequence $(\delta_i'>0)_{i\in \mb{N}}$ with $\lim_\omega \delta_i'=0$ such that $g_i$ is a $(\delta_i',1)$-quasimorphism for each $i$. Theorem \ref{thm:rigid} implies that for each $i$ there is a continuous morphism $\phi_i:\ns_i\to \nss_j$ such that $\mu_{\ns_i}(\{x\in\ns_i: \phi_i(x)\approx_{\epsilon_i} g_i(x)\})\geq 1-\epsilon_i$, where $\lim_\omega \epsilon_i= 0$. Hence $\lim_\omega g_i=_\lambda \lim_\omega \phi_i$, as required. Indeed, otherwise we have $\lambda( \lim_\omega g_i\neq \lim_\omega \phi_i)>0$, which implies (using monotonicity of $\lambda$) that $\lambda( \lim_\omega g_i\approx_\eta \lim_\omega \phi_i)<1-\eta$ for some $\eta>0$. But this event $\lim_\omega g_i\approx_\eta \lim_\omega \phi_i$ is $\big\{(x_i)\in \Omega:\{i: g_i(x_i)\approx_\eta \phi_i(x_i)\}\in \omega\big\}$, and this includes the set $\prod_{i\to\omega} \big\{x_i\in \ns_i: g_i(x_i)\approx_{\epsilon_i} \phi_i(x_i)\}$ (using that $\epsilon_i<\eta$ for a cofinite set of integers $i$); but the latter set has $\lambda$-measure 1, since $\mu_{\ns_i}(\{x\in\ns_i: \phi_i(x)\approx_{\epsilon_i} g_i(x)\})\geq 1-\epsilon_i$, and this contradicts that $\eta>0$.

There is a sequence $(b_i>0)_{i\in \mb{N}}$ such that $\phi_i$ is $b_i$-balanced for all $i$ and $\lim_\omega b_i=0$. Indeed, otherwise some $b>0$, $S'_2\in\omega$ satisfy that $\forall\, i\in S'_2$, $\phi_i$ is not $b$-balanced. Then there is $S_2\subset S'_2$ with $S_2\in \omega$, and $n\in [1/b]$, with $d_n\big(\mu_{\cu^n(\ns_i)}\co(\phi_i^{\db{n}})^{-1}, \mu_{\cu^n(\nss_j)}\big)\geq b$ for all $i\in S_2$.  As $\gamma^{\db{n}}$ is measure-preserving, we have $\lim_\omega d_n\big(\mu_{\cu^n(\ns_i)}\co(\phi_i^{\db{n}})^{-1}, \mu_{\cu^n(\nss_j)}\big)= \lim_\omega d_n\big(\mu_{\cu^n(\ns_i)}\co(\phi_i^{\db{n}})^{-1}, \mu^{\db{n}}\co(\gamma^{\db{n}})^{-1}\big)=0$ (using Lemma \ref{lem:Loebcomm}), a contradiction.

For each $i$ let $f_{s,i}=h\co \phi_i$, and apply \cite[Corollary 5.1]{Ross} again to obtain a sequence of Borel functions $(f_{r,i}:\ns_i\to \mb{C})_{i\in \mb{N}}$ such that $\lim_\omega f_{r,i}=_\lambda f_r$. Let $f_{e,i}=f_i-f_{s,i}-f_{r,i}$. Since $\lim_\omega g_i=_\lambda\lim_\omega \phi_i$, we have $\lim_\omega f_{s,i}=_\lambda q$, whence $\lim_\omega f_{e,i}=_\lambda f_e$. We also have $\lim_\omega \|f_{r,i}\|_{U^{k+1}}=\|f_r\|_{U^{k+1}}=0$. Since $q$ and $f_e$ are both $\mc{F}_k$-measurable, we have $\langle f_r,q\rangle$ and $\langle f_r,f_e\rangle$ both 0, and therefore $\lim_\omega \langle f_{r,i},f_{s,i}\rangle=\langle f_r,q\rangle=0$ and $\lim_\omega \langle f_{r,i},f_{e,i} \rangle= \langle f_r,f_e \rangle=0$.
Let $m$ be the maximum of $C$ and the complexity of $\nss_j$. Combining the properties in this paragraph and the previous one, we deduce that there is a set $S\in \omega$ such that for every $i\in S$ the decomposition $f_i=f_{s,i}+f_{r,i}+f_{e,i}$ satisfies the properties in the theorem with this value of $m$, the initial $\epsilon$, and the corresponding value $\mc{D}(\epsilon,m)$. This gives a contradiction for $i\in S$ with $i\geq m$.
\end{proof}
\noindent We deduce the following inverse theorem, which clearly implies Theorem \ref{thm:inverse-intro}.
\begin{theorem}\label{thm:inverse}
Let $k\in \mb{N}$, and let $b:\mb{R}_{>0}\to \mb{R}_{>0}$ be an arbitrary function. For every $\delta\in (0,1]$ there is $M>0$ such that for every compact nilspace $\ns$ that is an inverse limit of \textsc{cfr} coset nilspaces, and every 1-bounded Borel function $f:\ns\to \mb{C}$ such that $\|f\|_{U^{k+1}}\geq \delta$, for some $m\leq M$ there is a $b(m)$-balanced 1-bounded nilspace-polynomial $F\co\phi$ of degree $k$ and complexity at most $m$ such that $\langle f, F\co\phi\rangle \geq \delta^{2^{k+1}}/2$.
\end{theorem}
\begin{proof}
We apply Theorem \ref{thm:reglem-intro} with $\epsilon=\epsilon(\delta)>0$ and $\mc{D}$ to be fixed later. By property $(ii)$ in the theorem and the fact that $|f_s|\leq 1$, we have $|\langle f_e,f_s\rangle|\leq \epsilon$, and by property $(iii)$ we have $|\langle f_r,f_s\rangle|\leq \mc{D}(\epsilon,m)$. Therefore, taking the inner product of $f_s$ with each side of the decomposition $f=f_s+f_e+f_r$, we obtain $\langle f,f_s\rangle \geq \langle f_s,f_s\rangle- \epsilon - \mc{D}(\epsilon,m)$.

We also have $\|f_e\|_{L^1}\leq \epsilon$ and $|f_e|\leq 3$, whence $\|f_e\|_{U^{k+1}}\leq (3^{2^{k+1}-2}\epsilon^2)^{1/2^{k+1}}\leq 3 \epsilon^{1/2^k}$. Combining this with the above decomposition of $f$ and the bound $\|f_r\|_{U^{k+1}}\leq \mc{D}(\epsilon,m)$, we deduce that $\|f_s\|_{U^{k+1}}\geq \delta - 3 \epsilon^{1/2^k} - \mc{D}(\epsilon,m)$. This together with $|f_s|\leq 1$ implies that $\langle f_s,f_s\rangle = \|f_s\|_{L^2}^2\geq\|f_s\|_{U^{k+1}}^{2^{k+1}}\geq (\delta - 3 \epsilon^{1/2^k} - \mc{D}(\epsilon,m))^{2^{k+1}}$.

We now fix $\epsilon=\big(\frac{\delta}{3}(1-(\frac{5}{6})^{1/2^{k+1}})\big)^{2^k}$, and choose $\mc{D}$ so that the following hold: firstly, so that $\mc{D}(\epsilon,m)\leq b(m)$; secondly, so that by the last inequality in the previous paragraph we have $\langle f_s,f_s\rangle\geq  2\delta^{2^{k+1}}/3$; finally, so that $\epsilon +\mc{D}(\epsilon,m)\leq \delta^{2^{k+1}}/6$, which implies, by the last inequality in the first paragraph, that $\langle f,f_s\rangle \geq \delta^{2^{k+1}}/2$. We can then let $M$ be the number $N$ given by Theorem \ref{thm:reglem-intro} for this choice of $\epsilon$ and $\mc{D}$.
\end{proof}

\section{The case of simple abelian groups}\label{sec:Zp}
\noindent In this final section we use Theorem \ref{thm:reglem-intro} to prove Theorem \ref{thm:inverseZp-intro}. \enlargethispage{0.6cm} 

Recall that Definition \ref{def:balance} presupposes that for each $n$ a metric has been fixed on the space $\mc{P}(\cu^n(\ns))$ of Borel probabilities on $\cu^n(\ns)$ (equipped with the weak topology). 
For the proof of Theorem \ref{thm:inverseZp-intro} it is convenient to fix the  metrics in a process by induction on the step $k$ of $\ns$ as follows: having already defined a metric $d_{n,k-1}$ on $\mc{P}(\cu^n(\ns_{k-1}))$, we first let $d_{n,k}'$ be a metric on $\mc{P}(\cu^n(\ns))$ defined the standard way (see \cite[Theorem (17.19)]{Ke}), and then we define $d_{n,k}$ for $\mu,\nu\in \mc{P}(\cu^n(\ns))$ by
\begin{equation}\label{eq:metrize-k}
d_{n,k}(\mu,\nu) = d_{n,k}'(\mu,\nu)+d_{n,k-1}\big(\mu\co(\pi_{k-1}^{\db{n}})^{-1} ,\nu\co (\pi_{k-1}^{\db{n}})^{-1} \big).\vspace{-0.1cm}
\end{equation}
This construction is convenient for the proof because if $\phi$ is $b$-balanced relative to the metrics $d_{n,k}$, then $\pi_{k-1}\co \phi$ is automatically $b$-balanced relative to the metrics $d_{n,k-1}$. For the remainder of this section, we  suppose that we have fixed what we call a \emph{factor-consistent metrization for cubic measures} on \textsc{cfr} nilspaces, by which we mean the result of the following process: first we fix a sequence of metrics $d_{n,1}$ on $\mc{P}(\cu^n(\ns))$ ($n\geq 0$) for each $1$-step \textsc{cfr} nilspace $\ns$, then we fix metrics $d_{n,2}$ on $\mc{P}(\cu^n(\ns))$ for each $2$-step \textsc{cfr} nilspace $\ns$ using \eqref{eq:metrize-k} as above, and so on for increasing $k$.

In the proof of Theorem \ref{thm:inverseZp-intro}, a key ingredient is the following result, which ensures that the morphism that we obtain from Theorem \ref{thm:inverse} takes values in a toral nilspace. 
\begin{theorem}\label{thm:Zp-toral}
Fix any complexity notion and any factor-consistent metrization for cubic measures on \textsc{cfr} nilspaces. Then for every $M>0$ there exist $b>0$ and $p_0>0$ with the following property. Let $\nss$ be a $k$-step \textsc{cfr} nilspace of complexity at most $M$, and let $\phi:\mb{Z}_p\to\nss$ be a $b$-balanced morphism for a prime $p>p_0$. Then $\nss$ is toral. 
\end{theorem}
\noindent This section is mostly devoted to the proof of this result. The proof of Theorem \ref{thm:inverseZp-intro} is a simple combination of Theorems \ref{thm:Zp-toral} and \ref{thm:inverse}, and is given at the end of this section. 

Recall that a nilspace $\ns$ can be equipped with a filtration of  \emph{translation groups} $\tran_i(\ns)$, $i\geq 0$ (see \cite[Definition 3.2.27]{Cand:Notes1}), and that for \textsc{cfr} nilspaces these translation groups are Lie groups (see \cite[Theorem 2.9.10]{Cand:Notes2}).

In the proof of Theorem \ref{thm:Zp-toral}, we shall argue by induction on $k$. This will enable us to assume that $\nss_{k-1}$ is toral, and we shall then use the following characterization of such nilspaces, which will be very convenient for the rest of the argument.\vspace{-0.1cm}
\begin{theorem}\label{thm:cosetnilspace}
Let $\ns$ be a $k$-step \textsc{cfr} nilspace such that the factor $\ns_{k-1}$ is toral. Let $G$ denote the Lie group $\tran(\ns)$, let $G_\bullet$ denote the degree-$k$ filtration $(\tran_i(\ns))_{i\geq 0}$, and for an arbitrary fixed $x\in \ns$ let $\Gamma=\stab_G(x)$. Then $\ns$ is isomorphic as a compact nilspace to the coset nilspace $(G/\Gamma,G_\bullet)$.\vspace{-0.1cm}
\end{theorem}
\noindent This theorem tells us essentially that such a nilspace $\ns$ must be a \textsc{cfr} \emph{coset} nilspace, but it also gives us groups $G,\Gamma$ and a  filtration $G_\bullet$ with which we can represent $\ns$. The proof is an adaptation of \cite[Theorem 2.9.17]{Cand:Notes2}; see Theorem \ref{thm:cosetnilspace-app} in Appendix \ref{app:torality}.

Given Theorem \ref{thm:cosetnilspace}, for the proof of Theorem \ref{thm:Zp-toral} we can focus on coset nilspaces. This is useful thanks to the following description of morphisms from $\mb{Z}_p$ into such nilspaces. 
\begin{proposition}\label{prop:comdiag}
Let $\ns=(G/\Gamma,G_\bullet)$ be a coset nilspace. For a positive integer $N$ let  $\phi:\mb{Z}_N\to G/\Gamma$ be a morphism \textup{(}relative to the standard degree-1 cube structure on $\mb{Z}_N$\textup{)}. Then for every homomorphism $\beta:\mb{Z}\to \mb{Z}_N$ there is a polynomial map $g\in\poly(\mb{Z},G_\bullet)$ such that $\phi\co \beta = \pi_{\Gamma}\co g$.
\end{proposition}
\noindent The proof, adapting an argument from \cite{Szegedy:HFA}, is given at the end of Appendix \ref{app:torality}.

In the proof of Theorem \ref{thm:Zp-toral}, we use the following lemma in the inductive step.

\begin{lemma}\label{lem:conncomp}
Let $\ns$ be a \textsc{cfr} coset nilspace $(G/\Gamma,G_\bullet)$, and let $\nss$ be the coset nilspace $(G/(G^0\,\Gamma),G_\bullet)$ where $G^0$ is the identity component of $G$. Then the quotient map $q:G/\Gamma \to G/(G^0\,\Gamma)$ is a morphism of compact nilspaces, and $\nss$ is in bijection with the set of connected components of $\ns$. In particular $\nss$ is a finite \textup{(}discrete\textup{) }nilspace.
\end{lemma}
\begin{proof}
It is clear that $q$ is a (continuous) morphism, because any cube $\q\in \cu^n(\ns)$ lifts to a cube $\tilde\q\in\cu^n(G_\bullet)$, i.e.\ we have $\q=\tilde\q\Gamma^{\db{n}}$ (by definition of the coset nilspace structure), so $q\co \q = \tilde\q (G^0\, \Gamma)^{\db{n}}$ is indeed a cube on $\nss$.

We claim that the quotient map $\pi_\Gamma:G\to G/\Gamma$ induces a bijection from the set of cosets of $G^0\Gamma$ (i.e.\ the set $\nss$) to the set of connected components of $G/\Gamma$. First note that the image under $\pi_\Gamma$ of any coset of $G^0\Gamma$ is open, because $G^0$ is open (as $G$ is a Lie group) and $\pi_\Gamma$ is an open map. Since these images cover the compact set $G/\Gamma$, and clearly two distinct cosets of $G^0\Gamma$ are mapped to disjoint such images by $\pi_\Gamma$, these images form a finite partition of $G/\Gamma$. Moreover, the image of every coset $gG^0\Gamma$ is connected in $G/\Gamma$ (indeed for any points $gg_1\gamma_1,gg_2\gamma_2$ in this coset there are paths from $gg_i\gamma_i$ to $g\gamma_i$ via $G^0$ for $i=1,2$, and then $g\gamma_1$, $g\gamma_2$ are identified in the quotient), so each such image is included in one of the components of $G/\Gamma$, and therefore must be the whole component (otherwise this component would be a disjoint union of at least two such images, which are open sets, contradicting the connectedness of the component). This shows that each component of $G/\Gamma$ is an image under $\pi_{\Gamma}$ of a unique coset of $G^0\Gamma$, which proves our claim.
\end{proof}
We need two more lemmas before we can prove Theorem \ref{thm:Zp-toral}.
\begin{lemma}\label{lem:cubemapmorph}
Let $\nss$ be a coset nilspace, let $N\in\mb{N}$ and let $\phi:\mb{Z}_N\to\nss$ be a morphism. Then for each $k\in\mb{N}$ the map $\phi^{\db{k}}:\q\mapsto \phi\co\q$ is a nilspace morphism $\cu^k(\mb{Z}_N)\to \cu^k(\nss)$.
\end{lemma}
\begin{proof}
We are assuming that $\nss$ is the coset space $G/\Gamma$, for some filtered group $(G,G_\bullet)$ and $\Gamma\leq G$, and that $\cu^k(\nss)=\{\q\Gamma^{\db{k}}:\q\in \cu^k(G_\bullet)\}$. We view the abelian group $\cu^k(\mb{Z}_N)$ as a nilspace by equipping it with the standard cubes, and we view $\cu^k(\nss)$ as the coset nilspace $\wt G/\wt \Gamma$ where $\wt G$, $\wt \Gamma$ denote the group $\cu^k(G_\bullet)$ and subgroup $\cu^k(\Gamma_\bullet)$ respectively (with $\Gamma_i:=\Gamma\cap G_i$), and where $\wt G$ is equipped with the filtration $\wt G_\bullet = \big(G_i^{\db{k}}\cap \cu^k(G_\bullet)\big)_{i\geq 0}$. By Proposition \ref{prop:comdiag} there is a polynomial map $g\in\poly(\mb{Z},G_\bullet)$ such that, identifying $\mb{Z}_N$ with the set of integers $[0,N-1]$ with addition mod $N$, we have $\phi(n)=g(n)\Gamma$ for all $n$ (in particular $g$ is $N$-periodic mod $\Gamma$). Define \vspace{-0.1cm}
\begin{equation}\label{eq:kdimg}
g^{(k)}: \mb{Z}^{k+1}\to\wt G ,\;\; \mf{n}=(n_0,n_1,\dots,n_k)\mapsto \big(g(n_0+v\cdot (n_1,\dots,n_k))\big)_{v\in \db{k}}. \vspace{-0.1cm}
\end{equation}
The group isomorphism $\theta:\mb{Z}_N^{k+1}\to\cu^k(\mb{Z}_N)$, $\mf{n}\mapsto \big(n_0+v\cdot (n_1,\dots,n_k)\mod N\big)_{v\in\db{k}}$ is a nilspace isomorphism. Hence $\phi^{\db{k}}$ is a morphism if and only if 
the map $\mf{n}\mapsto g^{(k)}(\mf{n})\Gamma^{\db{k}}$ is a morphism $\mb{Z}_N^{k+1}\to \cu^k(\nss)$ (since the latter map is $\phi^{\db{k}}\co \theta$). 
Recall that the morphisms between two group nilspaces are the polynomial maps between the filtered groups \cite[Theorem 2.2.14]{Cand:Notes1}. Hence it suffices to prove that $g^{(k)}\in \poly(\mb{Z}^{k+1},\wt G_\bullet)$, as then $g^{(k)}$ is a morphism into $\wt G$ and then $g^{(k)}(\mf{n})\Gamma^{\db{k}}$ is a morphism as required.\\
\indent By Lemma \ref{lem:Taylor}, there is a unique expression $g(n)=g_0g_1^n\cdots g_k^{\binom{n}{k}}$, where $g_i\in G_i$. Substituting this expression into \eqref{eq:kdimg} and expanding, we see that $g^{(k)}(\mf{n})$ is a pointwise product of maps $h_j:\mb{Z}^{k+1}\to\wt G$, $j\in [0,k]$, of the form $h_j(\mf{n})= \Big(g_j^{\binom{n_0+v\cdot(n_1,\ldots,n_k)}{j}}\Big)_{v\in \db{k}}$. By Leibman's theorem \cite{Leib}, polynomial maps form a group under pointwise multiplication, so it suffices to show that for every $j\in [0,k]$ we have $h_j\in \poly(\mb{Z}^{k+1},\wt G_\bullet)$. We have $\binom{n_0+v\cdot(n_1,\ldots,n_k)}{j}=\sum_{\mf{i}=(i_0,\ldots,i_k)\in \mb{Z}_{\geq 0}^{k+1}, |\mf{i}|=j}\binom{n_0}{i_0}\binom{v_1 n_1}{i_1}\cdots \binom{v_k n_k}{i_k}$, by the identity of Chu--Vandermonde. Letting $\mf{i}'=(i_1,\ldots,i_k)$ be the restriction of $\mf{i}$ to its last $k$ coordinates, we note that $\binom{n_0}{i_0}\binom{v_1 n_1}{i_1}\cdots \binom{v_k n_k}{i_k}$ gives a non-zero contribution to the last sum above only if $\supp(\mf{i}')\subset \supp(v)$. 
We deduce that
$h_j(\mf{n}) = \prod_{ \mf{i},\, |\mf{i}|=j}  g_{\mf{i}}^{\tbinom{n_0}{i_0}\cdots \tbinom{n_k}{i_k}}$, where $g_{\mf{i}}$ is the element of $G^{\db{k}}$ with $g_{\mf{i}}(v)=g_j$ if $\supp(v)\supset \supp(\mf{i}')$, and $g_{\mf{i}}(v)=\id_G$ otherwise. Now observe that, since $|\supp(\mf{i}')|\leq j$, the set $\{v:\supp(v)\supset \supp(\mf{i}')\}$ is a face of codimension at most $j$ in $\db{k}$. Since $g_j\in G_j$, it follows that $g_{\mf{i}}\in \wt G_j$. 

We have shown that $h_j$ is a pointwise product of maps of the form $\mf{n}\mapsto g_{\mf{i}}^{\binom{\mf{n}}{\mf{i}}}$, where $\binom{\mf{n}}{\mf{i}}=\binom{n_0}{i_0}\binom{n_1}{i_1}\cdots \binom{n_k}{i_k}$. It is known that these maps are polynomial (see the proof of \cite[Lemma 6.7]{GTOrb}). This proves that $g^{(k)} \in\poly(\mb{Z}^{k+1},\wt G_\bullet)$, and the result follows.
\end{proof} 

\begin{remark}
In Lemma \ref{lem:cubemapmorph} we equipped the cube set $\cu^k(\nss)$ itself with a natural nilspace structure, but note that this was enabled by the specific \emph{coset}-nilspace nature of $\nss$. There is in fact a \emph{cube}space structure that one can define on $\cu^k(\ns)$ for a \emph{general} nilspace $\ns$: given a map $\q:\db{m}\to\cu^k(\ns)$, $v\mapsto \q(v)$ (where $\q(v)$ is itself a cube $w\mapsto \q(v)(w)$ in $\cu^k(\ns)$), we declare $\q$ to be an $m$-cube on $\cu^k(\ns)$ if for every $w\in \db{k}$, the map $\db{m}\to\ns$, $v\mapsto \q(v)(w)$ is in $\cu^m(\ns)$. It seems to be an interesting question whether this cubespace structure satisfies the completion axiom and thus defines a nilspace structure. The answer is affirmative when $\ns$ is a coset nilspace, because it can be checked that in this case this structure is equivalent to the one used on $\cu^k(\nss)$ above. This fact can be used to give an alternative proof of Lemma \ref{lem:cubemapmorph}.
\end{remark}

\begin{lemma}\label{lem:finabcase}
Let $\ab_1$, $\ab_2$ be finite abelian groups with coprime orders, and let $\ell\in\mb{N}$. Then every morphism $\mc{D}_1(\ab_1)\to \mc{D}_\ell(\ab_2)$ is constant.
\end{lemma}
\begin{proof}
We argue by induction on $\ell$. For $\ell=1$, note that a morphism $\phi:\mc{D}_1(\ab_1)\mapsto\mc{D}_1(\ab_2)$ satisfies $\Delta_s\Delta_t\phi(x)=0$ for every $s,t,x\in\ab_1$ (see \cite[formula (2.9)]{Cand:Notes1}), which means that $\phi$ is an affine homomorphism $\ab_1\to \ab_2$, so the map $\psi:x\mapsto \phi(x)-\phi(0)$ is a homomorphism. By standard group theory, the order $|\psi(\ab_1)|$ divides both $|\ab_1|$ and $|\ab_2|$, so we must have $|\psi(\ab_1)|=1$, so $\phi$ is constant. For $\ell>1$, note that for every morphism $\phi:\mc{D}_1(\ab_1)\to \mc{D}_\ell(\ab_2)$, for every $t\in \ab_1$ the map $\Delta_t\phi: x\mapsto \phi(x+t)-\phi(x)$ is a morphism $\mc{D}_1(\ab_1)\to \mc{D}_{\ell-1}(\ab_2)$, so by induction $\Delta_t\phi$ is a constant function of $x$, for each $t$. Hence $\Delta_s\Delta_t\phi(x)=0$ for all $s,t,x\in\ab_1$. Arguing as for $\ell=1$, we deduce that $\phi$ is constant.
\end{proof}

We can now prove the characterization of balanced morphisms on $\mb{Z}_p$.

\begin{proof}[Proof of Theorem \ref{thm:Zp-toral}]
By Theorem \ref{thm:k-cube-connect} it suffices to show that $\cu^k(\nss)$ is connected. We prove this by induction on $k$. The base case $k=0$ is trivial.

Let $k\geq 1$, and suppose for a contradiction that $\cu^k(\nss)$ is disconnected.

We have that $\pi_{k-1}\co\phi$ is also $b$-balanced (by our choice of a factor-consistent metrization), so we can assume by induction that $\nss_{k-1}$ is toral. Hence $\nss$ is isomorphic to a compact coset nilspace $(G/\Gamma,G_\bullet)$, by Theorem \ref{thm:cosetnilspace}. Letting $\wt G=\cu^k(G_\bullet)$ with the filtration $\wt{G}_\bullet = \big(G_j^{\db{k}}\cap \cu^k(G_\bullet)\big)_{j\geq 0}$, and $\wt{\Gamma} = \cu^k(\Gamma_\bullet)$, we have that $\cu^k(\nss)$ is homeomorphic to the compact coset space $\wt{G}/\wt{\Gamma}$, which we equip with the coset nilspace structure determined by $\wt{G}_\bullet$. By Lemma \ref{lem:cubemapmorph}, the map $\phi^{\db{k}}:\cu^k(\mb{Z}_p)\to\cu^k(\nss)$, $\q\mapsto \phi\co\q$ is a morphism. We apply Lemma \ref{lem:conncomp} to $\cu^k(\nss)$, and let $q:\wt{G}/\wt{\Gamma}\mapsto \wt{G}/(\wt{G}^0 \wt{\Gamma})$ be the resulting quotient morphism. Then $q \co \phi^{\db{k}}$ is a morphism from $\cu^k(\mb{Z}_p)$ to a discrete nilspace $\wt{Y}$ of finite cardinality equal to the number of connected components of $\cu^k(\nss)$. 

We claim that for $b$ sufficiently small (depending only on $M$), for every such component $C$ we have $\phi^{\db{k}}\big(\cu^k(\mb{Z}_p)\big)\cap C\neq \emptyset$. Indeed, by Lemma \ref{lem:compequim} the finitely many connected components of $\cu^k(\nss)$ all have equal Haar measure $\nu>0$. Hence, for any such component $C$, it follows from the Portmanteau Theorem \cite[(17.20)]{Ke} (using that $C$ is open) that the measure $\mu_{\cu^k(\mb{Z}_p)}\co(\phi^{\db{k}})^{-1}(C)$ is at least $\nu-o(1)_{b\to 0}$ (where $\mu_{\cu^k(\mb{Z}_p)}$ is the Haar measure on $\cu^k(\mb{Z}_p)$), so for $b$ sufficiently small this measure is positive, which proves our claim. This claim implies that $q \co \phi^{\db{k}}$ is surjective.

Now let $\wt{\nss}_i$ be the nilspace factor of $\wt{\nss}$ for the minimal $i\in [k]$ such that $\wt{\nss}_i$ is not the 1-point nilspace. In particular, it follows from minimality of $i$ that $\wt{\nss}_i$ is a finite abelian group $\ab$ with the degree-$i$ nilspace structure $\mc{D}_i(\ab)$. Since the factor map $\pi_i:\wt{\nss}\to \wt{\nss}_i$ is a surjective morphism, it follows that the map $\psi := \pi_i\co q \co \phi^{\db{k}}$ is a surjective morphism $\cu^k(\mb{Z}_p)\to \wt{\nss}_i$. For $p$ sufficiently large in terms of $M$, the orders $|\cu^k(\mb{Z}_p)|=p^{k+1}$ and $|\wt{\nss}_i|$ are coprime, so by Lemma \ref{lem:finabcase} the morphism $\psi$ must be constant, and therefore cannot be surjective, so we have a contradiction.
\end{proof}

Finally, having proved Theorem \ref{thm:Zp-toral}, we can prove the inverse theorem for $\mb{Z}_p$.

\begin{proof}[Proof of Theorem \ref{thm:inverseZp-intro}]
We first note that, having fixed an arbitrary complexity notion for \textsc{cfr} nilspaces $\nss$, there is a function $h:\mb{N}\to\mb{N}$ (which can be assumed to be increasing) such that if $\textrm{Comp}(\nss)\leq m$ then $\nss$ has at most $h(m)$ connected components. Now suppose that $\|f\|_{U^{k+1}(\mb{Z}_p)}\geq \delta$. We apply Theorem \ref{thm:inverse} with $\delta$, with a function $b$ to be specified later and with $\ns=\mb{Z}_p$. Let $M=M(k,\delta,b)>0$ be the resulting number and let $F\co \phi$ be the resulting nilspace polynomial, for an underlying \textsc{cfr} nilspace $\nss$ with $\textrm{Comp}(\nss)\leq m\leq M$, and with the morphism $\phi:\mb{Z}_p\to \ns$ being $b(m)$-balanced. If $p>h(m)$ and $b(m)$ is sufficiently small, then it follows by Theorem \ref{thm:Zp-toral}  that $\ns$ is toral. In particular, it is a connected nilmanifold, and by Proposition \ref{prop:comdiag} the nilspace polynomial is a $p$-periodic nilsequence as required. Thus, for $p>h(m)$ we obtain the conclusion of Theorem \ref{thm:inverseZp-intro} with $C_{k,\delta}=M$. For $p\leq h(m)$ we also obtain the conclusion, but for a simpler reason:  letting $\phi$ be the homomorphism embedding $\mb{Z}_p$ as a discrete subgroup of the circle group $\mb{R}/\mb{Z}$, and letting $F:\mb{R}/\mb{Z}\to \mb{C}$ be some  function with Lipschitz constant $O_p(1)$ that extends the function $f\co\phi^{-1}$ from $\phi(\mb{Z}_p)$ to all of $\mb{R}/\mb{Z}$, we then have $\langle f,F\co\phi\rangle=\|f\|_{L^2(\mb{Z}_p)}^2\geq \|f\|_{U^{k+1}(\mb{Z}_p)}^{2^{k+1}}\geq \delta^{2^{k+1}}$, and the conclusion of Theorem \ref{thm:inverseZp-intro} follows with constant $C_{k,\delta}$ still depending only on $k$ and $\delta$.
\end{proof}

\appendix

\section{Results from nilspace theory}\label{app:torality}
\noindent In this appendix our first and main aim is to prove Theorem \ref{thm:k-cube-connect}. We also gather some results from nilspace theory which are adaptations of results from previous works.

We begin with the following useful description of \textsc{cfr} $k$-step nilspaces whose $k-1$ factor is toral, which was stated as Theorem \ref{thm:cosetnilspace}.

\begin{theorem}\label{thm:cosetnilspace-app}
Let $\ns$ be a $k$-step \textsc{cfr} nilspace such that the factor $\ns_{k-1}$ is toral. Let $G$ denote the Lie group $\tran(\ns)$, let $G_\bullet$ denote the degree-$k$ filtration $(\tran_i(\ns))_{i\geq 0}$, and for an arbitrary fixed $x\in \ns$ let $\Gamma=\stab_G(x)$. Then $\ns$ is isomorphic as a compact nilspace to the coset space $G/\Gamma$ with cube sets $\cu^n(\ns)=(\cu^n(G_\bullet)\cdot \Gamma^{\db{n}})/\Gamma^{\db{n}}$, $n\geq 0$.
\end{theorem}
To prove this we adapt the proof of \cite[Theorem 2.9.17]{Cand:Notes2}.
\begin{proof}
Fix $x\in \ns$ and let $\Gamma=\stab_G(x)$.

We first claim that $\Gamma$ is discrete. Indeed, letting $h:\tran(\ns)\to\tran(\ns_{k-1})$ be the natural continuous homomorphism defined by $h(\alpha)(y)=\pi_{k-1}(\alpha(x))$ (see \cite[Lemma 2.9.3]{Cand:Notes2}), note that $h(\Gamma)$ is a subgroup of the stabilizer of $\pi_{k-1}(x)$ in $\tran(\ns_{k-1})$, and since $\ns_{k-1}$ is toral, this stabilizer is discrete (see the proof of \cite[Theorem 2.9.17]{Cand:Notes2}), so $h(\Gamma)$ is discrete. Then, since $h^{-1}(h(\Gamma))$ is a union of cosets of $\ker(h)$, it suffices to show that $\Gamma\cap\ker(h)$ is discrete. This follows from \cite[Lemma 2.9.9]{Cand:Notes2}, since no non-trivial element of $\tau(\ab_k)$  stabilizes $x$.\\
\indent By \cite[Corollary 2.9.12]{Cand:Notes2} the Lie group $\tran(\ns)^0$ acts transitively on the connected components of $\ns$, and since $\ns_{k-1}$ is toral, it follows that $\langle \tran(\ns)^0,\ab_k\rangle$ acts transitively on $\ns$. Indeed, if $x,y\in \ns$ are in different components, then there is $g'\in \tran(\ns_{k-1})^0$ such that $g'\pi_{k-1}(x)=\pi_{k-1}(y)$. Then there is $g\in \tran(\ns)^0$ such that $h(g)=g'$, and since $g$ is path-connected to the identity in $G$, it follows that $g x$ is in the same component as $x$. Moreover, by definition of $h$ we have $\pi_{k-1}(gx)=g'\pi_{k-1}(x)=\pi_{k-1}(y)$. There is therefore $z\in \ab_k$ such that $zg x=y$, which proves the claimed transitivity. Now since $G\supset \langle \tran(\ns)^0,\ab_k\rangle$, we have that $G$ also acts transitively on $\ns$, whence $\ns$ is homeomorphic to the coset space $G/\Gamma$ (see \cite[Ch.\ II, Theorem 3.2]{Helga}). In particular, since $\ns$ is compact, we have that $\Gamma$ is cocompact.

Recall from \cite[Definition 3.2.38]{Cand:Notes1} that two cubes $\q_1,\q_2\in \cu^n(\ns)$ are said to be translation equivalent if there is an element $\q\in \cu^n(G_\bullet)$ such that $\q_2(v)=\q(v)\cdot \q_1(v)$. We now show that $\cu^n(\ns)=\pi_{\Gamma}^{\db{n}}\big(\cu^n(G_\bullet)\big)$, i.e., that every cube on $\ns$ is translation equivalent to the constant $x$ cube. First we claim that for every cube $\q\in \cu^n(\ns)$ there is a cube $\q'\in \cu^n(\ns)$ that is translation equivalent to the constant $x$ cube and such that $\pi_{k-1}\co \q=\pi_{k-1}\co \q'$. Indeed, given $\q\in \cu^n(\ns)$, we have $\pi_{k-1}\co\q\in \cu^n(\ns_{k-1})$, and since $\ns$ is toral the latter cube is translation equivalent to the cube with constant value $x'=\pi_{k-1}(x)$, i.e.\ $\pi_{k-1}\co\q=\tilde\q\cdot x'$ for some cube $\tilde \q$ on the group $\tran(\ns_{k-1})^0$ with the filtration $\big(\tran_i(\ns_{k-1})^0\big)_{i\geq 0}$. By the unique factorization result for these cubes \cite[Lemma 2.2.5]{Cand:Notes1}, we have $\tilde\q={\tilde g_0}^{F_0}\cdots {\tilde g_{2^n-1}}^{F_{2^n-1}}$ where $\tilde g_j\in \tran_{\codim(F_j)}(\ns_{k-1})^0$. By \cite[Theorem 2.9.10 (ii)]{Cand:Notes2}, for each $j\in [0,2^n)$ there is $g_j\in \tran_{\codim(F_j)}(\ns)^0$ such that $h(g_j)=\tilde g_j$. 
Let $\q^*$ be the cube in $\cu^n(\tran(\ns)^0)$ defined by $\q^*={g_0}^{F_0}\cdots {g_{2^n-1}}^{F_{2^n-1}}$. Let $\q'=\q^*\cdot x$. This is in $\cu^n(\ns)$,  and is translation equivalent to the constant $x$ cube. By construction $\pi_{k-1}\co \q'$ $=\pi_{k-1}^{\db{n}}(\q^*\cdot x)= \big(\prod_j h(g_j)^{F_j}\big)\cdot x'= \big(\prod_j \tilde g_j^{F_j}\big)\cdot x'=\tilde \q \cdot x'=\pi_{k-1}\co \q$, as we claimed.

It follows from \cite[Theorem 3.2.19]{Cand:Notes1} and the definition of degree-$k$ bundles (in particular \cite[(3.5)]{Cand:Notes1}) that $\q-\q'\in \cu^n(\mc{D}_k(\ab_k))$. But then, using translations from $\tau(\ab_k)=\tran_k(\ns)$, we can correct $\q'$ further to obtain $\q$, thus showing that $\q$ is itself a translation cube with translations from $\tran(\ns)$. (Such a correction procedure has been used in previous arguments, see for instance the proof of \cite[Lemma 3.2.25]{Cand:Notes1}.)

We have thus shown that $\cu^n(\ns)\subset\pi_{\Gamma}^{\db{n}}\big(\cu^n(G_\bullet)\big)$. The opposite inclusion is clear, by definition of the groups $\tran_i(\ns)$.
\end{proof}

We can now prove Theorem \ref{thm:k-cube-connect}, which we restate here.
\begin{theorem}\label{thm:k-cube-connect-app}
Let $\ns$ be a $k$-step \textsc{cfr} nilspace. If $\cu^k(\ns)$ is connected, then $\ns$ is toral.
\end{theorem}

\begin{proof}
We argue by induction on $k$. For $k=1$ the statement is clear. For $k>1$, first note that $\cu^k(\ns_{k-1})$ is connected (by continuity of $\pi_{k-1}$), and so (since projection to a $k-1$ face of a $k$ cube is a continuous map) we have also that $\cu^{k-1}(\ns_{k-1})$ is connected, so by induction we have that $\ns_{k-1}$ is toral. Now suppose for a contradiction that $\ns$ is not toral. Then the last structure group $\ab_k$ must be a disconnected compact abelian Lie group. By quotienting out the torus factor of $\ab_k$ if necessary, we can assume that $\ns$ now has $k$-th structure group $\ab_k$ being a finite abelian group of cardinality greater than 1. We shall now deduce that $\cu^k(\ns)$ must be disconnected, a contradiction.

By Theorem \ref{thm:cosetnilspace-app} we have that $\ns$ is isomorphic to the coset nilspace $(G/\Gamma,G_\bullet)$ where $G=\tran(\ns)$ and $\Gamma=\stab_G(x)$ for some fixed point $x\in \ns$. 
Hence $\cu^k(\ns)=\cu^k(G_\bullet)/\Gamma^{\db{k}}$. Let $\sigma_k$ be the Gray code map on $G^{\db{k}}$\cite[Definition 2.2.22]{Cand:Notes1}, and recall that restricted to $\cu^k(G_\bullet)$ this map takes values in $G_k$ (see \cite[Proposition 2.2.25]{Cand:Notes2}) and that $G_k\cong\ab_k$ (see \cite[Lemma 3.2.37]{Cand:Notes1}). We know that shifting any value $\q(v)$ of a cube $\q\in \cu^k(G_\bullet)$ by any element of $\ab_k$ still gives a cube in $\cu^k(G_\bullet)$ (see \cite[Remark 3.2.12]{Cand:Notes1}). It follows that $\sigma_k$ maps $\cu^k(G_\bullet)$ onto $\ab_k$. On the other hand, the map $\sigma_k$ only takes the value $\id_G$ on $\Gamma^{\db{k}}$, since $\Gamma\cap G_k=\{\id_G\}$ (as the action of $G_k\cong\ab_k$ is free). Now let $C$ denote the identity component of $\cu^n(G_\bullet)$. It is standard that $C$ is normal in $\cu^n(G_\bullet)$. We also have $\sigma_k(C\cdot \Gamma^{\db{k}})=\{\id_G\}$. Indeed, since $\sigma_k$ is continuous and $\ab_k$ is discrete, for every element $c\cdot \gamma\in C\cdot \Gamma^{\db{k}}$ we have $\sigma_k(\gamma)=0$, and $c\cdot \gamma$ is in the same component as $\gamma$, so we must also have $\sigma_k(c\cdot \gamma)=0$. But then the product set $C\cdot \Gamma^{\db{k}}$ must be a \emph{proper} subgroup of $\cu^n(G_\bullet)$ (otherwise its image under $\sigma_k$ would be  $G_k$). Thus we have shown that $\cu^n(G_\bullet) / C\cdot \Gamma^{\db{k}}$ is not the one point space. Hence there are at least two disjoint cosets of $C\cdot \Gamma^{\db{k}}$ forming a cover of $\cu^n(G_\bullet)$. Since the latter group is a Lie group, $C$ is open, and therefore these covering cosets of $C\cdot \Gamma^{\db{k}}$ are open sets. But then the quotient map $q: \cu^n(G_\bullet)\to \cu^n(G_\bullet) / \Gamma^{\db{k}}$ (which is open) sends these cosets to disjoint open sets covering $\cu^k(G_\bullet)/\Gamma^{\db{k}}$, so $\cu^k(\ns)$ is disconnected.
\end{proof}

\noindent We add the following lemma concerning the Haar measures on cube sets.

\begin{lemma}\label{lem:compequim}
Let $\ns$ be a $k$-step \textsc{cfr} nilspace such that $\ns_{k-1}$ is toral. Then for every integer $n\geq 0$ the connected components of $\cu^n(\ns)$ have equal positive Haar measure.
\end{lemma}
\begin{proof}
Recall that $\cu^n(\ns)$ is a compact abelian bundle with base $\cu^n(\ns_{k-1})$, bundle projection $\pi:=\pi_{k-1}^{\db{n}}$, and structure group $\wt{\ab}_k:=\cu^n(\mc{D}_k(\ab_k))$, where $\ab_k$ is the $k$-th structure group of $\ns$ (see\cite[Lemma 2.2.12]{Cand:Notes2}). The Haar measure $\mu$ on $\cu^n(\ns)$ is invariant under the continuous action of $\wt{\ab}_k$, by construction (see \cite[Proposition 2.2.5]{Cand:Notes2}). Assuming that there is more than one component of $\cu^n(\ns)$, let $\q_1,\q_2$ be any points in distinct components $C_1$, $C_2$  respectively. Then, since $\ns_{k-1}$ is toral, by \cite[Theorem 2.9.17]{Cand:Notes2} there is a cube $\q\in \cu^n(\tran(\ns_{k-1})^0_\bullet)$ such that $\q\cdot \pi(\q_1)=\pi(\q_2)$. By \cite[Theorem 2.9.10]{Cand:Notes2} there is a cube $\wt{\q}\in \cu^n(\tran(\ns)^0_\bullet)$ such that $\pi(\wt{\q}\cdot \q_1)=\pi(\q_2)$. There is therefore $z\in \wt{\ab}_k$ such that $\wt{\q}\cdot \q_1 +\, z = \q_2$. Note that $\wt{\q}\cdot \q_1$ is still in $C_1$, since the map $\q_1\mapsto \wt{\q}\cdot \q_1$ is a composition of multiplications by face-group elements of the form $g^F$ where $F$ is a face in $\db{n}$ and $g$ is in the \emph{connected} Lie group $\tran_{\codim(F)}(\ns)^0$. Hence $(C_1 +z)\cap C_2$ is non-empty (containing $\q_2$), so $C_1 +z\subset C_2$ (since $C_1 +z$ is connected and $C_2$ is a maximal connected set), whence $\mu(C_1)=\mu(C_1 +z)\leq \mu(C_2)$. Similarly $\mu(C_2)\leq \mu(C_1)$.
\end{proof}
\noindent Next, we prove the properties of the $U^d$-seminorms from Definition \ref{def:Unorms}.
\begin{lemma}\label{lem:non-deg-app}
For every $k$-step compact nilspace $\ns$ and every $d\geq 2$, the function $f\mapsto \|f\|_{U^d}$ is a seminorm on $L^\infty(\ns)$.
\end{lemma}
\noindent The case of this lemma for compact abelian groups is given in several sources, all based essentially on the original argument of Gowers in \cite[Lemma 3.9]{GSz}. The case of nilmanifolds appears in \cite[Ch.\ 12, Proposition 12]{HKbook}. These two cases already yield (via inverse limits) the result for the class of nilspaces concerned in our main results. Below we recall another proof from \cite{CScouplings}, which works at the more general level of cubic couplings. Let us mention also that $\|\cdot\|_{U^d}$ is non-degenerate (and is therefore a norm on $L^\infty(\ns)$) when the step $k$ of $\ns$ is less than $d$. For compact abelian groups this follows from the fact that $\|f\|_{U^d}\geq \|f\|_{U^2} = \|\wh{f}\|_{\ell^4}$, and for nilmanifolds it is given in \cite[Ch.\ 12, Theorem 17]{HKbook}. For general compact nilspaces, the non-degeneracy follows from results in nilspace theory; as it is not needed in this paper, we omit the details.
\begin{proof}[Proof of Lemma \ref{lem:non-deg-app}]
The lemma follows from results in \cite{CScouplings}, namely \cite[Proposition 3.6]{CScouplings}, which shows that the Haar measures $\mu^{\db{n}}$ on $\cu^n(\ns)$ form a cubic coupling, and \cite[Corollary 3.17]{CScouplings}, which yields the seminorm properties for a general cubic coupling.
\end{proof}

\noindent We close this appendix with a proof of Proposition \ref{prop:comdiag}. Recall the following basic useful description of polynomial sequences (see for instance \cite[Lemma 2.8]{CS}).

\begin{lemma}[Taylor expansion]\label{lem:Taylor}
Let $g \in \poly(\mb{Z}, G_\bullet)$, where $G_\bullet$ has degree at most $s$. Then there are unique \emph{Taylor coefficients} $g_i \in G_i$ such that for all $n \in \mb{Z}$ we have $g(n) = g_0 g_1^n g_2^{\binom{n}{2}} \cdots g_s^{\binom{n}{s}}$. Conversely, every such expression defines a map $g \in \poly(\mb{Z},G_\bullet)$. Moreover, if $H\leq G$ and $g$ is $H$-valued then $g_i \in H$ for each $i$.
\end{lemma}

\begin{proof}[Proof of Proposition \ref{prop:comdiag}]
Since $\phi\co\beta$ is a morphism $\mb{Z}\to G/\Gamma$, it suffices to prove the following statement: for every morphism $\phi:\mb{Z}\to G/\Gamma$, there is a morphism $\psi:\mb{Z}\to G$ (whence $\psi\in \poly(\mb{Z},G_\bullet)$) such that $\pi_{\Gamma}\co\psi=\phi$. We prove this by descending induction on $j\in [k+1]$, showing that the statement holds for maps $\phi$ taking values in $(G_j\Gamma)/\Gamma$. For $j=k+1$, since $G_{k+1}=\{\id_G\}$, the map $\phi$ is constant and the statement is trivially verified letting $\psi$ be a constant $\Gamma$-valued map. For $j< k+1$, suppose that the statement holds for $j+1$ and that $\phi$ takes values in $(G_j\Gamma)/\Gamma$. It follows from the filtration property that $G_{j+1}\Gamma$ is a normal subgroup of $G_j\Gamma$ and that the quotient $G_j\Gamma / (G_{j+1}\Gamma )$ is an abelian group. Denoting this abelian group by $A_j$, let $q_j:(G_j\Gamma)/\Gamma\to A_j$ be the quotient map for the action of $G_{j+1}$ on $(G_j\Gamma)/\Gamma$. Note that $q_j$ is a nilspace morphism. More precisely, for every cube $\q \Gamma^{\db{n}}$ on $(G_j\Gamma)/\Gamma$ (where $\q\in G_j^{\db{n}}\cap\cu^n(G_\bullet)$), we have $q_j\co (\q \Gamma^{\db{n}})= (\tilde q_j\co \q) \Gamma^{\db{n}}$ where $\tilde q_j$ is the quotient homomorphism $G_j\to G_j/G_{j+1}$; this implies that every $(j+1)$-face of $q_j\co (\q \Gamma^{\db{n}})$ has value $0$ under the Gray-code map $\sigma_{j+1}$, so $q_j$ is a morphism into $\mc{D}_j(A_j)$.
 It follows that $q_j\co\phi$ is a morphism $\mb{Z}\to \mc{D}_j(A_j)$, and is in particular a polynomial map of degree at most $k$, so by Lemma \ref{lem:Taylor} we have $q_j\co\phi(x)=\sum_{\ell=0}^k a_\ell\binom{x}{\ell}$ for $x\in \mb{Z}$, for some $a_\ell\in A_j$, and binomial coefficients $\binom{x}{\ell}$. Since $q_j$ is surjective, there exist elements $b_0,b_1,\dots,b_k$ in $G_j$ such that $q_j(b_\ell\Gamma) = a_\ell$ for each $\ell$.
Let $\alpha:\mb{Z}\to G$ be the polynomial map $\alpha(x)=\prod_{\ell=0}^k b_\ell^{\binom{x}{\ell}}$, and note that $q_j(\alpha(x)\Gamma)=q_j\co\phi(x)$ for all $x$. It follows that the map $\alpha^{-1}\phi$ is a morphism $\mb{Z}\to (G_{j+1}\Gamma)/\Gamma$, so by induction there is a map $\psi'\in \poly(\mb{Z},G_\bullet)$ such that $\alpha^{-1}(x)\phi(x)=\psi'(x)\Gamma$ for all $x$. Then $\psi(x):=\alpha(x) \psi'(x)$ is a map in $\poly(\mb{Z},G_\bullet)$ with the required property. 
\end{proof}

\section{Miscellaneous measure-theoretic results}\label{app:measure}

\begin{lemma}\label{lem:level}
Let $(\Omega,\mc{A},\lambda)$ be a probability space, let $\mc{B}$ be a sub-$\sigma$-algebra of $\mc{A}$, and suppose that $S\in \mc{A}$ satisfies $\|1_S - \mb{E}(1_S|\mc{B})\|_{L^2}\leq \epsilon$. Then $S'=\{x\in \Omega: \mb{E}(1_S|\mc{B})(x)>\epsilon^{1/2}\}$ satisfies $\lambda(S\Delta S')< 5\epsilon^{1/2}$.
\end{lemma}

\begin{proof}
We first observe that  $\lambda(S'\setminus S)\,\epsilon^{1/2} < \int_\Omega  (1-1_S)\mb{E}(1_S|\mc{B}) \ud\lambda$, which equals $\int_\Omega  \mb{E}(1_S|\mc{B}) -  1_S\mb{E}(1_S|\mc{B})\ud\lambda = \lambda(S) -  \|\mb{E}(1_S|\mc{B})\|_{L^2}^2$. 
Moreover, from the assumption and the triangle inequality we have $\|\mb{E}(1_S|\mc{B})\|_{L^2}\geq \|1_S \|_{L^2}-\epsilon$, whence $\|\mb{E}(1_S|\mc{B})\|_{L^2}^2\geq \|1_S \|_{L^2}^2-2\epsilon = \lambda(S)-2\epsilon$. Therefore $\lambda(S'\setminus S)< 2\epsilon^{1/2}$.

On the other hand, we have $\lambda(S)- 2\epsilon \leq \|\mb{E}(1_S|\mc{B})\|_{L^2}^2 = \langle \mb{E}(1_S|\mc{B}), \mb{E}(1_S|\mc{B})\rangle = \langle 1_S , \mb{E}(1_S|\mc{B})\rangle \leq \int_{S\cap S'} \mb{E}(1_S|\mc{B}) \ud\lambda +  \int_{S\setminus S'} \mb{E}(1_S|\mc{B}) \ud\lambda \leq \lambda(S\cap S') + \epsilon^{1/2}$, so $\lambda(S'\cap S)\geq \lambda(S) -3\epsilon^{1/2}$, whence $\lambda(S\setminus S')\leq 3\epsilon^{1/2}$.

Combining the main two inequalities above, the result follows.
\end{proof}
\noindent We use this lemma to prove the following fact about mod 0 intersections of conditionally independent $\sigma$-algebras.
\begin{lemma}\label{lem:CIalginter}
Let $(\Omega,\mc{A},\lambda)$ be a probability space, let $\mc{B}_0,\mc{B}_1$ be sub-$\sigma$-algebras of $\mc{A}$ such that $\mc{B}_0\upmod_{\lambda}\mc{B}_1$, let $S_i\in \mc{B}_i$, $i=0,1$, and suppose that $\lambda(S_0\Delta S_1)\leq \epsilon$. Then there exists $C\in \mc{B}_0\wedge\mc{B}_1$ such that $\lambda(C\Delta S_i)\leq 10\epsilon^{1/4}$ for $i=0,1$.
\end{lemma}
\begin{proof}
The assumption $\|1_{S_0}- 1_{S_1}\|_{L^2}^2 \leq \epsilon$ implies $\|1_{S_0}- \mb{E}(1_{S_0}|\mc{B}_1)\|_{L^2}\leq  \|1_{S_0}-1_{S_1}\|_{L^2} + \|1_{S_1}-\mb{E}(1_{S_0}|\mc{B}_1)\|_{L^2}\leq \epsilon^{1/2} + \|\mb{E}(1_{S_1}-1_{S_0}|\mc{B}_1)\|_{L^2}\leq 2\epsilon^{1/2}$. The assumption $\mc{B}_0\upmod_{\lambda}\mc{B}_1$ implies that $\mb{E}(1_{S_0}|\mc{B}_1)$ is $\mc{B}_0\wedge\mc{B}_1$-measurable (in particular $\mb{E}(1_{S_0}|\mc{B}_1)=\mb{E}(1_{S_0}|\mc{B}_0\wedge \mc{B}_1)$). By Lemma \ref{lem:level} with $\mc{B}=\mc{B}_0\wedge\mc{B}_1$ and $\mc{A}=\mc{B}_0$, the set $C=\{x\in \Omega: \mb{E}(1_{S_0}|\mc{B}_1)> (2\epsilon^{1/2})^{1/2}\}$ is in $\mc{B}_0\wedge\mc{B}_1$ and satisfies $\lambda(C\Delta S_0)\leq 5 (2\epsilon^{1/2})^{1/2} \leq 10\epsilon^{1/4}$. Similarly, by Lemma \ref{lem:level} with $\mc{A}=\mc{B}_1$ instead of $\mc{A}=\mc{B}_0$, this set $C$ satisfies $\lambda(C\Delta S_1) \leq 10\epsilon^{1/4}$.
\end{proof}
\noindent We can use this lemma in turn to prove the following fact about ultraproducts of conditionally independent $\sigma$-algebras.
\begin{lemma}\label{lem:upcialgs}
Let $(\mf{X},\mc{A},\lambda)$ be the ultraproduct of probability spaces $(X_i,\mc{A}_i,\lambda_i)$. For each $i$ let $\mc{B}_{i,0},\mc{B}_{i,1}$ be sub-$\sigma$-algebras of $\mc{A}_i$ such that $\mc{B}_{i,0}\upmod_{\lambda_i} \mc{B}_{i,1}$. For $j=0,1$ let $\mc{B}_j$ be the Loeb $\sigma$-algebra corresponding to the sequence $(\mc{B}_{i,j})_{i\in \mb{N}}$, and let $\mc{C}$ be the Loeb $\sigma$-algebra corresponding to  $(\mc{B}_{i,0}\wedge_{\lambda_i} \mc{B}_{i,1})_{i\in \mb{N}}$. Then $\mc{B}_0\wedge_\lambda \mc{B}_1 =_\lambda \mc{C}$ and $\mc{B}_0\upmod_\lambda \mc{B}_1$.
\end{lemma}
\begin{proof}
The inclusion $\mc{B}_0\wedge_\lambda \mc{B}_1 \supset_\lambda \mc{C}$ is clear, for if $A\in \mc{C}$ then there are sets $A_i\in \mc{B}_{i,0}\wedge_{\lambda_i} \mc{B}_{i,1}$ such that $A=_\lambda \prod_{i\to \omega} A_i$, so $\prod_{i\to \omega} A_i$ is in $\mc{B}_j$ up to a null set, $j=0,1$, whence $A\in \mc{B}_0\wedge_\lambda \mc{B}_1$. For the opposite inclusion, let $Q$ be in $\mc{B}_0\wedge_\lambda \mc{B}_1$, so for $j=0,1$ there are sets $Q_{i,j}\in \mc{B}_{i,j}$ for each $i\in \mb{N}$ such that $Q=_\lambda \prod_{i\to\omega} Q_{i,j}$. Then $0=\lambda\big((\prod_{i\to \omega} Q_{i,0}) \Delta (\prod_{i\to \omega} Q_{i,1})\big)=\lambda\big( \prod_{i\to \omega} (Q_{i,0} \Delta  Q_{i,1})\big)$, so letting $\epsilon_i=\lambda_i(Q_{i,0} \Delta  Q_{i,1})$, we have $\lim_\omega \epsilon_i =0$. By Lemma \ref{lem:CIalginter}, for each $i$ there is $C_i\in \mc{B}_{i,0}\wedge_{\lambda_i} \mc{B}_{i,1}$ such that $\lambda(C_i\Delta Q_{i,j})\leq 10 \epsilon_i^{1/4}$ for $j=0,1$. Let $R=\prod_{i\to \omega} C_i$. By construction $R\in \mc{C}$, and by the last inequality we have $R=_\lambda Q$, so the required inclusion holds. Finally, the desired conclusion $\mc{B}_0\upmod_\lambda \mc{B}_1$ can be seen to follow from $\mc{B}_{i,0}\upmod_{\lambda_i} \mc{B}_{i,1}$, $i\in\mb{N}$, using the definition of conditional independence  \cite[Definition 2.9]{CScouplings} and basic facts about Loeb probability spaces. More precisely, by \cite[Theorem 2.4 and Remark 2.5]{CScouplings} it suffices to show that every function $f$ in $L^\infty(\mc{B}_1)$ satisfies $\mb{E}(f|\mc{B}_0)=_\lambda \mb{E}(f|\mc{B}_0\wedge_\lambda\mc{B}_1)$. To show this, we use first that $f$ is $\lambda$-almost-surely equal to a measurable function of the form $f'=\lim_\omega f_i'$ (see \cite[Corollary 5.1]{Ross}), and then we prove the equality $\mb{E}(f'|\mc{B}_0)=_\lambda \mb{E}(f'|\mc{B}_0\wedge_\lambda\mc{B}_1)$, by deducing it from the fact that, by the assumption $\mc{B}_{i,0}\upmod_{\lambda_i} \mc{B}_{i,1}$, the analogous equality holds for the $f'_i$. This last deduction is enabled by the fact that $\mb{E}(\cdot|\mc{B}_0)=\lim_\omega \mb{E}(\cdot|\mc{B}_{i,0})$, a fact which is confirmed in a straightforward way by checking that for any function of the form $g=\lim_\omega g_i \in L^1(\mc{A})$ (with each $g_i$ measurable) we have that $\lim_\omega \mb{E}(g_i|\mc{B}_{i,0})$ satisfies the defining property of the conditional expectation $\mb{E}(g|\mc{B}_0)$, i.e.\ that for every $h\in L^1(\mc{B}_0)$ we have $\int_{\mf{X}} h \, g\ud \lambda = \int_{\mf{X}} h \, \lim_\omega \mb{E}(g_i|\mc{B}_{i,0})\ud \lambda$. This last equality is seen using an \emph{$S$-integrable lifting} of $h$ (see \cite[Theorem 6.4]{Ross}), commuting ultralimit and integrals as afforded by \cite[Theorem 6.2, part 4]{Ross}, and basic properties of ultralimits.
\end{proof}
\noindent We also prove the following approximation result for measure-preserving group actions.
\begin{lemma}\label{lem:finapprox}
Let $G$ be an amenable group acting on a Borel probability space $(\Omega,\mc{A},\lambda)$ by measure-preserving transformations, and let $S\in\mc{A}$ be such that for some $\epsilon>0$ we have $\lambda\big(S\Delta (g\cdot S)\big)\leq \epsilon$ for every $g\in G$. Then there exists $S'\in \mc{A}$ such that $g\cdot S'=_\lambda S'$ for all $g\in G$ and $\lambda(S\Delta S')\leq 5\epsilon^{1/4}$.
\end{lemma}
\begin{proof}
We first suppose that $G$ is countable. Let $(F_j)_{j\in  \mb{N}}$ be a F{\o}lner sequence in $G$ and for each $j$ let $h_j=\mb{E}_{g\in F_j} 1_{g\cdot S}$. By the mean ergodic theorem for amenable groups \cite[Theorem 2.1]{Weiss}, letting $\mc{B}$ be the $\sigma$-algebra of $G$-invariant sets in $\mc{A}$, and $f$ be a version of $\mb{E}(1_S|\mc{B})$, we have $\|f-h_j\|_{L^2}\to 0$ as $j\to\infty$. Note that for every $j$ we have $\|1_S-f\|_{L^2}\leq \|1_S-h_j\|_{L^2}+ \|h_j-f\|_{L^2}\leq  \|h_j-f\|_{L^2} + \mb{E}_{g\in F_j} \|1_S-1_{g\cdot S}\|_{L^2}\leq \|h_j-f\|_{L^2} + \epsilon^{1/2}$, so letting $j\to\infty$ yields $\|1_S-f\|_{L^2}\leq \epsilon^{1/2}$. By Lemma \ref{lem:level}, the set $S'=\{x\in\Omega:f(x)>\epsilon^{1/4}\}$ satisfies $\lambda(S \Delta S')\leq 5\epsilon^{1/4}$, and since $f$ is $G$-invariant, we have $g\cdot S'=_\lambda S'$ for every $g\in G$.

We now reduce the general case to the countable case. It suffices to prove that if $G$ is a group acting on a separable metric space $(X,d)$ by isometries, then there is a countable group $G_0\leq G$ such that if $x\in X$ is a fixed point for $G_0$ then it is a fixed point for $G$ (we then apply this with $X$ the measure algebra of $\mc{A}$). Let $(x_i)_i$ be a dense sequence in $X$. For each $i$, the orbit $G\cdot x_i$ is itself separable, so there is a countable set $S_i\subset G$ such that $S_i\cdot x_i$ is dense in this orbit. Let $G_0$ be the subgroup of $G$ generated by $\bigcup_i S_i$. Observe that for every $i\in\mb{N}$, $g\in G$ and $\epsilon>0$, there is $g'\in S_i\subset G_0$ such that $d(g\cdot x_i,g'\cdot x_i)< \epsilon$. Suppose for a contradiction that there is $x\in X$ that is $G_0$-invariant but not $G$-invariant, so $d(g\cdot x,x)=\epsilon>0$. Then by the density of $(x_i)_i$ there is $i$ such that $d(x,x_i)<\epsilon/100$, so $d(g\cdot x_i,x_i)\geq d(g\cdot x_i,x)-d(x, x_i) \geq d(g\cdot x, x)-d(g\cdot x_i,g\cdot x)-d(x, x_i)$, which by the isometry property equals $d(g\cdot x, x)-2 d(x, x_i)\geq 98\epsilon/100$. Hence $d(g\cdot x_i,x_i)\geq 98\epsilon/100$. 
By the earlier observation, there is $g'\in G_0$ such that $d(g\cdot x_i,g'\cdot x_i)< \epsilon/100$, so $d(g'\cdot x_i,x_i)\geq d(g\cdot x_i,x_i) - d(g\cdot x_i,g'\cdot x_i)\geq 97\epsilon/100$. Combining this last inequality with $d(x,x_i)<\epsilon/100$ and the triangle inequality and isometry property, we deduce that $d(g'\cdot x,x)\geq d(g'\cdot x_i,x_i)-2 d(x,x_i)\geq 95\epsilon/100$, which contradicts that $x$ is $G_0$-invariant.
\end{proof}

\begin{lemma}\label{lem:Loebcomm}
Let $\nss$ be a compact Polish space, let $d$ be a metric compatible with the weak topology on $\mc{P}(\nss)$, and let $(\ns_i,\lambda_i)_{i\in \mb{N}}$ be a sequence of Borel probability spaces. For each $i\in\mb{N}$ let $f_i:\ns_i\to\nss$ be a  Borel function, and let $\omega$ be a non-principal ultrafilter on $\mb{N}$. Then, letting $f=\lim_\omega f_i$, we have $\lim_\omega d(\lambda_i\co f_i^{-1}, \lambda\co f^{-1})=0$. 
\end{lemma}

\begin{proof}
As shown in \cite[Theorem (17.19)]{Ke}, one can always metrize this space of probability measures with a metric of the form $d'(\mu,\nu)=\sum_{r\in\mb{N}} \tfrac{1}{2^r} |\int h_r \ud\mu - \int h_r \ud\nu|$, for a sequence of continuous functions $h_r:\nss\to \mb{C}$ with $\|h_r\|_{\infty}\leq 1$, $r\in \mb{N}$. Since $d$ and $d'$ metrize the same topology, it suffices to prove that $\lim_\omega d'(\lambda_i\co f_i^{-1}, \lambda\co f^{-1})=0$.

Suppose for a contradiction that for some $b\in (0,1)$ and some set $S\in\omega$, for every $i\in S$ we have $d'(\lambda_i\co f_i^{-1}, \lambda\co f^{-1})>b$. Then, for each $i\in S$, a short argument by contradiction shows that there exists $r=r(i)\in [1, 2 \lceil \log_2(2/b)\rceil\,]$ such that $|\int_{\ns_i} h_r\co f_i \ud \lambda_i - \int_{\ns} h_r\co f \ud\lambda|\geq b/2$. Using the ultrafilter properties, we then deduce that for some fixed integer $r$ there is a set $S'\subset S$ with $S'\in\omega$ such that for all $i\in S'$ we have $|\int_{\ns_i} h_r\co f_i \ud \lambda_i - \int_{\ns} h_r\co f \ud\lambda|\geq b/2$. Now we have two exhaustive possibilities. The first one is that some $S''\subset S'$ with $S''\in \omega$ satisfies $\int_{\ns_i} h_r\co f_i \ud\lambda_i \geq \int_{\ns} h_r\co f \ud \lambda +  b/2$ for all $i\in S''$; but then, commuting ultralimit and integrals (as in the proof of Lemma \ref{lem:upcialgs}), we obtain $\int_{\ns} h_r\co f \ud\lambda= \lim_\omega \int_{\ns_i} h_r\co f_i \ud \lambda_i \geq \int_{\ns} h_r\co f \ud\lambda + b/2 > \int_{\ns} h_r\co f \ud\lambda$, a contradiction. The other option is that some $S''\subset S'$ with $S''\in \omega$ satisfies $\int_{\ns} h_r\co f \ud\lambda \geq$  $\int_{\ns_i} h_r\co f_i \ud \lambda_i +  b/2$ for all $i\in S''$; then we deduce similarly that $\int_{\ns} h_r\co f \ud\lambda$ $= \lim_\omega \int_{\ns_i} h_r\co f_i \ud \lambda_i \leq \int_{\ns} h_r\co f \ud\lambda - b/2 < \int_{\ns} h_r\co f \ud\lambda$, obtaining again a contradiction.
\end{proof}

\noindent We finish with a lemma concerning the interaction of the Loeb-measure construction with products, when the underlying measures are couplings on Borel probability spaces.

\begin{lemma}\label{lem:applemf}
Let $(\ns_i)_{i\in \mb{N}}$, $(\nss_i)_{i\in \mb{N}}$ be sequences of Polish spaces, and for each $i\in\mb{N}$ let $\mu_i$ be a Borel probability measure on $\mc{B}(\ns_i)$ and $\nu_i$ be a Borel probability measure on $\mc{B}(\ns_i)\otimes \mc{B}(\nss_i)$. Let $(\mf{X},\mc{L}_{\mf{X}},\mu)$, $(\mf{X}\times \mf{Y},\mc{L}_{\mf{X}\times \mf{Y}}, \nu)$ be the corresponding Loeb probability spaces. Suppose that the projection $\pi_i:\ns_i\times \nss_i\to \ns_i$, $(x,y)\mapsto x$ is measure preserving for every $i\in\mb{N}$. Then the projection $\pi:\mf{X}\times \mf{Y}\to \mf{X}$, $(x,y)\mapsto x$ is measurable with respect to $\mc{L}_{\mf{X}}$, $\mc{L}_{\mf{X}\times \mf{Y}}$, and is measure-preserving with respect to $\mu,\nu$.
\end{lemma} 

\begin{proof}
The preimage under $\pi$ of any internal measurable set in $\mf{X}$ is an internal measurable set in $\mf{X}\times\mf{Y}$, and it is also clear that if $A$ is an internal measurable subset of $\mf{X}$ then $\nu\co \pi^{-1}(A)=\mu(A)$. (These claims follow from the fact the projections $\pi_i$ are measure-preserving maps and that taking ultraproducts commutes with taking preimages under the projections.) Now $\mc{L}_{\mf{X}}$ consists precisely of sets $S$ such that for every $\epsilon>0$ there exist internal measurable sets $A_i,A_o\subset \mf{X}$ with $A_i\subset S\subset A_o$ and $\mu(A_o\setminus A_i)<\epsilon$ \cite[\S 2.1]{Ross}. This combined with the properties already established for $\pi$ for internal sets implies that $\pi^{-1}(\mc{L}_{\mf{X}})\subset \mc{L}_{\mf{X}\times \mf{Y}}$ and $\mu\co \pi^{-1}=\nu$, as required.
\end{proof}


\begin{thebibliography}{1}

\bibitem{BTZ} V. Bergelson, T. Tao, T. Ziegler, \emph{An inverse theorem for the uniformity seminorms associated with the action $\mb{F}_p^\infty$}, Geom. Funct. Anal. \textbf{19} (2010), no. 6, 1539--1596.

\bibitem{CamSzeg} O. A. Camarena, B. Szegedy, \emph{Nilspaces, nilmanifolds and their morphisms}, preprint. \href{http://arxiv.org/abs/1009.3825}{arXiv:1009.3825}

\bibitem{Cand:Notes1} P. Candela, \emph{Notes on nilspaces: algebraic aspects}, Discrete Analysis, 2017, Paper No. 15, 59 pp.

\bibitem{Cand:Notes2} P. Candela, \emph{Notes on compact nilspaces}, Discrete Analysis, 2017, Paper No. 16, 57pp.

\bibitem{CGS} P. Candela, D. Gonz\'alez-S\'anchez, B. Szegedy, \emph{On nilspace systems and their morphisms}, Ergodic Theory Dynam. Systems \textbf{40} (2020), no. 11, 3015--3029.

\bibitem{CS} P. Candela, O. Sisask, \emph{Convergence results for systems of linear forms on cyclic groups, and periodic nilsequences}, SIAM J. Discrete Math. \textbf{28} (2014) (2), 786--810.

\bibitem{CScouplings} P. Candela, B. Szegedy, \emph{Nilspace factors for general uniformity seminorms, cubic exchangeability and limits}, to appear in Mem. Amer. Math. Soc. \href{http://arxiv.org/abs/1803.08758}{arXiv:1803.08758}

\bibitem{Cutland} N. J. Cutland, \emph{Nonstandard measure theory and its applications}, 
Bull. London Math. Soc. \textbf{15} (1983), no. 6, 529--589. 

\bibitem{E&S} G. Elek, B. Szegedy, \emph{A measure-theoretic approach to the theory of dense hypergraphs}, Adv. Math. \textbf{231} (2012), no. 3-4, 1731--1772.

\bibitem{Fremlin2} D. H. Fremlin, \emph{Measure theory}, Vol. 2, Broad foundations. Corrected second printing of the 2001 original. Torres Fremlin, Colchester, 2003.

\bibitem{Fremlin3} D. H. Fremlin, \emph{Measure theory}, Vol. 3, Measure algebras. Corrected second printing of the 2002 original. Torres Fremlin, Colchester, 2004.

\bibitem{Georg} G. Georganopoulos, \emph{Sur l'approximation des fonctions continues par des fonctions lipschitziennes}, C. R. Acad. Sci. Paris S\'er. A-B 264 1967 A319--A321.

\bibitem{GHFA} W. T. Gowers, \emph{Generalizations of Fourier analysis, and how to apply them}, Bull. Amer. Math. Soc. \textbf{54} (2017), no. 1, 1--44.

\bibitem{GSz} W. T. Gowers, \emph{A new proof of Szemer\'edi's theorem}, GAFA \textbf{11} (2001), 465--588.

\bibitem{GM} W. T. Gowers, L. Mili\'cevi\'c, \emph{An inverse theorem for Freiman multi-homomorphisms}, preprint. \href{http://arxiv.org/abs/2002.11667}{arXiv:2002.11667}

\bibitem{GTarith} B. Green, T. Tao, \emph{An arithmetic regularity lemma, an associated counting lemma, and applications}, in An Irregular Mind, Bolyai Soc. Math. Stud. \textbf{21}, J\'anos Bolyai Mathematical Society, Budapest, 2010, 261--334.

\bibitem{GTprimes} B. Green, T. Tao, \emph{The primes contain arbitrarily long arithmetic progressions}, Ann. of Math. (2) \textbf{167} (2008), no. 2, 481--547.

\bibitem{GTOrb} B. Green, T. Tao, \emph{The quantitative behaviour of polynomial orbits on nilmanifolds}, Ann. of Math. (2) \textbf{175} (2012), no. 2, 465--540.

\bibitem{GTZ} B. Green, T. Tao, T. Ziegler, \emph{An inverse theorem for the Gowers $U^{s+1}[N]$-norm}, Ann. of Math. \textbf{176} (2012), no. 2, 1231--1372.

\bibitem{GMV1} Y. Gutman, F. Manners, P. P. Varj\'u, \emph{The structure theory of nilspaces I},  J. Anal. Math. \textbf{140} (2020), no. 1, 299--369.

\bibitem{GMV2} Y. Gutman, F. Manners, P. P. Varj\'u, \emph{The structure theory of nilspaces II: Representation as nilmanifolds},  Trans. Amer. Math. Soc. \textbf{371} (2019), no. 7, 4951--4992. 

\bibitem{GMV3} Y. Gutman, F. Manners, P. P. Varj\'u, \emph{The structure theory of nilspaces III: Inverse limit representations and topological dynamics},  Adv. Math. \textbf{365} (2020), 107059, 53 pp.

\bibitem{GlaGutYe} E. Glasner, Y. Gutman, X. Ye, \emph{Higher order regionally proximal equivalence relations for general minimal group actions}, Advances in Mathematics \textbf{333} (2018), 1004--1041.

\bibitem{GutLian}  Y. Gutman, Z. Lian, \emph{Strictly ergodic distal models and a new approach to the Host--Kra factors}, preprint. \href{http://arxiv.org/abs/1909.11349}{arXiv:1909.11349} 

\bibitem{Helga} S. Helgason, \emph{Differential geometry, Lie groups, and symmetric spaces}, Pure and Applied Mathematics, 80. Academic Press, Inc., New York-London, 1978.

\bibitem{HK} B. Host, B. Kra, \emph{Nonconventional ergodic averages and nilmanifolds}, Ann. of Math. (2) \textbf{161} (2005), no. 1, 397--488.

\bibitem{HKbook} B. Host, B. Kra, \emph{Nilpotent structures in ergodic theory}, Mathematical Surveys and Monographs, Volume 236, 2018.

\bibitem{HKparas} B. Host, B. Kra, \emph{Parallelepipeds, nilpotent groups and Gowers norms}, Bull. Soc. Math. France \textbf{136} (2008), no. 3, 405--437.

\bibitem{Ke} A. S. Kechris, \emph{Classical descriptive set theory}, Graduate Texts in Mathematics, \textbf{156}. Springer-Verlag, New York, 1995.

\bibitem{Leib} A. Leibman, \emph{Polynomial mappings of groups}, Israel J. Math. \textbf{129} (2002), 29--60.

\bibitem{Ma} G. W. Mackey, \emph{Point realizations of transformation groups}, Illinois J. Math. \textbf{6} (1962), 327--335. 

\bibitem{Man1} F. Manners, \emph{Periodic nilsequences and inverse theorems on cyclic groups}, preprint. \href{http://arxiv.org/abs/1404.7742}{arXiv:1404.7742} 

\bibitem{Man2} F. Manners, \emph{Quantitative bounds in the inverse theorem for the Gowers $U^{s+1}$-norms over cyclic groups}, preprint. \href{https://arxiv.org/pdf/1811.00718.pdf}{arXiv:1811.00718}. 

\bibitem{Palais} R. S. Palais, \emph{The classification of G-spaces}, Mem. Amer. Math. Soc. No. \textbf{36}, 1960.

\bibitem{Ross} D. A. Ross, \emph{Loeb measure and probability}. Nonstandard analysis (Edinburgh, 1996), 91--120,
NATO Adv. Sci. Inst. Ser. C: Math. Phys. Sci., 493, Kluwer Acad. Publ., Dordrecht, 1997. 

\bibitem{Szegedy1} B. Szegedy, \emph{Higher order Fourier analysis as an algebraic theory I}, preprint 2009. \href{http://arxiv.org/abs/0903.0897}{arXiv:0903.0897}

\bibitem{Szegedy:reg} B. Szegedy, \emph{Gowers norms, regularization and limits of functions on abelian groups}, preprint 2010. \href{http://arxiv.org/abs/1010.6211}{arXiv:1010.6211}

\bibitem{Szegedy:HFA} B. Szegedy, \emph{On higher order Fourier analysis}, preprint 2012. \href{http://arxiv.org/abs/1203.2260}{arXiv:1203.2260}

\bibitem{TaoH} T. Tao, \emph{Hilbert's fifth problem and related topics}, Graduate Studies in Mathematics, 153. American Mathematical Society, Providence, RI, 2014. xiv+338 pp.

\bibitem{T&Z} T. Tao, T. Ziegler, \emph{The inverse conjecture for the Gowers norm over finite fields via the correspondence principle}, Anal. PDE \textbf{3} (2010), 1--20.

\bibitem{T&Z2} T. Tao, T. Ziegler, \emph{The inverse conjecture for the Gowers norm over finite fields in low characteristic}, Ann. Comb. 16 (2012), no. 1, 121--188.

\bibitem{Warner} E. Warner, \emph{Ultraproducts and the foundations of higher order Fourier analysis}, thesis available online at \href{https://www.math.columbia.edu/~warner/notes/UndergradThesis.pdf}{www.math.columbia.edu/\~{}warner/notes/UndergradThesis.pdf}. Retrieved 31-01-2022.

\bibitem{Weiss} B. Weiss, \emph{Actions of amenable groups}, Topics in dynamics and ergodic theory, 226--262, London Math. Soc. Lecture Note Ser., \textbf{310}, Cambridge Univ. Press, Cambridge, 2003. 
\end{thebibliography}
\end{document}